\documentclass[a4paper,11pt]{article}
\usepackage{amsthm,amsmath,amssymb,amsfonts,bbm,color,geometry,epsfig,listings,hyperref,enumerate,comment,nicefrac,verbatim,multirow,titlesec}
\newtheorem{lemma}{Lemma}[section]
\newtheorem{remark}[lemma]{Remark}
\newtheorem{proposition}[lemma]{Proposition}
\newtheorem{theorem}[lemma]{Theorem}

\newtheorem{corollary}[lemma]{Corollary}

\newtheorem{setting}[lemma]{Setting}

\newcommand{\smallsum}{\textstyle\sum}
\newcommand{\1}{\ensuremath{\mathbbm{1}}}
\providecommand{\N}{{\ensuremath{\mathbbm{N}}}}
\providecommand{\Z}{{\ensuremath{\mathbbm{Z}}}}
\providecommand{\R}{{\ensuremath{\mathbbm{R}}}}

\providecommand{\E}{{\ensuremath{\mathbb{E}}}}
\renewcommand{\P}{{\ensuremath{\mathbb{P}}}}
\newcommand{\eps}{{\ensuremath{\varepsilon}}}

\newcommand{\funcF}{F}
\newcommand{\LipConst}{L}
\newcommand{\Var}{{\ensuremath{\operatorname{Var}}}}

\newcommand{\uniform}{\ensuremath{\mathcal{R}}}
\newcommand{\unif}{\ensuremath{\mathfrak{r}}}
\newcommand{\ES}{\ensuremath{S}}

\newcommand{\RN}{\operatorname{RV}}

\newcommand{\blue}[1]{}

\newcommand{\rhoh}{\varrho}
\newcommand{\uvec}{\mathbf{e}}

\makeatletter
\newcommand{\vast}{\bBigg@{4}}
\newcommand{\Vast}{\bBigg@{5}}
\makeatother

\title{Overcoming the curse of dimensionality 
in the 
numerical\\ approximation of
parabolic partial
differential\\
equations
with gradient-dependent nonlinearities
}

\author{Martin Hutzenthaler$^{1}$, Arnulf Jentzen$^{2,3}$, and Thomas Kruse$^{3}$
\bigskip
\\
\small{$^1$ Faculty of Mathematics, University of Duisburg-Essen,} 
\\
\small{45117 Essen, Germany, e-mail: martin.hutzenthaler@uni-due.de}
\smallskip
\\
\small{$^2$ SAM, Department of Mathematics, ETH Z\"urich,}
\\
\small{8092 Z\"urich, Switzerland, e-mail: arnulf.jentzen@sam.math.ethz.ch}
\smallskip
\\
\small{$^3$ Faculty of Mathematics and Computer Science, University of M\"unster,}
\\
\small{48149 M\"unster, Germany, e-mail: ajentzen@uni-muenster.de}
\smallskip
\\
\small{$^4$ Institute of Mathematics, University of Gie{\ss}en,}
\\
\small{35392 Gie{\ss}en, Germany, e-mail: thomas.kruse@math.uni-giessen.de}
}

\begin{document}

\maketitle
\makeatletter
\let\@makefnmark\relax
\let\@thefnmark\relax
\@footnotetext{\emph{AMS 2010 subject classification:} 65M75}
\@footnotetext{\emph{Keywords and phrases:}
curse of dimensionality, partial differential equation, PDE, backward stochastic differential equation, BSDE, multilevel Picard, multilevel Monte Carlo, gradient-dependent nonlinearity
  }
\makeatother
\abstract{
Partial differential equations (PDEs) are a fundamental tool in the modeling of many real world phenomena. In a number of such real world phenomena the PDEs under consideration
contain gradient-dependent nonlinearities and are high-dimensional. 
Such high-dimensional nonlinear PDEs can in nearly all cases not be solved explicitly and it is one of the most challenging tasks in applied mathematics to solve high-dimensional nonlinear PDEs approximately.
It is especially very challenging to design approximation algorithms for nonlinear PDEs for which one can rigorously prove that they do overcome the so-called curse of dimensionality in the sense that the number of computational operations of the approximation algorithm needed to achieve an approximation precision of size $\eps > 0$ grows at most polynomially in both the PDE dimension $d \in \N$ and the reciprocal of the prescribed approximation accuracy $\eps$.
In particular, to the best of our knowledge there exists no approximation algorithm in the scientific literature which has been proven to overcome 
the curse of dimensionality in the case of a class of nonlinear PDEs with general time horizons and 
gradient-dependent nonlinearities. 
It is the key contribution of this article to overcome this difficulty. More specifically, it is the key contribution of this article (i) to propose a new full-history recursive multilevel Picard approximation algorithm for high-dimensional nonlinear heat equations with general time horizons and gradient-dependent nonlinearities and (ii) to rigorously prove that this full-history recursive multilevel Picard approximation algorithm 
does indeed overcome the curse of dimensionality in the case of such nonlinear heat equations with gradient-dependent nonlinearities.

}
\tableofcontents

\section{Introduction}

Partial differential equations (PDEs) play a prominent role in the modeling of many real world phenomena. For instance, PDEs appear in financial engineering in models for the pricing of financial derivatives, PDEs emerge in biology in models that aim to better understand biodiversity in ecosystems, PDEs such as the Schr\"odinger equation appear in quantum physics to describe the wave function of a quantum-mechanical system, PDEs are used in operations research to characterize the value function of control problems, PDEs provide solutions for backward stochastic differential equations (BSDEs) which itself appear in several models from applications, and stochastic PDEs such as the Zakai equation or the Kushner equation appear in nonlinear filtering problems to describe the density of the state of a physical system with only partial information available.
The PDEs in the above named models 
contain often nonlinearities and are typically high-dimensional, where, e.g., in the models from financial engineering the dimension of the PDE usually corresponds to the number of financial assets in the associated hedging or trading portfolio, where, e.g., in models that aim to better understand biodiversity the dimension of the PDE corresponds to the number of traits of the considered species in the considered ecosystem, where, e.g., in quantum physics the dimension of the PDE is, loosely speaking, three times the number of electrons in the considered physical system,
where, e.g., in optimal control problems the dimension of the PDE is determined by the dimension of the state space of the control problem, and
where, e.g., in nonlinear filtering problems the dimension of the PDE corresponds to the degrees of freedom in the considered physical system.

Such high-dimensional nonlinear PDEs can in nearly all cases not be solved explicitly and it is one of the most challenging tasks in applied mathematics to solve high-dimensional nonlinear PDEs approximately. In particular, it is very challenging to design approximation methods for nonlinear PDEs for which one can rigorously prove that they do overcome the so-called curse of dimensionality in the sense that the number of computational operations of the approximation method needed to achieve an approximation precision of size $\eps > 0$ grows at most polynomially in both the PDE dimension $d \in \N$ and the reciprocal of the prescribed approximation accuracy $\eps$. Recently, several new stochastic approximation methods for certain classes of high-dimensional nonlinear PDEs have been proposed and studied in the scientific literature. In particular, we refer, e.g., to \cite{BouchardTouzi2004, GobetLemorWarin2005, Lemor2006, gobet2016approximation, BriandLabart2014,geiss2016simulation} for BSDE based approximation methods for PDEs in which nested conditional expectations are discretized through suitable regression methods, 
we refer, e.g., to \cite{HenryLabordere2012,HenryLabordereTanTouzi2014,
HenryLabordereOudjaneTanTouziWarin2016, BouchardTanWarinZou2017} for branching diffusion approximation methods for PDEs, we refer, e.g., to \cite{uchiyama1993solving,MeadeFernandez1994,Lagaris1998ArtificialNN,LiLuo2003,
EHanJentzen2017CMStat,BeckEJentzen2017,EYu2018,
FujiiTakahashiTakahashi2017,HanJentzenE2017,HenryLabordere2017,
Farahmand2017DeepRL,Raissi2018,
SirignanoSpiliopoulos2017,BeckerCheriditoJentzen2018arXiv,Becketal2018,
Magill2018NeuralNT,LongLuMaDong2018,HanLong2018,Berg2018AUD,
PhamWarin2019,LyeMishraRay2019, GoudenegeMolent2019,BeckBeckerCheridito2019,
JacquierOumgari2019,HurePhamWarin2019,ChanMikaelWarin2019,
BeckerCheriditoJentzen2019,Dockhorn2019,chen2019deep,grohs2019deep}
 for deep learning based approximation methods for PDEs, and we refer to \cite{EHutzenthalerJentzenKruse2017,HutzenthalerKruse2017,hutzenthaler2019overcoming,beck2019overcoming, giles20019generalised} for numerical simulations, approximation results, and extensions of the in \cite{EHutzenthalerJentzenKruse2016,HJKNW2018} recently introduced full-history recursive multilevel Picard approximation methods for PDEs (in the following we abbreviate \textit{full-history recursive multilevel Picard} as \textit{MLP}). Branching diffusion approximation methods are also in the case of certain nonlinear PDEs as efficient as plain vanilla Monte Carlo approximations in the case of linear PDEs but the error analysis only applies in the case where the time horizon $T \in (0,\infty)$ and the initial condition, respectively, are sufficiently small and branching diffusion approximation methods are actually not working anymore in the case where the time horizon $T \in (0,\infty)$ exceeds a certain threshold (cf., e.g., \cite[Theorem~3.12]{HenryLabordereOudjaneTanTouziWarin2016}). For MLP approximation methods it has been recently shown in \cite{HJKNW2018,hutzenthaler2019overcoming,beck2019overcoming} that such algorithms do indeed overcome the curse of dimensionality for certain classes of gradient-independent PDEs. Numerical simulations for deep learning based approximation methods for nonlinear PDEs in high-dimensions are very encouraging (see, e.g., the above named references \cite{uchiyama1993solving,MeadeFernandez1994,Lagaris1998ArtificialNN,LiLuo2003,
EHanJentzen2017CMStat,BeckEJentzen2017,EYu2018,
FujiiTakahashiTakahashi2017,HanJentzenE2017,HenryLabordere2017,
Farahmand2017DeepRL,Raissi2018,
SirignanoSpiliopoulos2017,BeckerCheriditoJentzen2018arXiv,Becketal2018,
Magill2018NeuralNT,LongLuMaDong2018,HanLong2018,Berg2018AUD,
PhamWarin2019,LyeMishraRay2019, GoudenegeMolent2019,BeckBeckerCheridito2019,
JacquierOumgari2019,HurePhamWarin2019,ChanMikaelWarin2019,
BeckerCheriditoJentzen2019,Dockhorn2019,chen2019deep,grohs2019deep}) but so far there is only a partial error analysis available for such algorithms (which, in turn, is strongly based on the above mentioned error analysis for the MLP approximation method;
cf.\ \cite{hutzenthaler2019proof} and, e.g., \cite{HanLong2018,SirignanoSpiliopoulos2017,BernerGrohsJentzen2018, ElbraechterSchwab2018,GrohsWurstemberger2018,
JentzenSalimovaWelti2018,KutyniokPeterseb2019,ReisingerZhang2019,
GrohsHornungJentzen2019}). To sum up, to the best of our knowledge until today the MLP approximation method (see \cite{HJKNW2018}) is the only approximation method in the scientific literature for which it has been shown that it does overcome the curse of dimensionality in the numerical approximation of semilinear PDEs with general time horizons.

The above mentioned articles \cite{HJKNW2018,hutzenthaler2019overcoming,beck2019overcoming,giles20019generalised} prove, however, only in the case of gradient-independent nonlinearities that MLP approximation methods overcome the curse of dimensionality and it remains an open problem to overcome the curse of dimensionality in the case of PDEs with gradient-dependent nonlinearities. This is precisely the subject of this article. 
More specifically, in this article we propose a new MLP approximation method for nonlinear heat equations with gradient-dependent nonlinearities and the main result of this article, Theorem~\ref{cor:comp_and_error2} in Section~\ref{sec:rate} below, proves that the number of realizations of scalar random variables required by this MLP approximation method to achieve a precision of size $\eps > 0$ grows at most polynomially in both the PDE dimension $d \in \N$ and the reciprocal of the prescribed approximation accuracy $\eps$. To illustrate the findings of the main result of this article in more detail, we now present in the following theorem a special case of Theorem~\ref{cor:comp_and_error2}.
   \begin{theorem}
\label{thm:introduction}
Let $T,\delta,\lambda \in (0,\infty)$, 
let
$u_d = ( u_d(t,x) )_{ (t,x) \in [0,T] \times \R^d }\in C^{1,2}([0,T]\times\R^d,\R)$, $d\in \N$,
be at most polynomially growing functions,
 let $f_d \in C( \R \times \R^d,\R)$, $d\in \N$,
 let $g_d \in C( \R^d,\R)$, $d\in \N$,
 let $L_{d,i}\in \R$, $d,i \in \N$,
assume for all $d\in \N$,
$t\in[0,T)$, $x=(x_1,x_2, \ldots,x_d)$, $\mathfrak x=(\mathfrak x_1,\mathfrak x_2, \ldots,\mathfrak x_d)$, $z=(z_1,z_2,\ldots,z_d)$, $\mathfrak{z}=(\mathfrak z_1, \mathfrak z_2, \ldots, \mathfrak z_d)\in\R^d$, $y,\mathfrak{y} \in \R$ that
\begin{equation}\label{eq:fLipschitz_intro}
\max\{|f_d(y,z)-f_d(\mathfrak y,\mathfrak{z})|,|g_d(x)-g_d(\mathfrak  x)|\}\le
\textstyle{
\sum_{j=1}^d}L_{d,j}
\big(d^\lambda |x_j-\mathfrak{x}_j|+ |y-\mathfrak y|+|z_j-\mathfrak{z}_j|
\big),
\end{equation}
\begin{equation}\label{eq:PDE_intro}
\big( \tfrac{ \partial }{ \partial t } u_d \big)( t, x )=  ( \Delta_x u_d )( t, x ) + f_d\big( u(t,x), ( \nabla_x u_d )(t, x) \big),  \qquad u_d(0,x) = g_d(x),
\end{equation}
and
$
d^{-\lambda}(|g_d(0)|+|f_d(0,0)|)+\sum_{i=1}^d L_{d,i}\le \lambda,
$
let
$
  ( 
    \Omega, \mathcal{F}, \P 
  )
$
be a probability space,
let
$
  \Theta = \cup_{ n \in \N } \Z^n
$,
let
$
  Z^{d, \theta } \colon \Omega \to \R^d 
$, $d\in \N$,
$ \theta \in \Theta $,
be i.i.d.\ standard normal random variables,
let $\unif^\theta\colon \Omega\to(0,1)$, $\theta\in \Theta$, be 
i.i.d.\ random variables,
assume 
for all $b\in (0,1)$
that
$\P(\unif^0\le b)=\sqrt{b}$,
assume that
$(Z^{d,\theta})_{(d, \theta) \in \N \times \Theta}$ and
$(\unif^\theta)_{ \theta \in \Theta}$ are independent,
let
$ 
  {\bf U}_{ n,M}^{d,\theta }
=
  (
  {\bf U}_{ n,M}^{d,\theta, 0},{\bf U}_{ n,M}^{d,\theta, 1},\ldots,{\bf U}_{ n,M}^{d,\theta, d}
  )
  \colon(0,T]\times\R^d\times\Omega\to\R^{1+d}
$,
$n,M,d\in\Z$, $\theta\in\Theta$,
satisfy
for all 
$
  n,M,d \in \N
$,
$ \theta \in \Theta $,
$ t\in (0,T]$,
$x \in \R^d$
that $
{\bf U}_{-1,M}^{d,\theta}(t,x)={\bf U}_{0,M}^{d,\theta}(t,x)=0$ and
\begin{equation}  \begin{split}\label{eq:def:U_intro}
 {\bf U}_{n,M}^{d,\theta}(t,x)
  &=
  \left(
    g_d(x)
    , 0
  \right)
  +
  \textstyle \sum\limits_{i=1}^{M^n} \displaystyle \tfrac{1}{M^n} \big(g_d(x+[2t]^{1/2}Z^{(\theta,0,-i)})-g_d(x)\big)
  \big(
  1 , [2t]^{-1/2}
  Z^{d,(\theta, 0, -i)}
  \big)
  \\
  &\quad +\textstyle\sum\limits_{l=0}^{n-1}\sum\limits_{i=1}^{M^{n-l}}\displaystyle \tfrac{2t [\unif^{(\theta, l,i)}]^{1/2}}{M^{n-l}}  \big[ 
 f_d\big({\bf U}_{l,M}^{d,(\theta,l,i)}(t(1-\unif^{(\theta, l,i)}),x+[2t\unif^{(\theta, l,i)}]^{1/2}Z^{d,(\theta,l,i)})\big)
\\
 & \quad  
 -\1_{\N}(l)f_d\big( {\bf U}_{l-1,M}^{d,(\theta,-l,i)}(t(1-\unif^{(\theta, l,i)}),x+[2t\unif^{(\theta, l,i)}]^{1/2}Z^{d,(\theta,l,i)})\big)
  \big]\big(
  1 ,
   [2t\unif^{(\theta, l,i)}]^{-1/2}
  Z^{d,(\theta,l,i)}
  \big),
\end{split}     \end{equation}
and for every $d,M,n \in \N$ let $\RN_{d,n,M}\in \N$ be the number of 
realizations of scalar random variables which are used to compute one realization
of $ {\bf U}_{n,M}^{d,0}(T,0)\colon \Omega \to \R$ (cf.\ \eqref{c18} for a
precise definition).
Then there exist $c\in \R$ and $N=(N_{d,\eps})_{(d, \eps) \in \N \times (0,1]}\colon \N \times (0,1] \to \N$ such that
for all $d\in \N$, $\eps \in (0,1]$ it holds that 
$ \sum_{n=1}^{N_{d,\eps}}\RN_{d,n,\lfloor n^{1/4} \rfloor} \le c d^c \varepsilon^{-(2+\delta)}$ and
 \begin{equation}\label{eq:intro_main}
\sup_{ n \in \N \cap [N_{d,\eps},\infty) } \Big[
 \E\big[|{\bf U}_{{n},\lfloor n^{1/4} \rfloor}^{d,0,0}(T,0)-u_d(T,0)|^2\big]+
\max_{i\in\{1,2,\ldots,d\}}
    \E\big[ |{\bf U}_{{n},\lfloor n^{1/4} \rfloor}^{d,0,i}(T,0)-( \tfrac{ \partial }{ \partial x_i } u_d )(T,0)|^2\big]
 \Big]^{\nicefrac 12}
 \le \eps.
 \end{equation}
\end{theorem}
Theorem~\ref{thm:introduction} is an immediate consequence of Corollary~\ref{cor:comp_and_error3}
in Section~\ref{sec:rate} below. 
Corollary~\ref{cor:comp_and_error3}, in turn, follows from Theorem~\ref{cor:comp_and_error2} in Section~\ref{sec:rate}, which is the main result of this article. In the following we add a few comments regarding some of the mathematical objects appearing in Theorem~\ref{thm:introduction} above. The real number $T \in (0,\infty)$ in Theorem~\ref{thm:introduction} above describes the time horizon of the PDE under consideration (see \eqref{eq:PDE_intro} in Theorem~\ref{thm:introduction} above). Theorem~\ref{thm:introduction} proves under suitable Lipschitz assumptions that the MLP approximation method in \eqref{eq:def:U_intro} above overcomes the curse of dimensionality in the numerical approximation of the gradient-dependent semilinear PDEs in \eqref{eq:PDE_intro} above. Theorem~\ref{thm:introduction} even proves that the computational effort of the MLP approximation method in \eqref{eq:def:U_intro} required to obtain a precision of size $\eps\in (0,1]$ is bounded by $c d^c \eps^{ ( 2 + \delta ) }$ where $c\in \R$ is a constant which is completely independent of the PDE dimension $d\in \N$ and where $\delta \in (0,\infty)$ is an arbitrarily small positive real number which describes the convergence order which we lose when compared to standard Monte Carlo approximations of linear heat equations. The real number $\lambda \in (0,\infty)$ in Theorem~\ref{thm:introduction} above is an arbitrary large constant which we employ to formulate the Lipschitz and growth assumptions in Theorem~\ref{thm:introduction} (see \eqref{eq:fLipschitz_intro} and below \eqref{eq:PDE_intro} in Theorem~\ref{thm:introduction} above). The functions $u_d \colon [0,T]\times \R^d \to \R$, $d \in \N$, in Theorem~\ref{thm:introduction} above are the solutions of the PDEs under consideration; see \eqref{eq:PDE_intro} in Theorem~\ref{thm:introduction} above. Note that for every $d \in \N$ we have that \eqref{eq:PDE_intro} is a PDE where the time variable $t \in [0,T]$ takes values in the interval $[0,T]$ and where the space variable $x \in \R^d$ takes values in the $d$-dimensional Euclidean space $\R^d$. The functions $f_d \colon \R \times \R^d\to \R$, $d \in \N$, describe the nonlinearities of the PDEs in \eqref{eq:PDE_intro} and the functions $g_d\colon \R^d\to \R$, $d\in \N$, describe the initial conditions of the PDEs in \eqref{eq:PDE_intro}. The quantities $\lfloor n^{ 1 / 4 } \rfloor$, $n \in \N,$ in \eqref{eq:intro_main} in Theorem~\ref{thm:introduction} above describe evaluations of the standard floor function in the sense that for all $n \in \N$ it holds that $\lfloor n^{ 1 / 4 }  \rfloor = \max( [0,n^{ 1 / 4 }]\cap \N )$.

Theorem~\ref{thm:introduction} above in this introductory section is a special case of the more general approximation results in Section~\ref{sec:rate} in this article and these more general approximations results treat more general PDEs than \eqref{eq:PDE_intro} as well as more general MLP approximation methods than \eqref{eq:def:U_intro}. More specifically, in \eqref{eq:PDE_intro} above we have for every $d \in \N$ that the nonlinearity $f_d$ depends only on the PDE solution $u_d$ and the spatial  gradient $\nabla_x u_d$ of the PDE solution but not on $t \in [0,T]$ and $x \in \R^d$ 
while in Corollary~\ref{cor:comp_and_error}, Theorem~\ref{cor:comp_and_error2}, and Corollary~\ref{cor:comp_and_error3} in Section~\ref{sec:rate} the nonlinearities of the PDEs may also depend on $t \in [0,T]$ and $x \in \R^d$.
Corollary~\ref{cor:comp_and_error} and Theorem~\ref{cor:comp_and_error2} also provide error analyses for a more general class of MLP approximation methods. In particular, in Theorem~\ref{thm:introduction} above the family $\unif^\theta\colon \Omega \to (0,1)$, $\theta \in (0,1)$, of i.i.d.\ random variables satisfies for all $b\in (0,1)$ that $\P(\unif^0\le b)=\sqrt{b}$
and Corollary~\ref{cor:comp_and_error} and Theorem~\ref{cor:comp_and_error2} are proved under the more general hypothesis that there exists $\alpha\in (0,1)$ such that for 
all $b\in (0,1)$ it holds that $\P(\unif^0\le b)=b^\alpha$  (see, e.g., \eqref{eq:def:Ucor2} in Theorem~\ref{cor:comp_and_error2}). Furthermore, the more general approximation result in Corollary~\ref{cor:comp_and_error} in Section~\ref{sec:rate} also provides an explicit upper bound for the constant $c \in \R$ in Theorem~\ref{thm:introduction} above (see \eqref{eq:fin_cor} in Corollary~\ref{cor:comp_and_error}). 

The remainder of this article is organized as follows. In Section~\ref{sec:it_int} we establish a few identities and upper bounds for certain iterated deterministic integrals. The results of Section~\ref{sec:it_int} are then used in Section~\ref{sec:error_analysis} in which we introduce and analyze the considered MLP approximation methods. In Section~\ref{sec:ex_sol} we establish suitable a priori bounds for exact solutions of PDEs of the form \eqref{eq:PDE_intro}. In Section~\ref{sec:rate} we combine the findings of Sections~\ref{sec:error_analysis} and~\ref{sec:ex_sol} to etablish in Theorem~\ref{cor:comp_and_error2} below the main approximation result of this article.

\section{Analysis of certain deterministic iterated integrals}\label{sec:it_int}

In this section we establish in Corollary~\ref{cor:ub.it.int.lp} below an upper bound for products of certain independent random variables. Corollary~\ref{cor:ub.it.int.lp} below is a central ingredient in our error analysis for MLP approximations in Section~\ref{sec:error_analysis} below. Our proof of Corollary~\ref{cor:ub.it.int.lp} employs a few elementary identities and estimates for certain deterministic iterated integrals which are provided in Lemma~\ref{lem:iterated_int}, Lemma~\ref{lem:iterated_int_beta}, and Corollary~\ref{cor:iterated_int_beta_bound} below.

\subsection{Identities for certain deterministic iterated integrals}

\begin{lemma}\label{lem:iterated_int}
Let $T,\beta,\gamma \in(0,\infty)$, let $\rho\colon(0,1)\to(0,\infty)$ be $\mathcal B((0,1))/\mathcal B((0,\infty))$-measurable, and
let $\rhoh\colon [0,T]^2\to (0,\infty)$ satisfy for all $t\in[0,T)$, $s\in(t,T]$ that $\rhoh(t,s)=\tfrac{1}{T-t}\rho(\tfrac{s-t}{T-t})$.
Then it holds for all $j\in\N$, $s_0\in [0,T)$ that
\begin{multline}\label{eq:iterated_int}
\int_{s_0}^T\frac{1}{(s_1-s_0)^\beta [\rhoh(s_0,s_1)]^\gamma}\int_{s_1}^T\frac{1}{(s_2-s_1)^\beta [\rhoh(s_1,s_2)]^\gamma }\ldots \int_{s_{j-1}}^T\frac{1}{(s_{j}-s_{j-1})^\beta [\rhoh(s_{j-1},s_{j})]^\gamma}\,ds_{j}\ldots ds_2 \,ds_1\\
=(T-s_0)^{j(1+\gamma-\beta)}
\left[\prod_{i=0}^{j-1} \int_0^1 \frac{(1-s)^{i(1+\gamma-\beta)}}{s^\beta [\rho(s)]^\gamma}\,ds\right].
\end{multline}
\end{lemma}
\begin{proof}[Proof of Lemma~\ref{lem:iterated_int}]
We prove \eqref{eq:iterated_int} by induction on $j\in \N$. For the base case $j=1$ note that integration by substitution yields that for all $s_0\in [0,T)$ it holds that
\begin{equation}
\int_{s_0}^T\frac{1}{(s_1-s_0)^\beta  [\rhoh(s_0,s_1)]^\gamma}\,ds_1=\int_0^1\frac{(T-s_0)^{1-\beta}}{z[\rhoh(s_0,s_0+z(T-s_0))]^{\gamma}}\,dz=(T-s_0)^{1+\gamma-\beta}\int_0^1 \frac{1}{z^\beta [\rho(z)]^\gamma}\,dz.
\end{equation}
This proves
\eqref{eq:iterated_int} in
the base case $j=1$.
For the induction step
$\N\ni j\rightsquigarrow j+1\in\N$ note that the induction hypothesis and integration by substitution imply
for all $s_0\in [0,T)$ that
\begin{equation}
\begin{split}
&\int_{s_0}^T\frac{1}{(s_1-s_0)^\beta [\rhoh(s_0,s_1)]^\gamma}
\bigg[\int_{s_1}^T\frac{1}{(s_2-s_1)^\beta [\rhoh(s_1,s_2)]^\gamma}\ldots \int_{s_j}^T\frac{1}{(s_{j+1}-s_j)^\beta
[\rhoh(s_j,s_{j+1})]^\gamma}\,ds_{j+1}\ldots ds_2\bigg]\, ds_1\\
&=\int_{s_0}^T\frac{(T-s_1)^{j(1+\gamma-\beta)}}{(s_1-s_0)^\beta [\rhoh(s_0,s_1)]^\gamma }\left[\prod_{i=0}^{j-1} \int_0^1 \frac{(1-s)^{i(1+\gamma-\beta)}}{s^\beta [\rho(s)]^\gamma}\,ds \right]\, ds_1\\
&=\left[\prod_{i=0}^{j-1} \int_0^1 \frac{(1-s)^{i(1+\gamma-\beta)}}{s^\beta [\rho(s)]^\gamma}\,ds \right]
\int_{0}^1\frac{(T-s_0)^{j(1+\gamma-\beta)}(1-z)^{j(1+\gamma-\beta)}}
               {(T-s_0)^{\beta}z^\beta [\rhoh(s_0,s_0+z(T-s_0))]^\gamma }
               (T-s_0)dz\\
&=(T-s_0)^{(j+1)(1+\gamma-\beta)}\left[\prod_{i=0}^j \int_0^1 \frac{(1-s)^{i(1+\gamma-\beta)}}{s^\beta [\rho(s)]^\gamma}\,ds \right].
\end{split}
\end{equation}
Induction thus proves~\eqref{eq:iterated_int}.
This completes the proof of Lemma~\ref{lem:iterated_int}.
\end{proof}

\begin{lemma}\label{lem:iterated_int_beta}
Let $\alpha \in (0,1)$, $T,\gamma \in (0,\infty)$, $\beta \in (0,\alpha \gamma+1)$
and
let $\rho\colon(0,1)\to(0,\infty)$ 
and $\rhoh\colon [0,T]^2\to (0,\infty)$ 
satisfy for all $r\in(0,1)$, $t\in[0,T)$, $s\in(t,T]$ that 
$\rho(r)=\frac{1-\alpha}{r^\alpha}$
and $\rhoh(t,s)=\tfrac{1}{T-t}\rho(\tfrac{s-t}{T-t})$.
Then it holds for all $j\in \N$, $s_0\in [0,T)$ that
\begin{multline}\label{eq:iterated_int_beta}
\int_{s_0}^T\frac{1}{(s_1-s_0)^\beta [\rhoh(s_0,s_1)]^\gamma}\int_{s_1}^T\frac{1}{(s_2-s_1)^\beta [\rhoh(s_1,s_2)]^\gamma }\ldots \int_{s_{j-1}}^T\frac{1}{(s_{j}-s_{j-1})^\beta [\rhoh(s_{j-1},s_{j})]^\gamma}\,ds_{j}\ldots ds_2 \,ds_1\\
=\left[\frac{(T-s_0)^{(1+\gamma-\beta)}\Gamma(\alpha\gamma-\beta+1)}{(1-\alpha)^\gamma}\right]^{j}\left[\prod_{i=0}^{j-1}
\frac{\Gamma(i(1+\gamma-\beta)+1)}{\Gamma(\alpha\gamma-\beta+i(1+\gamma-\beta)+2)}\right]. 
\end{multline}
\end{lemma}

\begin{proof}[Proof of Lemma~\ref{lem:iterated_int_beta}]
Throughout this proof let $B \colon (0,\infty)\times (0,\infty)\to \R $ satisfy for all $x,y\in (0,\infty)$ that 
\begin{equation}\label{eq:def_beta}
B(x,y)=\int_0^1s^{x-1}(1-s)^{y-1}\,ds.
\end{equation}
Note that \eqref{eq:def_beta} and the fact that for all $x, y \in (0,\infty)$ it holds that $B(x,y) = \frac{\Gamma(x)\Gamma(y)}{ \Gamma(x+y) }$ ensure that for all 
$i\in \N_0$ it holds that
\begin{equation}\label{eq:clalc.srhos}
\begin{split}
\int_0^1 \frac{(1-s)^{i(1+\gamma-\beta)}}{s^\beta [\rho(s)]^\gamma}\,ds
& =\frac{1}{(1-\alpha)^\gamma }\int_0^1s^{\alpha\gamma-\beta}(1-s)^{i(1+\gamma-\beta)} \,ds
 =\frac	{B(\alpha\gamma-\beta+1,i(1+\gamma-\beta)+1)}{(1-\alpha)^\gamma}\\
 &=\frac{\Gamma(\alpha\gamma-\beta+1)\Gamma(i(1+\gamma-\beta)+1)}{(1-\alpha)^\gamma\Gamma(\alpha\gamma-\beta+i(1+\gamma-\beta)+2)}.
 \end{split}
\end{equation}
Lemma~\ref{lem:iterated_int} hence
implies that for all $j\in \N$, $s_0\in [0,T)$ it holds that 
\begin{equation}
\begin{split}
&\int_{s_0}^T\frac{1}{(s_1-s_0)^\beta [\rhoh(s_0,s_1)]^\gamma}\int_{s_1}^T\frac{1}{(s_2-s_1)^\beta [\rhoh(s_1,s_2)]^\gamma }\ldots \int_{s_{j-1}}^T\frac{1}{(s_{j}-s_{j-1})^\beta [\rhoh(s_{j-1},s_{j})]^\gamma}\,ds_{j}\ldots ds_2\, ds_1\\
&=
(T-s_0)^{j(1+\gamma-\beta)}\left[\prod_{i=0}^{j-1}\int_0^1 \frac{(1-s)^{i(1+\gamma-\beta)}}{s^\beta [\rho(s)]^\gamma}\,ds\right]
\\
&=(T-s_0)^{j(1+\gamma-\beta)}\left[\prod_{i=0}^{j-1} \frac{\Gamma(\alpha\gamma-\beta+1)\Gamma(i(1+\gamma-\beta)+1)}{(1-\alpha)^\gamma\Gamma(\alpha\gamma-\beta+i(1+\gamma-\beta)+2)}\right]\\
&=
\left[\frac{(T-s_0)^{(1+\gamma-\beta)}\Gamma(\alpha\gamma-\beta+1)}{(1-\alpha)^\gamma}\right]^{j}\left[\prod_{i=0}^{j-1}
\frac{\Gamma(i(1+\gamma-\beta)+1)}{\Gamma(\alpha\gamma-\beta+i(1+\gamma-\beta)+2)}\right]. 
\end{split}
\end{equation}
This establishes \eqref{eq:iterated_int_beta}. The proof of Lemma \ref{lem:iterated_int_beta} is thus completed.
\end{proof}

\subsection{Estimates for certain deterministic iterated integrals}

\begin{corollary}\label{cor:iterated_int_beta_bound}
Let $\alpha \in (0,1)$, $T,\gamma \in (0,\infty)$, $\beta \in [\alpha\gamma,\alpha \gamma+1]$ and
let $\rho\colon(0,1)\to(0,\infty)$ and
$\rhoh\colon [0,T]^2\to (0,\infty)$ satisfy for all
$r\in (0,1)$, $t\in[0,T)$, $s\in(t,T]$ that
$\rho(r)=\frac{1-\alpha}{r^\alpha}$
and $\rhoh(t,s)=\tfrac{1}{T-t}\rho(\tfrac{s-t}{T-t})$.
 Then it holds for all $j\in \N$, $s_0\in [0,T)$ that
% \blue{
% \begin{multline}\label{eq:iterated_int_beta_bound}
%\int_{s_0}^T\frac{1}{(s_1-s_0)^\beta [\rhoh(s_0,s_1)]^\gamma}\int_{s_1}^T\frac{1}{(s_2-s_1)^\beta [\rhoh(s_1,s_2)]^\gamma }\ldots \int_{s_j}^T\frac{1}{(s_{j+1}-s_j)^\beta [\rhoh(s_j,s_{j+1})]^\gamma}ds_{j+1}\ldots ds_2 ds_1\\
%\le \left[\frac{(T-s_0)^{(1+\gamma-\beta)}(\alpha \gamma -\beta +2)^{(\beta -\alpha \gamma)}\Gamma(\alpha\gamma-\beta+1)}{(1-\alpha)^\gamma( \gamma -\beta +1)^{(\alpha \gamma -\beta +1)}}\right]^{j+1}
%\left[
%\frac{\Gamma\left(\frac{1}{1+\gamma-\beta}\right)}{\Gamma\left(1+j+\frac{1}{1+\gamma-\beta}\right)}
%\right]^{\alpha \gamma -\beta +1}
%\end{multline}
%and that}
\begin{multline}\label{eq:iterated_int_beta_bound2}
\int_{s_0}^T\frac{1}{(s_1-s_0)^\beta [\rhoh(s_0,s_1)]^\gamma}\int_{s_1}^T\frac{1}{(s_2-s_1)^\beta [\rhoh(s_1,s_2)]^\gamma }\ldots \int_{s_{j-1}}^T\frac{1}{(s_{j}-s_{j-1})^\beta [\rhoh(s_{j-1},s_{j})]^\gamma}\,ds_{j}\ldots ds_2 \,ds_1\\
\le \left[\frac{(T-s_0)^{(1+\gamma-\beta)}\Gamma(\alpha\gamma-\beta+1)}{(1-\alpha)^\gamma
(1+\gamma-\beta)^{\alpha \gamma-\beta+1} }\right]^{j}
\left[e^{1+\gamma-\beta}((1+\gamma-\beta)(j-1)+1) \right]^{\frac{(\beta-\alpha \gamma)(\alpha\gamma-\beta+1)}{1+\gamma-\beta}}
\left[
\tfrac{\Gamma\left(\frac{1}{1+\gamma-\beta}\right)}{\Gamma\left(j+\frac{1}{1+\gamma-\beta}\right)}
\right]^{\alpha \gamma -\beta +1}
.
\end{multline}
%\blue{In the current version we only need \eqref{eq:iterated_int_beta_bound2}, for the final version we can thus delete \eqref{eq:iterated_int_beta_bound} and its proof.}
\end{corollary}
\begin{proof}[Proof of Corollary~\ref{cor:iterated_int_beta_bound}]
First, observe that Wendel's inequality for the gamma function
(see, e.g., Wendel~\cite{wendel1948note} and Qi~\cite[Section 2.1]{Qi2010})
ensures that for all $x\in (0,\infty)$, $s\in [0,1]$ it holds that
\begin{equation}\label{eq:wendel}
\frac{\Gamma(x)}{\Gamma(x+s)}\le \frac{1}{x^s} \left[\frac{x+s}{x} \right]^{1-s}.
\end{equation}
Moreover, 
note that the fact that for all $x\in (0,\infty)$ it holds that $\ln^\prime(x)=x^{-1}$ demonstrates that
for all $j\in \N$, $\lambda \in (0,\infty)$ it holds that
\begin{equation}
\sum_{i=0}^{j-1}\frac{1}{i+\lambda}=\frac{1}{\lambda}+\sum_{i=1}^{j-1}\frac{1}{i+\lambda}\le \frac{1}{\lambda}+\sum_{i=1}^{j-1}\int_{i}^{i+1} \frac{1}{s-1+\lambda}\,ds =\frac{1}{\lambda}+\int_1^{j} \frac{1}{s-1+\lambda}\,ds=\frac{1}{\lambda}+\ln\!\left(\frac{j-1}{\lambda}+1\right).
\end{equation}
%(Alternative: $\sum_{i=1}^{j+1}\frac{1}{i}\le \gamma + \log(j+2)$ with $\gamma$ the Euler-Mascheroni constant)
Combining this, \eqref{eq:wendel}, and the fact that $\alpha \gamma-\beta+1\in [0,1]$ with the fact that for all $x\in [0,\infty)$ it holds that $1+x\le e^{x}$ proves that for all $j\in \N$ it holds that
\begin{equation}
\begin{split}
&\left[
\prod_{i=0}^{j-1}
\frac{\Gamma(i(1+\gamma-\beta)+1)}{\Gamma(\alpha\gamma-\beta+i(1+\gamma-\beta)+2)}
\right]
\\&\le 
\left[\prod_{i=0}^{j-1}
\left(\frac{\alpha\gamma-\beta+i(1+\gamma-\beta)+2}{i(1+\gamma-\beta)+1} \right)^{1-(\alpha \gamma-\beta+1) }\frac{1}{(i(1+\gamma-\beta)+1)^{\alpha \gamma -\beta +1}}
\right]\\
&=
\left[\prod_{i=0}^{j-1}
\left(1+\frac{\alpha\gamma-\beta+1}{i(1+\gamma-\beta)+1} \right)^{\beta-\alpha \gamma }\frac{1}{(i(1+\gamma-\beta)+1)^{\alpha \gamma -\beta +1}}
\right]\\
&\le 
\left[\prod_{i=0}^{j-1}
e^{\frac{(\beta-\alpha \gamma)(\alpha\gamma-\beta+1)}{i(1+\gamma-\beta)+1}}
\frac{1}{(i(1+\gamma-\beta)+1)^{\alpha \gamma -\beta +1}}
\right]\\
&=
e^{\frac{(\beta-\alpha \gamma)(\alpha\gamma-\beta+1)}{1+\gamma-\beta}\sum_{i=0}^{j-1}\frac{1}{i+\frac{1}{1+\gamma-\beta}}}
\left[\prod_{i=0}^{j-1}\frac{1}{(i(1+\gamma-\beta)+1)^{\alpha \gamma -\beta +1}}\right]\\
&\le
 e^{\frac{(\beta-\alpha \gamma)(\alpha\gamma-\beta+1)}{1+\gamma-\beta}(1+\gamma-\beta+\ln((1+\gamma-\beta)(j-1)+1))}
\left[\prod_{i=0}^{j-1}\frac{1}{(i(1+\gamma-\beta)+1)^{\alpha \gamma -\beta +1}}\right]\\
&=
\left[e^{1+\gamma-\beta}((1+\gamma-\beta)(j-1)+1) \right]^{\frac{(\beta-\alpha \gamma)(\alpha\gamma-\beta+1)}{1+\gamma-\beta}}
\left[
\tfrac{\Gamma\left(\frac{1}{1+\gamma-\beta}\right)}{(1+\gamma-\beta)^{j}\Gamma\left(j+\frac{1}{1+\gamma-\beta}\right)}
\right]^{\alpha \gamma -\beta +1}.
\end{split}
\end{equation}
Lemma~\ref{lem:iterated_int_beta} hence implies that for all $j\in \N$, $s_0\in [0,T)$ it holds that 
\begin{equation}
\begin{split}
&\int_{s_0}^T\frac{1}{(s_1-s_0)^\beta [\rhoh(s_0,s_1)]^\gamma}\int_{s_1}^T\frac{1}{(s_2-s_1)^\beta [\rhoh(s_1,s_2)]^\gamma }\ldots \int_{s_{j-1}}^T\frac{1}{(s_{j}-s_{j-1})^\beta [\rhoh(s_{j-1},s_{j})]^\gamma}\,ds_{j}\ldots ds_2 \,ds_1\\
&\le \left[\tfrac{(T-s_0)^{(1+\gamma-\beta)}\Gamma(\alpha\gamma-\beta+1)}{(1-\alpha)^\gamma
(1+\gamma-\beta)^{\alpha \gamma-\beta+1} }\right]^{j}
\left[e^{1+\gamma-\beta}((1+\gamma-\beta)(j-1)+1) \right]^{\frac{(\beta-\alpha \gamma)(\alpha\gamma-\beta+1)}{1+\gamma-\beta}}
\left[
\tfrac{\Gamma\left(\frac{1}{1+\gamma-\beta}\right)}{\Gamma\left(j+\frac{1}{1+\gamma-\beta}\right)}
\right]^{\alpha \gamma -\beta +1}.
\end{split}
\end{equation}
This establishes \eqref{eq:iterated_int_beta_bound2}. 
The proof of Corollary~\ref{cor:iterated_int_beta_bound} is thus completed.
\end{proof}

\subsection{Estimates for products of certain independent random variables}

%\begin{lemma}\label{lem:fac_lemma0}
%Let $T\in (0,\infty)$, $d\in \N$, $F\in C((0,1)\times [0,T) \times \R^d, [0,\infty))$, let $(\Omega,\mathcal F,\P)$ be a probability space, let $\rho\colon \Omega\to (0,1)$ and $\tau\colon \Omega \to (0,T)$ be random variables, let $W\colon [0,T]\times \Omega \to \R^d$ be a standard Brownian motion with continuous sample paths, let $f\colon [0,T) \to [0,\infty]$ satisfy for all $t\in [0,T)$ that $f(t)=\E[F(\rho, t, W_{t+(T-t)\rho}-W_t)]$, let $\mathcal G\subseteq \mathcal F$ be a sigma-algebra, assume that $\rho$, $\tau$, and $W$ are independent, assume that $\tau$ is $\mathcal G/\mathcal B((0,T))$-measurable, and assume that $\sigma(\rho, (W_{\max\{s,\tau\}}-W_\tau)_{s\in [0,T]})$ and $\mathcal G$ are independent. Then it holds $\P$-a.s.\ that
%\begin{equation}
%\E[F(\rho,\tau, W_{\tau+(T-\tau)\rho}-W_\tau)|\mathcal G ]=f(\tau).
%\end{equation}
%\end{lemma}
%\begin{proof}[Proof of Lemma~\ref{lem:fac_lemma0}]
%Lemma~\ref{lem:fac_lemma} follows from \cite[Lemma~2.9]{jentzen2018exponential} (applied with
%$D=(0,T)$, $\mathcal D=\mathcal B((0,T))$, $E=(0,1)\times C([0,T],\R^d)$, 
%$\mathcal E=\mathcal B((0,1))\otimes \mathcal B(C([0,T],\R^d))$,  
% $\mathcal X=\mathcal G$, 
% $\mathcal Y=\sigma(\rho, (W_{\max\{s,\tau\}}-W_\tau)_{s\in [0,T]})$,
% $X=\tau$,
% $Y=(\rho, (W_{\max\{s,\tau\}}-W_\tau)_{s\in [0,T]})$,
% $\Phi(t,(r,x))=F(r,t,x_{t+(T-t)r}-x_t)$ for $t\in (0,T)$, $r\in (0,1)$, $x\in C([0,T],\R^d)$ in the notation of  \cite[Lemma~2.9]{jentzen2018exponential}). 
%\end{proof}

\begin{lemma}\label{lem:fac_lemma}
Let $T\in (0,\infty)$, $d\in \N$, $F\in C((0,1)\times [0,T) \times \R^d, [0,\infty))$, let $(\Omega,\mathcal F,\P)$ be a probability space, let $\rho\colon \Omega\to (0,1)$ and $\tau\colon \Omega \to (0,T)$ be random variables, let $W\colon [0,T]\times \Omega \to \R^d$ be a standard Brownian motion with continuous sample paths, let $f\colon [0,T) \to [0,\infty]$ satisfy for all $t\in [0,T)$ that $f(t)=\E[F(\rho, t, W_{t+(T-t)\rho}-W_t)]$, let $\mathcal G\subseteq \mathcal F$ be a sigma-algebra, 
let $\mathcal H=\sigma(\mathcal G \cup \sigma(\tau,(W_{\min\{s,\tau\}})_{s\in [0,T]}))$, and
assume that $\rho$, $\tau$, $W$, and $\mathcal G$ are independent. Then it holds $\P$-a.s.\ that
\begin{equation}\label{eq:fac_lemma}
\E[F(\rho,\tau, W_{\tau+(T-\tau)\rho}-W_\tau)|\mathcal H ]=f(\tau).
\end{equation}
\end{lemma}

\begin{proof}[Proof of Lemma~\ref{lem:fac_lemma}]
First, note that independence of $\rho$, $\tau$, $W$, and $\mathcal G$ ensures that it holds $\P$-a.s.\ that
\begin{equation}\label{eq:fac_lemma_ind}
\E[F(\rho,\tau, W_{\tau+(T-\tau)\rho}-W_\tau)|\mathcal H ]=
\E[F(\rho,\tau, W_{\tau+(T-\tau)\rho}-W_\tau)|\sigma(\tau,(W_{\min\{s,\tau\}})_{s\in [0,T]})]
.
\end{equation}
Next note that Hutzenthaler et al.~\cite[Lemma 2.2]{HJKNW2018} (applied with $\mathcal G=\sigma(\rho, (W_{s})_{s\in [0,T]})$, $S=(0,T)$, $\mathcal S=\mathcal B((0,T))$, $U(t,\omega)=\mathbf 1_{A}(t)\mathbf 1_{B}((W_{\min\{s,t\}}(\omega))_{s\in [0,T]})) F(\rho(\omega),t, W_{t+(T-t)\rho(\omega)}(\omega)-W_t(\omega))$, and $Y=\tau$ 
for $t\in (0,T)$, $\omega \in \Omega$ in the notation of \cite[Lemma 2.2]{HJKNW2018}) proves that for all $A\in \mathcal B((0,T))$, $B\in \mathcal B(C([0,T],\R^d))$ it holds that
\begin{equation}
\begin{split}
&\E[\mathbf 1_{A}(\tau)\mathbf 1_{B}((W_{\min\{s,\tau\}})_{s\in [0,T]}) F(\rho,\tau, W_{\tau+(T-\tau)\rho}-W_\tau)]\\
&=
\int_{(0,T)}
\E[\mathbf 1_{A}(t)\mathbf 1_{B}((W_{\min\{s,t\}})_{s\in [0,T]}) F(\rho,t, W_{t+(T-t)\rho}-W_t)](\tau(\P))(dt).
\end{split}
\end{equation}
Independence of Brownian increments hence proves that for all $A\in \mathcal B((0,T))$, $B\in \mathcal B(C([0,T],\R^d))$ it holds that
\begin{equation}
\begin{split}
&\E[\mathbf 1_{A}(\tau)\mathbf 1_{B}((W_{\min\{s,\tau\}})_{s\in [0,T]}) F(\rho,\tau, W_{\tau+(T-\tau)\rho}-W_\tau)]\\
&=
\int_{(0,T)}
\E[\mathbf 1_{A}(t)\mathbf 1_{B}((W_{\min\{s,t\}})_{s\in [0,T]})] \E[F(\rho,t, W_{t+(T-t)\rho}-W_t)](\tau(\P))(dt)\\
&=
\int_{(0,T)}
\E[\mathbf 1_{A}(t)\mathbf 1_{B}((W_{\min\{s,t\}})_{s\in [0,T]})] f(t)(\tau(\P)) (dt).
\end{split}
\end{equation}
\sloppy
Hence, Hutzenthaler et al.~\cite[Lemma 2.2]{HJKNW2018} (applied with $\mathcal G=\sigma(\rho, (W_{s})_{s\in [0,T]})$, $S=(0,T)$, $\mathcal S=\mathcal B((0,T))$, $U(t,\omega)=\mathbf 1_{A}(t)\mathbf 1_{B}((W_{\min\{s,t\}}(\omega))_{s\in [0,T]})) g(t)$, and $Y=\tau$ 
for $t\in (0,T)$, $\omega \in \Omega$ in the notation of \cite[Lemma 2.2]{HJKNW2018}) proves that for all $A\in \mathcal B((0,T))$, $B\in \mathcal B(C([0,T],\R^d))$ it holds that
\begin{equation}
\begin{split}
\E[\mathbf 1_{A}(\tau)\mathbf 1_{B}((W_{\min\{s,\tau\}})_{s\in [0,T]}) F(\rho,\tau, W_{\tau+(T-\tau)\rho}-W_\tau)]
&=
\int_{(0,T)}
\E[\mathbf 1_{A}(t)\mathbf 1_{B}((W_{\min\{s,t\}})_{s\in [0,T]})f(t)] (\tau(\P)) (dt)\\
&=
\E[\mathbf 1_{A}(\tau)\mathbf 1_{B}((W_{\min\{s,\tau\}})_{s\in [0,T]})f(\tau)].
\end{split}
\end{equation}
This together with \eqref{eq:fac_lemma_ind} proves \eqref{eq:fac_lemma}. The proof of Lemma~\ref{lem:fac_lemma} is thus completed.
\end{proof}

\begin{corollary}\label{cor:ub.it.int.lp}
Let $T\in(0,\infty)$, $d\in \N$,
$j\in \N_0$, 
$\uvec_1=(1,0,0\ldots,0)$, $\uvec_2=(0,1,0,\ldots,0)$, \ldots, 
$\uvec_{d+1}=(0,0,\ldots, 0,1)\in \R^{d+1}$,
$\nu_0,\nu_1, \ldots, \nu_j\in \{1,2,\ldots, d+1\}$,
$\alpha\in(0,1)$, $p\in (1,\infty)$ satisfy
$\alpha(p-1)\le \frac{p}{2} \le \alpha(p-1)+1$,
let $\langle \cdot, \cdot \rangle\colon \R^{d+1}\times \R^{d+1} \to \R$ 
be the standard scalar product on $\R^{ d + 1 }$,
%satisfy for all $v=(v_1,v_2,\ldots, v_{d+1})$, $w=(w_1,w_2,\ldots, w_{d+1})\in \R^{d+1}$ that $\langle v, w\rangle=\sum_{i=1}^{d+1}v_iw_i$,
let
$
  ( 
    \Omega, \mathcal{F}, \P
  )
$
be a probability space,
let
$
  W=(W^1,W^2,\ldots,W^d) \colon [0,T] \times \Omega \to \R^d 
$ 
be a standard Brownian motion
with continuous sample paths,
let $\unif^{(n)}\colon \Omega\to(0,1)$, $n\in \N_0$, be i.i.d.\ random variables, assume that $W$ and $(\unif^{(n)})_{n\in \N_0}$ are independent,
let $\rho\colon(0,1)\to(0,\infty)$ 
and $\rhoh\colon [0,T]^2\to (0,\infty)$ satisfy
for all $b\in (0,1)$, $t\in[0,T)$, $s\in(t,T]$ that
$\rho(b)=\frac{1-\alpha}{b^\alpha}$,
 $\P(\unif^{(0)}\le b)=\int_0^b \rho(u)\,du$, 
and $\rhoh(t,s)=\tfrac{1}{T-t}\rho(\tfrac{s-t}{T-t})$,
let $\ES\colon \N_0 \times [0,T)\times \Omega \to [0,T)$
 satisfy for all $n \in \N_0$, $t\in [0,T)$
that $\ES(0,t)=t$ and
$\ES(n+1,t)=\ES(n,t)+(T-\ES(n,t))\unif^{(n)}$, and let $t\in [0,T)$.
 Then
\begin{equation}\label{eq:ub.it.int.lp.b}
\begin{split}
&\E\!\left[ \left| \textstyle\prod\limits_{i=0}^{j}\displaystyle
   \tfrac{1}{\rhoh(\ES(i,t),\ES(i+1,t))}
   \big\langle \uvec_{\nu_{i}},
     \big(   
   1,
     \tfrac{ 
     W_{\ES(i+1,t)}- W_{\ES(i,t)}
     }{\ES(i+1,t)-\ES(i,t) } 
     \big)\big\rangle  \right|^p \right]
     \\
      &\leq
      \left[
      \max\!\left\{
      (T-t)^{\frac{p}{2}}, \tfrac{2^{\frac{p}{2}}\Gamma(\frac{p+1}{2})}{\sqrt{\pi}}
      \right\}
    \tfrac{(T-t)^{\frac{p}{2}}\Gamma(\frac{p}{2})}{(1-\alpha)^{p-1}
(\frac{p}{2})^{\alpha(p-1)-\frac{p}{2}+1} }\right]^{j+1}
\left[e^{\frac{p}{2}}\left(\tfrac{pj}{2}+1\right) \right]^{\frac{1}{2p}}
\left[
\tfrac{\Gamma(\frac{2}{p})}{\Gamma(1+j+\frac{2}{p})}
\right]^{\alpha (p-1) -\frac{p}{2} +1}.
  \end{split}
\end{equation}
\end{corollary}
\begin{proof}[Proof of Corollary~\ref{cor:ub.it.int.lp}]
Throughout this proof let $\mathbb F_n\subseteq \mathcal F$, $n\in \N_0$, satisfy for all $n\in \N$ that  $\mathbb F_0=\{\emptyset, \Omega\}$ and that
$\mathbb F_n=\sigma(\mathbb F_{n-1} \cup \sigma(\ES(n,t), (W_{\min\{s,\ES(n,t)\}})_{s\in [0,T]}))$
and let $v = ( v_1, v_2, \ldots, v_d) \in \R^d$ satisfy $v_1 = v_2 = ... = v_{ d } = 1$.
Note that for all $r\in [0,T)$, $s\in [r,T]$, $i\in\{1,2,\ldots,d\}$, $n\in \N$ it holds that
$\ES(n,r)>\ES(n-1,r)$ and
\begin{equation}\label{eq:norm_moment}
  \E\big[|W^i_s-W^i_r|^p\big]=\frac{(2(s-r))^{\frac{p}{2}}\Gamma\!\left(\frac{p+1}{2}\right)}{\sqrt{\pi}}.
\end{equation}
Next we claim that for all $k\in \{1,2,\ldots,j+1\}$ it holds $\P$-a.s.\ that
%{\footnotesize
\begin{equation}\label{eq:ub.it.int.lp2}
\begin{split}
&
\E\! \left[
\prod_{i=k}^{j+1}
\left|
   \tfrac{1}{\rhoh(\ES(i-1,t),\ES(i,t))}
    \big\langle \uvec_{\nu_{i-1}},
     \big(   
   1,
     \tfrac{ 
     W_{\ES(i,t)}- W_{\ES(i-1,t)}
     }{\ES(i,t)-\ES(i-1,t) } 
     \big)\big\rangle
     \right|^p
     \Bigg | \mathbb F_{k-1}
     \right]
     =
     \int_{\ES(k-1,t)}^T 
     \tfrac{
      \big\langle \uvec_{\nu_{k-1}},
     \big(   
   1,
   \big(
    \frac{2}{s_k-\ES(k-1,t)}
    \big)^{\frac{p}{2}}
    \frac{\Gamma\!\left(\frac{p+1}{2}\right)}{\sqrt{\pi}}v
     \big)\big\rangle
     }
     {[\rhoh(\ES(k-1,t),s_k)]^{p-1}}
      \\
      &
      \qquad 
    \cdot 
     \int_{s_k}^T 
     \tfrac{
      \big\langle \uvec_{\nu_{k}},
     \big(   
   1,
   \big(
    \frac{2}{s_{k+1}-s_k}
    \big)^{\frac{p}{2}}
    \frac{\Gamma\!\left(\frac{p+1}{2}\right)}{\sqrt{\pi}}v
     \big)\big\rangle}
     {[\rhoh(s_k,s_{k+1})]^{p-1}}
     \ldots
      \int_{s_j}^T 
     \tfrac{
 \big\langle \uvec_{\nu_{j}},     
     \big(   
   1,
   \left(
    \frac{2}{s_{j+1}-s_j}
    \right)^{\frac{p}{2}}
    \frac{\Gamma\!\left(\frac{p+1}{2}\right)}{\sqrt{\pi}}v
     \big)\big\rangle}
     {[\rhoh(s_j,s_{j+1})]^{p-1}}
     \,ds_{j+1}
     \ldots
     ds_{k+1}
     \,ds_k
 .
  \end{split}
\end{equation}
To prove \eqref{eq:ub.it.int.lp2} we proceed by backward induction on $k\in \{1,2,\ldots,j+1\}$.
For the base case $k=j+1$ note that the fact that 
$\ES(j+1,t)=\ES(j,t)+(T-\ES(j,t))\unif^{(j)}$,
Lemma~\ref{lem:fac_lemma} (applied with $F(r,s,x)=\big|\frac{1}{\varrho(s,s+(T-s)r)}\big\langle\uvec_{\nu_j},\big(   
   1,
     \tfrac{ 
     x
     }{(T-s)r } 
     \big)\big\rangle
     \big|^p$, $\rho=\unif^{(j)}$, $\tau=\ES(j,t)$, 
     $\mathcal G=\mathbb{F}_{j-1}$
     for $r\in (0,1)$, $s\in[0,T)$, $x\in \R^d$ in the notation of
     Lemma~\ref{lem:fac_lemma}),
%disintegration of a random function and an independent random argument 
Hutzenthaler et al.~\cite[Lemma 2.3]{HJKNW2018},
the hypothesis that $W$ and $\unif^{(j)}$ are independent,
\eqref{eq:norm_moment},
and the fact that for all $r\in [0,T)$, $s\in (r,T]$ it holds that $\rhoh(r,s)=\frac{1}{T-r}\rho(\frac{s-r}{T-r})$
ensure that it holds $\P$-a.s.\ that
\begin{equation}\label{eq:ub.it.int.lp2a}
\begin{split}
&
\E\! \left[
\left|
   \tfrac{1}{\rhoh(\ES(j,t),\ES(j+1,t))}
    \big\langle \uvec_{\nu_{j}},
     \big(   
   1,
     \tfrac{ 
     W_{\ES(j+1,t)}- W_{\ES(j,t)}
     }{\ES(j+1,t)-\ES(j,t) } 
     \big)\big\rangle
     \right|^p
     \Bigg | \mathbb F_{j}
     \right]
%     =\E\! \left[
%\left|
%   \tfrac{1}{\rhoh(\ES(j,t),\ES^{(j)}_{\ES(j,t)})}
%     \Big(   
%   1,
%     \tfrac{ 
%     W_{\ES^{(j)}_{\ES(j,t)}}- W_{\ES(j,t)}
%     }{\ES^{(j)}_{\ES(j,t)}-\ES(j,t) } 
%     \Big)_{\nu_{j}}
%     \right|^p
%     \Bigg | \mathbb F_{j}
%     \right]
     \\
     &
     =\E\! \left[
\left|
   \tfrac{1}{\rhoh(\ES(j,t),\ES(j,t)+(T-\ES(j,t))\unif^{(j)})}
    \big\langle \uvec_{\nu_{j}},
     \big(   
   1,
     \tfrac{ 
     W_{\ES(j,t)+(T-\ES(j,t))\unif^{(j)}}- W_{\ES(j,t)}
     }{(T-\ES(j,t))\unif^{(j)} } 
     \big)\big\rangle
     \right|^p
     \Bigg | \mathbb F_{j}
     \right]
     \\
     &
     =\E\! \left[
\left|
   \tfrac{1}{\rhoh(s,s+(T-s)\unif^{(j)})}
    \big\langle \uvec_{\nu_{j}},
     \big(   
   1,
     \tfrac{ 
     W_{s+(T-s)\unif^{(j)}}- W_{s}
     }{(T-s)\unif^{(j)} } 
     \big)\big\rangle
     \right|^p
     \right]\Bigg |_{s=\ES(j,t)}
     \\
     &
     =
\int_0^1     
     \tfrac{
      \E\! \left[
     \left|
      \big\langle \uvec_{\nu_{j}},
     \big(   
   1,
     \frac{ 
     W_{s+(T-s)r}- W_{s}
     }{(T-s)r } 
     \big)\big\rangle
     \right|^p
     \right]
     }{[\rhoh(s,s+(T-s)r)]^p}
    \rho(r)
     dr\Bigg|_{s=\ES(j,t)}
     \\
     &
     =
\int_s^T     
     \tfrac{
      \E\! \left[
     \left|
      \big\langle \uvec_{\nu_{j}},
     \big(   
   1,
     \frac{ 
     W_{s_{j+1}}- W_{s}
     }{s_{j+1}-s } 
     \big)\big\rangle
     \right|^p
     \right]
     }{[\rhoh(s,s_{j+1})]^p}
     \tfrac{1}{T-s}
    \rho\!\left(\tfrac{s_{j+1}-s}{T-s}\right)
     \,ds_{j+1}\Bigg|_{s=\ES(j,t)}
        \\
     &
     =
     \int_{\ES(j,t)}^T 
     \tfrac{
      \big\langle \uvec_{\nu_{j}},
     \big(   
   1,
   \left(
    \frac{2}{s_{j+1}-\ES(j,t)}
    \right)^{\frac{p}{2}}
    \frac{\Gamma\!\left(\frac{p+1}{2}\right)}{\sqrt{\pi}}v
     \big)\big\rangle 
     }
     {[\rhoh(\ES(j,t),s_{j+1})]^{p-1}}
     \,ds_{j+1}
 .
  \end{split}
\end{equation}
This establishes \eqref{eq:ub.it.int.lp2} in the base case $k=j+1$.
For the induction step $\{2,3\ldots,j+1\} \ni k+1\rightsquigarrow k \in \{1,2,\ldots,j\}$ 
assume that 
there exists $k\in \{1,2,\ldots,j\}$ which satisfies that
\begin{equation}\label{eq:ub.it.int.lp15}
\begin{split}
&
\E\! \left[
\prod_{i=k+1}^{j+1}
\left|
   \tfrac{1}{\rhoh(\ES(i-1,t),\ES(i,t))}
    \big\langle \uvec_{\nu_{i-1}},
     \big(   
   1,
     \tfrac{ 
     W_{\ES(i,t)}- W_{\ES(i-1,t)}
     }{\ES(i,t)-\ES(i-1,t) } 
     \big)\big\rangle
     \right|^p
     \Bigg | \mathbb F_{k}
     \right]
     =
     \int_{\ES(k,t)}^T 
     \tfrac{
      \big\langle \uvec_{\nu_{k}},
     \big(   
   1,
   \big(
    \frac{2}{s_{k+1}-\ES(k,t)}
    \big)^{\frac{p}{2}}
    \frac{\Gamma\!\left(\frac{p+1}{2}\right)}{\sqrt{\pi}}v
     \big)\big\rangle
     }
     {[\rhoh(\ES(k,t),s_{k+1})]^{p-1}}
      \\
      &
      \qquad 
    \cdot 
     \int_{s_{k+1}}^T 
     \tfrac{
      \big\langle \uvec_{\nu_{k+1}},
     \big(   
   1,
   \big(
    \frac{2}{s_{k+2}-s_{k+1}}
    \big)^{\frac{p}{2}}
    \frac{\Gamma\!\left(\frac{p+1}{2}\right)}{\sqrt{\pi}}v
     \big)\big\rangle}
     {[\rhoh(s_{k+1},s_{k+2})]^{p-1}}
     \ldots
      \int_{s_j}^T 
     \tfrac{
 \big\langle \uvec_{\nu_{j}},     
     \big(   
   1,
   \left(
    \frac{2}{s_{j+1}-s_j}
    \right)^{\frac{p}{2}}
    \frac{\Gamma\!\left(\frac{p+1}{2}\right)}{\sqrt{\pi}}v
     \big)\big\rangle}
     {[\rhoh(s_j,s_{j+1})]^{p-1}}
     \,ds_{j+1}
     \ldots
     ds_{k+2}
     \,ds_{k+1}
 .
  \end{split}
\end{equation}
Observe that 
\eqref{eq:ub.it.int.lp15}, the tower property,
the
fact that the random variable 
$\tfrac{1}{\rhoh(\ES(k-1,t),\ES(k,t))}
     \Big(   
   1,
     \tfrac{ 
     W_{\ES(k,t)}- W_{\ES(k-1,t)}
     }{\ES(k,t)-\ES(k-1,t) } 
     \Big)
     $ is 
     $\mathbb F_k$-measurable,
     the induction hypothesis,
     the fact that $\ES(k,t)=\ES(k-1,t)+(T-\ES(k-1,t))\unif^{(k-1)}$,
and
the fact that conditioned on $\ES(k-1,t)$ the $\sigma$-algebras
$\sigma(\unif^{(k-1)},(W_{t}-W_{\ES(k-1,t)})_{t\in[\ES(k-1,t),T]})$
and $\mathbb{F}_{k-1}$ are independent
ensure
that it holds $\P$-a.s.\ that
\begin{equation}
\begin{split}
&
\E\! \left[
\prod_{i=k}^{j+1}
\left|
   \tfrac{1}{\rhoh(\ES(i-1,t),\ES(i,t))}
    \big\langle \uvec_{\nu_{i-1}},
     \big(   
   1,
     \tfrac{ 
     W_{\ES(i,t)}- W_{\ES(i-1,t)}
     }{\ES(i,t)-\ES(i-1,t) } 
     \big)\big\rangle 
     \right|^p
     \Bigg | \mathbb F_{k-1}
     \right]
     \\
     &
     =\E\! \Bigg[
     \left|
   \tfrac{1}{\rhoh(\ES(k-1,t),\ES(k,t))}
    \big\langle \uvec_{\nu_{k-1}},
     \big(   
   1,
     \tfrac{ 
     W_{\ES(k,t)}- W_{\ES(k-1,t)}
     }{\ES(k,t)-\ES(k-1,t) } 
     \big)\big\rangle
     \right|^p\\
     &\qquad \cdot
     \E\!\left[
\prod_{i=k+1}^{j+1}
\left|
   \tfrac{1}{\rhoh(\ES(i-1,t),\ES(i,t))}
    \big\langle \uvec_{\nu_{i-1}},
     \big(   
   1,
     \tfrac{ 
     W_{\ES(i,t)}- W_{\ES(i-1,t)}
     }{\ES(i,t)-\ES(i-1,t) } 
     \big)\big\rangle
     \right|^p
     \Bigg | \mathbb F_{k}
     \right]
     \Bigg | \mathbb F_{k-1}
     \Bigg]
      \\
      &
     =
     \E\! \Bigg[
     \left|
   \tfrac{1}{\rhoh(\ES(k-1,t),\ES(k,t)  )}
    \big\langle \uvec_{\nu_{k-1}},
     \big(   
   1,
     \tfrac{ 
     W_{\ES(k,t)}- W_{\ES(k-1,t)}
     }{\ES(k,t)-\ES(k-1,t) } 
     \big)\big\rangle 
     \right|^p
     \int_{\ES(k,t)}^T 
     \tfrac{
      \big\langle \uvec_{\nu_{k}},
     \big(   
   1,
   \left(
    \frac{2}{s_{k+1}-\ES(k,t)}
    \right)^{\frac{p}{2}}
    \frac{\Gamma\!\left(\frac{p+1}{2}\right)}{\sqrt{\pi}}v
     \big)\big\rangle
     }
     {[\rhoh(\ES(k,t),s_{k+1})]^{p-1}}
     \\
     &\qquad 
     \cdot
     \int_{s_{k+1}}^T 
     \tfrac{
      \big\langle \uvec_{\nu_{k+1}},
     \big(   
   1,
   \left(
    \frac{2}{s_{k+2}-s_{k+1}}
    \right)^{\frac{p}{2}}
    \frac{\Gamma\!\left(\frac{p+1}{2}\right)}{\sqrt{\pi}}v
     \big)\big\rangle
     }
     {[\rhoh(s_{k+1},s_{k+2})]^{p-1}}
     \ldots
      \int_{s_j}^T 
     \tfrac{
 \big\langle \uvec_{\nu_{j}},     
     \big(   
   1,
   \left(
    \frac{2}{s_{j+1}-s_j}
    \right)^{\frac{p}{2}}
    \frac{\Gamma\!\left(\frac{p+1}{2}\right)}{\sqrt{\pi}}v
     \big)\big\rangle}
     {[\rhoh(s_j,s_{j+1})]^{p-1}}
     \,ds_{j+1}
     \ldots
     ds_{k+2}
     \,ds_{k+1}
     \Bigg | \mathbb F_{k-1}
     \Bigg]
 \\ &
 =
     \E\! \Bigg[
     \Bigg|
   \tfrac{ 
 \big\langle \uvec_{\nu_{k-1}},   
   \big(   
   1,
     \frac{ 
     W_{s+(T-s)\unif^{(k-1)}}- W_{s}
     }{(T-s)\unif^{(k-1)} } 
     \big)\big\rangle}{\rhoh(s,s+(T-s)\unif^{(k-1)})}
     \Bigg|^p
     \int_{s+(T-s)\unif^{(k-1)}}^T 
     \tfrac{
 \big\langle \uvec_{\nu_{k}},     
     \big(   
   1,
   \left(
    \frac{2}{s_{k+1}-(s+(T-s)\unif^{(k-1)})}
    \right)^{\frac{p}{2}}
    \frac{\Gamma\!\left(\frac{p+1}{2}\right)}{\sqrt{\pi}}v
     \big)\big\rangle}
     {[\rhoh(s+(T-s)\unif^{(k-1)},s_{k+1})]^{p-1}}
 \\
     &\qquad 
     \cdot
     \int_{s_{k+1}}^T 
     \tfrac{
 \big\langle \uvec_{\nu_{k+1}},     
     \big(   
   1,
   \left(
    \frac{2}{s_{k+2}-s_{k+1}}
    \right)^{\frac{p}{2}}
    \frac{\Gamma\!\left(\frac{p+1}{2}\right)}{\sqrt{\pi}}v
     \big)\big\rangle}
     {[\rhoh(s_{k+1},s_{k+2})]^{p-1}}
     \ldots
      \int_{s_j}^T 
     \tfrac{
 \big\langle \uvec_{\nu_{j}},     
     \big(   
   1,
   \left(
    \frac{2}{s_{j+1}-s_j}
    \right)^{\frac{p}{2}}
    \frac{\Gamma\!\left(\frac{p+1}{2}\right)}{\sqrt{\pi}}v
     \big)\big\rangle}
     {[\rhoh(s_j,s_{j+1})]^{p-1}}
     \,ds_{j+1}
     \ldots
     ds_{k+2}
     \,ds_{k+1}
     \Bigg]
     \Bigg|_{s=\ES(k-1,t)}.
\end{split}
\end{equation}
This,
%disintegration of a random function and an independent random argument (see, e.g.,
Hutzenthaler et al.~\cite[Lemma 2.3]{HJKNW2018},
the hypothesis that $W$ and $\unif^{(k-1)}$ are independent,
\eqref{eq:norm_moment},
and the fact that for all $r\in [0,T)$, $s\in (r,T]$ it holds that $\rhoh(r,s)=\frac{1}{T-r}\rho(\frac{s-r}{T-r})$
ensure that it holds $\P$-a.s.\ that
\begin{equation}
\begin{split}
&
\E\! \left[
\prod_{i=k}^{j+1}
\left|
   \tfrac{1}{\rhoh(\ES(i-1,t),\ES(i,t))}
    \big\langle \uvec_{\nu_{i-1}},
     \big(   
   1,
     \tfrac{ 
     W_{\ES(i,t)}- W_{\ES(i-1,t)}
     }{\ES(i,t)-\ES(i-1,t) } 
     \big)\big\rangle
     \right|^p
     \Bigg | \mathbb F_{k-1}
     \right]
      \\
      &
 =
     \int_0^1
   \tfrac{ 
 \big\langle \uvec_{\nu_{k-1}},   
   \big(   
   1,
     \frac{(2((T-s)r))^{\frac{p}{2}}\Gamma\!\left(\frac{p+1}{2}\right)}
     {\sqrt{\pi}|(T-s)r|^p }v 
     \big)
     \big\rangle
     \rho(r)}{[\rhoh(s,s+(T-s)r)]^p}
     \int_{s+(T-s)r}^T 
     \tfrac{
 \big\langle \uvec_{\nu_{k}},     
     \big(   
   1,
   \left(
    \frac{2}{s_{k+1}-(s+(T-s)r)}
    \right)^{\frac{p}{2}}
    \frac{\Gamma\!\left(\frac{p+1}{2}\right)}{\sqrt{\pi}}v
     \big)\big\rangle }
     {[\rhoh(s+(T-s)r,s_{k+1})]^{p-1}}
     \\
     &\qquad 
     \cdot
     \int_{s_{k+1}}^T 
     \tfrac{
 \big\langle \uvec_{\nu_{k+1}},     
     \big(   
   1,
   \left(
    \frac{2}{s_{k+2}-s_{k+1}}
    \right)^{\frac{p}{2}}
    \frac{\Gamma\!\left(\frac{p+1}{2}\right)}{\sqrt{\pi}}v
     \big)\big\rangle}
     {[\rhoh(s_{k+1},s_{k+2})]^{p-1}}
     \ldots
      \int_{s_j}^T 
     \tfrac{
 \big\langle \uvec_{\nu_{j}},     
     \big(   
   1,
   \left(
    \frac{2}{s_{j+1}-s_j}
    \right)^{\frac{p}{2}}
    \frac{\Gamma\!\left(\frac{p+1}{2}\right)}{\sqrt{\pi}}v
     \big)\big\rangle}
     {[\rhoh(s_j,s_{j+1})]^{p-1}}
     \,ds_{j+1}
     \ldots
     ds_{k+2}
     \,ds_{k+1}
     \,dr
     \Bigg|_{s=\ES(k-1,t)}
     \\
     &
     =
     \int_{\ES(k-1,t)}^T 
     \tfrac{
 \big\langle \uvec_{\nu_{k-1}},     
     \big(   
   1,
   \left(
    \frac{2}{s_k-\ES(k-1,t)}
    \right)^{\frac{p}{2}}
    \frac{\Gamma\!\left(\frac{p+1}{2}\right)}{\sqrt{\pi}}v
     \big)\big\rangle}
     {[\rhoh(\ES(k-1,t),s_k)]^{p-1}}
     \int_{s_k}^T 
     \tfrac{
    \big\langle \uvec_{\nu_{k}},  
     \big(   
   1,
   \left(
    \frac{2}{s_{k+1}-s_k}
    \right)^{\frac{p}{2}}
    \frac{\Gamma\!\left(\frac{p+1}{2}\right)}{\sqrt{\pi}}v
     \big)\big\rangle}
     {[\rhoh(s_k,s_{k+1})]^{p-1}}
      \\
      &
      \qquad 
    \cdot 
     \int_{s_{k+1}}^T 
     \tfrac{
 \big\langle \uvec_{\nu_{k+1}},     
     \big(   
   1,
   \left(
    \frac{2}{s_{k+2}-s_{k+1}}
    \right)^{\frac{p}{2}}
    \frac{\Gamma\!\left(\frac{p+1}{2}\right)}{\sqrt{\pi}}v
     \big)\big\rangle}
     {[\rhoh(s_{k+1},s_{k+2})]^{p-1}}
     \ldots
      \int_{s_j}^T 
     \tfrac{
 \big\langle \uvec_{\nu_{j}},     
     \big(   
   1,
   \left(
    \frac{2}{s_{j+1}-s_j}
    \right)^{\frac{p}{2}}
    \frac{\Gamma\!\left(\frac{p+1}{2}\right)}{\sqrt{\pi}}v
     \big)\big\rangle}
     {[\rhoh(s_j,s_{j+1})]^{p-1}}
     \,ds_{j+1}
     \ldots
     ds_{k+2}
     \,ds_{k+1}
     \,ds_k
 .
  \end{split}
\end{equation}
This completes the induction step. Induction hence proves \eqref{eq:ub.it.int.lp2}.

Next \eqref{eq:ub.it.int.lp2} implies that
\begin{equation}\label{eq:ub.it.int.lp3}
\begin{split}
&
\E\! \left[
\prod_{i=1}^{j+1}
\left|
   \tfrac{1}{\rhoh(\ES(i-1,t),\ES(i,t))}
    \big\langle \uvec_{\nu_{i-1}},
     \big(   
   1,
     \tfrac{ 
     W_{\ES(i,t)}- W_{\ES(i-1,t)}
     }{\ES(i,t)-\ES(i-1,t) } 
     \big)\big\rangle
     \right|^p
     \right]
      \le 
      \left[
      \max\left\{
      (T-t)^{\frac{p}{2}}, \tfrac{2^{\frac{p}{2}}\Gamma\!\left(\frac{p+1}{2}\right)}{\sqrt{\pi}}
      \right\}
      \right]^{j+1}
      \\
      &\quad \cdot
     \int_t^T 
     \frac{1}
     {(s_1-t)^{\frac{p}{2}}[\rhoh(t,s_1)]^{p-1}}
     \int_{s_1}^T 
     \frac{1}
     {(s_2-s_1)^{\frac{p}{2}}[\rhoh(s_1,s_2)]^{p-1}}
     \ldots
      \int_{s_j}^T 
     \frac{1}
     {(s_{j+1}-s_j)^{\frac{p}{2}}[\rhoh(s_j,s_{j+1})]^{p-1}}
     \,ds_{j+1}
     \ldots
     ds_2 
     \,ds_1
 .
  \end{split}
\end{equation}
Inequality \eqref{eq:ub.it.int.lp3}, Corollary~\ref{cor:iterated_int_beta_bound}
(applied with $\beta=\frac{p}{2}$ and $\gamma=p-1$ in the notation
of Corollary~\ref{cor:iterated_int_beta_bound})
together with $\frac{p}{2}\in [\alpha(p-1),\alpha(p-1)+1]$,
and the fact that 
$(\frac{p}{2}-\alpha (p-1))(1-(\frac{p}{2}-\alpha(p-1)))\le \frac{1}{4}$
%\eqref{eq:iterated_int_beta_bound2} 
show that 
%for all $j\in \N_0$, $t\in [0,T)$, $\nu_0, \ldots, \nu_j\in \{1,\ldots, d+1\}$ 
\begin{equation}\label{eq:ub.it.int.lp5}
\begin{split}
&\left\|
\prod_{i=1}^{j+1}
   \tfrac{1}{\rhoh(\ES(i-1,t),\ES(i,t))}
    \big\langle \uvec_{\nu_{i-1}},
     \big(   
   1,
     \tfrac{ 
     W_{\ES(i,t)}- W_{\ES(i-1,t)}
     }{\ES(i,t)-\ES(i-1,t) } 
     \big)\big\rangle
   \right\|_{L^p(\P;\R)}^p=
\E\! \left[
\prod_{i=1}^{j+1}
\left|
   \tfrac{1}{\rhoh(\ES(i-1,t),\ES(i,t))}
    \big\langle \uvec_{\nu_{i-1}},
     \big(   
   1,
     \tfrac{ 
     W_{\ES(i,t)}- W_{\ES(i-1,t)}
     }{\ES(i,t)-\ES(i-1,t) } 
     \big)\big\rangle
     \right|^p
     \right]
     \\
      &\le 
      \left[
      \max\left\{
      (T-t)^{\frac{p}{2}}, \tfrac{2^{\frac{p}{2}}\Gamma\!\left(\frac{p+1}{2}\right)}{\sqrt{\pi}}
      \right\}
      \right]^{j+1}
    \left[\tfrac{(T-t)^{\frac{p}{2}}\Gamma(\alpha(p-1)-\frac{p}{2}+1)}{(1-\alpha)^{p-1}
(\frac{p}{2})^{\alpha(p-1)-\frac{p}{2}+1} }\right]^{j+1}
\left[e^{\frac{p}{2}}\left(\tfrac{pj}{2}+1\right) \right]^{\frac{2(\frac{p}{2}-\alpha (p-1))(\alpha(p-1)-\frac{p}{2}+1)}{p}}
\\
&\qquad \cdot
\left[
\tfrac{\Gamma\left(\frac{2}{p}\right)}{\Gamma\left(1+j+\frac{2}{p}\right)}
\right]^{\alpha (p-1) -\frac{p}{2} +1}  
\\
      &\leq
      \left[
      \max\left\{
      (T-t)^{\frac{p}{2}}, \tfrac{2^{\frac{p}{2}}\Gamma\!\left(\frac{p+1}{2}\right)}{\sqrt{\pi}}
      \right\}
    \tfrac{(T-t)^{\frac{p}{2}}\Gamma(\alpha(p-1)-\frac{p}{2}+1)}{(1-\alpha)^{p-1}
(\frac{p}{2})^{\alpha(p-1)-\frac{p}{2}+1} }\right]^{j+1}
\left[e^{\frac{p}{2}}\left(\tfrac{pj}{2}+1\right) \right]^{\frac{1}{2p}}
\left[
\tfrac{\Gamma\left(\frac{2}{p}\right)}{\Gamma\left(1+j+\frac{2}{p}\right)}
\right]^{\alpha (p-1) -\frac{p}{2} +1}.
%\\
%      &\le 
%      \left[
%      \max\left\{
%      T-t, 2\tfrac{\Gamma\!\left(\frac{p+1}{2}\right)^{\frac{2}{p}}}{\pi^{\frac{1}{p}}}
%      \right\}
%  \tfrac{(T-t)\Gamma(\frac{p}{2})^{\frac{2}{p}}}{(1-\alpha)^{\frac{2(p-1)}{p}}
%(\frac{p}{2})^{\frac{2(\alpha(p-1)-\frac{p}{2}+1)}{p}} }\right]^{\frac{p(j+1)}{2}}
%\left[e^{\frac{p}{2}}\left(\tfrac{pj}{2}+1\right) \right]^{\frac{1}{2p}}
%\left[
%\tfrac{\Gamma\left(\frac{2}{p}\right)}{\Gamma\left(1+j+\frac{2}{p}\right)}
%\right]^{p\left(\frac{\alpha}{2}-\frac{(1-\alpha)(p-2)}{2p} \right)}
  \end{split}
\end{equation}
This together with
$\alpha(p-1)-\frac{p}{2}+1\le p-1-\frac{p}{2}+1=\frac{p}{2}$
implies \eqref{eq:ub.it.int.lp.b}. 
The proof of Corollary~\ref{cor:ub.it.int.lp} is thus completed.
\end{proof}

\section[Full-history recursive multilevel Picard (MLP) approximation methods]{Full-history recursive multilevel Picard approximation methods}

\label{sec:error_analysis}

In this section we introduce and analyze a class of new MLP approximation methods for nonlinear heat equations with gradient-dependent nonlinearities. In the main result of this section, Proposition~\ref{thm:rate} in Subsection~\ref{subsec:error_analysis} below, we provide a detailed error analysis for these new MLP approximation methods.
We will employ Proposition~\ref{thm:rate} in our proofs of the approximation results in Section~\ref{sec:rate} below (cf. Corollary~\ref{cor:comp_and_error}, Theorem~\ref{cor:comp_and_error2}, and Corollary~\ref{cor:comp_and_error3} in Section~\ref{sec:rate} below).

%The results of Section~\ref{sec:it_int} are then used in Section~\ref{sec:error_analysis} in which we introduce and analyze the considered MLP approximation methods.
%
%
%This section provides an error analysis for the approximations~\eqref{eq:def:U_intro}. In fact, we analyze a generalized version of \eqref{eq:def:U_intro}, where the random variables $\unif^\theta$, $\theta \in \Theta$, are allowed to have a more general distribution than in the formulation of Theorem~\ref{thm:introduction}. Moreover, the analysis does not require existence of a classical solution of the PDE~\eqref{eq:PDE_intro}, but only existence of a mild solution satisfying the stochastic fixed point equation~\eqref{eq:feynmankacuinfty}.
%
%Setting~\ref{s:full.discretization} fixes the framework, Subsection~\ref{subsec:basic_prop} collects some basic properties of the approximations~\eqref{eq:def:U}, and Subsection~\ref{subsec:error_analysis} contains Proposition~\ref{thm:rate}, the main result of this section.

\subsection{Description of MLP approximations}

\begin{setting}\label{s:full.discretization}
%For every
%topological space $(E,\mathcal E)$ let $\mathcal{B}(E)$
%be the Borel-sigma-algebra on $(E,\mathcal E)$,
Let $\left\|\cdot\right\|_1\colon (\cup_{n\in \N}\R^n)\to \R$ satisfy for 
all $n\in \N$, $x=(x_1,x_2,\ldots,x_n)\in \R^n$ that
$\|x\|_1=\sum_{i=1}^n |x_i|$,
let 
$ T \in (0,\infty) $, 
$ d \in \N $, 
$
  \Theta = \cup_{ n \in \N } \Z^n
$,
$L=(L_1,L_2,\ldots,L_{d+1})\in [0,\infty)^{d+1}$,
$K=(K_1,K_2,\ldots, K_d)\in [0,\infty)^d$,
$\uvec_1=(1,0,\ldots,0)$, $\uvec_2=(0,1,0,\ldots,0)$, \ldots, 
$\uvec_{d+1}=(0,0,\ldots,1)\in \R^{d+1}$,
$\rho \in C((0,1),(0,\infty))$,
${\bf u}=({\bf u}_1, {\bf u}_2, \ldots, {\bf u}_{d+1}) \in C([0,T)\times\R^d,\R^{1+d})$,
let $\langle \cdot, \cdot \rangle\colon \R^{d+1}\times \R^{d+1} \to \R$ 
satisfy for all $v=(v_1,v_2,\ldots, v_{d+1})$, $w=(w_1,w_2,\ldots, w_{d+1})\in \R^{d+1}$ that $\langle v, w\rangle=\sum_{i=1}^{d+1}v_iw_i$,
let $\rhoh \colon[0,T]^2\to\R$
satisfy for all $t\in [0,T)$, $s\in (t,T)$ that
$\rhoh(t,s)=\frac{1}{T-t}\rho(\frac{s-t}{T-t})$,
let
$
  ( 
    \Omega, \mathcal{F}, \P
  )
$
be a probability space, 
%with a normal filtration $(\mathbb{F}_t)_{t\in[t,T]}$,
% which satisfies the usual conditions,
let
$
  W^{ \theta }=(W^{\theta,1}, W^{\theta,2}, \ldots, W^{\theta,d}) \colon [0,T] \times \Omega \to \R^d 
$, 
$ \theta \in \Theta $,
be standard
%$(\mathbb{F}_t)_{t\in[0,T]}$
Brownian motions
with continuous sample paths,
let $\unif^\theta\colon \Omega\to(0,1)$, $\theta\in \Theta$, be 
i.i.d.\ random variables,
assume
for all $b\in (0,1)$
that $\P(\unif^0\le b)=\int_0^b \rho(s)\,ds$, 
assume that
$(W^\theta)_{\theta \in \Theta}$ and $(\unif^\theta)_{\theta \in \Theta}$
are independent,
let $\uniform^{(n)}\colon [0,T)\times \Omega\to [0,T)$, $n \in \N_0$, and $\ES\colon \N_0 \times [0,T)\times \Omega \to [0,T)$
 satisfy for all $n \in \N_0$, $t\in [0,T)$
that $\uniform^{(n)} _t = t+ (T-t)\unif^{(n)}$, $\ES(0,t)=t$, and
$\ES(n+1,t)=\uniform^{(n)}_{\ES(n,t)}$,
let
$ f\in C([0,T]\times\R^d\times\R^{1+d},\R)$, 
$ 
  g\in C(\R^d, \R)
$, and
$
  \funcF \colon 
  C([0,T)\times\R^d,\R^{1+d})
  \to
  C([0,T)\times\R^d,\R)
$
 satisfy for all $t\in[0,T)$, $x=(x_1,x_2,\ldots,x_d)$, $\mathfrak{x}=(\mathfrak{x}_1, \mathfrak{x}_2,\ldots,\mathfrak{x}_d)\in\R^d$,
$u=(u_1,u_2,\ldots,u_{d+1})$, $\mathfrak{u}=(\mathfrak{u}_1,\mathfrak{u}_2,\ldots,\mathfrak{u}_{d+1})\in\R^{1+d}$, ${\bf v}\in C([0,T)\times\R^d,\R^{1+d})$
that
\begin{equation}  \begin{split}\label{eq:fLipschitz}
 \max\{|f(t,x,u)-f(t,x,\mathfrak{u})|,|g(x)-g(\mathfrak{x})|\}\leq
    \left[\textstyle \sum\limits _{\nu=1}^{d+1} \displaystyle
    \LipConst_\nu \left|u_\nu-\mathfrak{u}_\nu\right|\right]
    +\left[\textstyle\sum\limits _{\nu=1}^d \displaystyle K_\nu |x_\nu-\mathfrak{x}_\nu|\right],
\end{split}     \end{equation}
\begin{equation}\label{eq:feynmankacintegrability}
\E\!\left[\big\|g(x+W^0_{T}-W^0_t)\big(1,\tfrac{W^0_T-W^0_t}{T-t} \big)
          \big\|_1
      +
      \int_t^{T}\big\|[(\funcF({\bf u}))(t,x+W^0_{s}-W^0_{t})]
      \big(1,\tfrac{W^0_s-W^0_t}{s-t} \big)\big\|_1
      \,ds
      \right]<\infty,
\end{equation}
\begin{equation}\label{eq:feynmankacuinfty}
{\bf u}(t,x)=\E\!\left[g(x+W^0_{T}-W^0_t)\big(1,\tfrac{W^0_T-W^0_t}{T-t} \big)+
      \int_t^{T}[(\funcF({\bf u}))(t,x+W^0_{s}-W^0_{t})]
      \big(1,\tfrac{W^0_s-W^0_t}{s-t} \big)
      \,ds
      \right],
\end{equation}
and $(F({\bf v}))(t,x)=f(t,x,{\bf v}(t,x))$,
and
let
$ 
  {\bf U}_{ n,M}^{\theta }=({\bf U}_{ n,M}^{\theta,1},
{\bf U}_{ n,M}^{\theta,2}, \ldots, {\bf U}_{ n,M}^{\theta,d+1})  
  \colon[0,T)\times\R^d\times\Omega\to\R^{1+d}
$,
$n,M\in\Z$, $\theta\in\Theta$, satisfy
for all 
$
  n,M \in \N
$,
$ \theta \in \Theta $,
$ t\in [0,T)$,
$x \in \R^d$
that $
{\bf U}_{-1,M}^{\theta}(t,x)={\bf U}_{0,M}^{\theta}(t,x)=0$ and
\begin{equation}  \begin{split}\label{eq:def:U}
  &{\bf U}_{n,M}^{\theta}(t,x)
  =
  \big(
    g(x)
    , 0
  \big)
  +
  \sum_{i=1}^{M^n}\frac{\big(g(x+W^{(\theta,0,-i)}_T-W^{(\theta,0,-i)}_t)-g(x)\big)}{M^n}
  \Big(
  1 ,
  \tfrac{ 
  W^{(\theta, 0, -i)}_{T}- W^{(\theta, 0, -i)}_{t}
  }{ T - t }
  \Big)
  \\
  &+\sum_{l=0}^{n-1}\sum_{i=1}^{M^{n-l}}
  \frac{
  \big(\funcF({\bf U}_{l,M}^{(\theta,l,i)})-\1_{\N}(l)\funcF( {\bf U}_{l-1,M}^{(\theta,-l,i)})\big)
  (\uniform^{(\theta, l,i)}_t,x+W_{\uniform^{(\theta, l,i)}_t}^{(\theta,l,i)}-W_t^{(\theta,l,i)})
  }
  {M^{n-l}\rhoh(t,\uniform^{(\theta, l,i)}_t)}
  \left(
  1 ,
  \tfrac{ 
  W_{\uniform^{(\theta, l,i)}_t}^{(\theta,l,i)}- W^{(\theta, l, i)}_{t}
  }{ \uniform^{(\theta, l,i)}_t-t}
  \right).
\end{split}     \end{equation}
\end{setting}

\subsection{Properties of MLP approximations}\label{subsec:basic_prop}
\begin{lemma}[Measurability properties]\label{properties_approx}
 Assume Setting~\ref{s:full.discretization}
 and
 let $M\in\N$.
 Then
  \begin{enumerate}[(i)]
  \item  \label{properties_approx:item1}
  for all $n \in \N_0$, $\theta\in\Theta $ it holds that
  $
    {\bf U}_{ n,M}^{\theta } \colon [0, T) \times \R^d \times \Omega \to \R
  $
  is a continuous random field,
  
  \item  \label{properties_approx:item2}
  for all $n \in \N_0$, $\theta \in \Theta$ it holds that
  $
    \sigma( {\bf U}^\theta_{n, M} )
  \subseteq
    \sigma( (\unif^{(\theta, \vartheta)})_{\vartheta \in \Theta}, (W^{(\theta, \vartheta)})_{\vartheta \in \Theta}) 
  $,

  \item  \label{properties_approx:item4}
  for all $n, m \in \N_0$, $i,j,k,l, \in \Z$, $\theta \in \Theta$ with $(i,j) \neq (k,l)$ 
  it holds that
  $
    {\bf U}^{(\theta,i,j)}_{n,M}
  $
  and
  $
    {\bf U}^{(\theta,k,l)}_{m,M}
  $
  are independent,
  \item  \label{properties_approx:item3}
  for all $n \in \N_0$, $\theta\in\Theta $ it holds hat
  ${\bf U}_{ n,M}^{\theta }$, $W^\theta$, and $\unif^\theta$ are independent,
  \item  \label{properties_approx:item5}
  for all $n \in \N_0$ it holds that 
  $
    {\bf U}^\theta_{n, M}
  $,
  $
   \theta \in \Theta
  $,
  are identically distributed, and
  \item  \label{properties_approx:item6}
  for all $\theta \in \Theta$, $l\in \N$, $i\in \N$, $t\in [0,T)$, $x\in \R^d$ it holds that
\begin{equation}
\tfrac{
  \funcF( {\bf U}_{l-1,M}^{(\theta,-l,i)})
  (\uniform^{(\theta, l,i)}_t,x+W_{\uniform^{(\theta, l,i)}_t}^{(\theta,l,i)}-W_t^{(\theta,l,i)})
  }
  {\rhoh(t,\uniform^{(\theta, l,i)}_t)}
  \left(
  1 ,
  \tfrac{ 
  W_{\uniform^{(\theta, l,i)}_t}^{(\theta,l,i)}- W^{(\theta, l, i)}_{t}
  }{ \uniform^{(\theta, l,i)}_t-t}
  \right) 
  \end{equation}
  and
  \begin{equation}
  \tfrac{
  \funcF( {\bf U}_{l-1,M}^{(\theta,l,i)})
  (\uniform^{(\theta, l,i)}_t,x+W_{\uniform^{(\theta, l,i)}_t}^{(\theta,l,i)}-W_t^{(\theta,l,i)})
  }
  {\rhoh(t,\uniform^{(\theta, l,i)}_t)}
  \left(
  1 ,
  \tfrac{ 
  W_{\uniform^{(\theta, l,i)}_t}^{(\theta,l,i)}- W^{(\theta, l, i)}_{t}
  }{ \uniform^{(\theta, l,i)}_t-t}
  \right)
\end{equation} 
are identically distributed.
  \end{enumerate}
\end{lemma}
\begin{proof}[Proof of Lemma~\ref{properties_approx}]
First, observe that \eqref{eq:def:U}, the hypothesis that 
for all $M \in \N$, $\theta\in\Theta $ it holds that
${\bf U}^\theta_{0, M} = 0$,
the fact  that
for all $\theta \in \Theta$ it holds that $W^\theta$ and $\uniform^\theta$ are continuous random fields,
the hypothesis that $f\in C([0,T]\times \R^d \times \R \times \R^d, \R)$,
the hypothesis that $g\in C(\R^d, \R)$,
the fact that
$\rhoh |_{\{(s,t)\in [0,T)^2\colon s<t\}}\in C( \{(s,t)\in [0,T)^2\colon s<t\}, \R) $,
 and induction on $\N_0$ establish Item~\eqref{properties_approx:item1}.
Next note that Item~\eqref{properties_approx:item1},
the hypothesis that $f\in C([0,T]\times \R^d \times \R \times \R^d, \R)$,
and, e.g., Beck et al.~\cite[Lemma 2.4]{Becketal2018}
assure that 
for all $n \in \N_0$, $\theta\in\Theta $ it holds that 
$F({\bf U}^\theta_{n, M})$ is $ ( \mathcal{B}((0, T) \times \R^d ) \otimes \sigma({\bf U}^\theta_{n, M}) )/ \mathcal{B}(\R )$-measurable.
The hypothesis that 
for all $M \in \N$, $\theta\in\Theta $ it holds that
${\bf U}^\theta_{0, M} = 0$,
\eqref{eq:def:U}, 
the fact that 
for all $\theta \in \Theta$ it holds that 
$W^\theta$ is $ ( \mathcal{B}([0, T]) \otimes \sigma(W^\theta) )/ \mathcal{B}(\R )$-measurable,
the fact that 
for all $\theta \in \Theta$ it holds that 
$\uniform^\theta$ is $ ( \mathcal{B}([0, T)) \otimes \sigma(\unif^\theta) )/ \mathcal{B}([0, T) )$-measurable,
and induction on $\N_0$ hence prove Item~\eqref{properties_approx:item2}.
In addition, note that Item~\eqref{properties_approx:item2} and the fact that
for all $i,j,k,l, \in \Z$, $\theta \in \Theta$ with $(i,j) \neq (k,l)$ it holds that 
$((\unif^{(\theta,i,j, \vartheta)}, W^{(\theta,i,j, \vartheta)}))_{\vartheta \in \Theta }$
and
$((\unif^{(\theta,k,l, \vartheta)}, W^{(\theta,k,l, \vartheta)}))_{\vartheta \in \Theta }$
are independent prove Item~\eqref{properties_approx:item4}.
Furthermore, observe that Item~\eqref{properties_approx:item2} and the fact that 
for all $\theta \in \Theta$ it holds that
$(\unif^{(\theta, \vartheta)})_{\vartheta \in \Theta}$, $(W^{(\theta, \vartheta)})_{\vartheta \in \Theta}$,
$W^\theta$, and $\unif^\theta$ are independent establish Item~\eqref{properties_approx:item3}.
Next observe that the hypothesis that 
for all $\theta\in\Theta $ it holds that
${\bf U}^\theta_{0, M} = 0$,
the hypothesis that $(W^\theta)_{\theta \in \Theta}$ are i.i.d., 
the hypothesis that $(\uniform^\theta)_{\theta \in \Theta}$ are i.i.d., 
Items~\eqref{properties_approx:item1}--\eqref{properties_approx:item4},
 Hutzenthaler et al.~\cite[Corollary~2.5]{HJKNW2018},
and induction on $\N_0$ establish Item~\eqref{properties_approx:item5}.
Furthermore, observe that Item~\eqref{properties_approx:item2} and the fact that 
for all $\theta \in \Theta$, $l\in \N$, $i\in \N$ it holds that
$(\unif^{(\theta,-l,i, \vartheta)})_{\vartheta \in \Theta}$, $(W^{(\theta,-l,i, \vartheta)})_{\vartheta \in \Theta}$,
$W^{\theta,l,i}$, and $\unif^{\theta,l,i}$ are independent,
and,
%disintegration of a continuous random field
%and an independent random argument (see, 
e.g., Hutzenthaler et al.~\cite[Lemma 2.3]{HJKNW2018}
imply that for every $\theta \in \Theta$, $l,i\in \N$, $t\in [0,T)$, $x\in \R^d$ and every
bounded $\mathcal B(\R^{d+1})/\mathcal B(\R)$-measurable $\psi \colon \R^{d+1}\to \R$ it holds that
\begin{equation}
\begin{split}
&\E\!\left[
\psi\!\left(
\tfrac{
  \funcF( {\bf U}_{l-1,M}^{(\theta,-l,i)})
  (\uniform^{(\theta, l,i)}_t,x+W_{\uniform^{(\theta, l,i)}_t}^{(\theta,l,i)}-W_t^{(\theta,l,i)})
  }
  {\rhoh(t,\uniform^{(\theta, l,i)}_t)}
  \left(
  1 ,
  \tfrac{ 
  W_{\uniform^{(\theta, l,i)}_t}^{(\theta,l,i)}- W^{(\theta, l, i)}_{t}
  }{ \uniform^{(\theta, l,i)}_t-t}
  \right)
  \right)
  \right]\\
  &=\E\!\left[
  \E\!\left[
\psi\!\left(
\tfrac{
  \funcF( {\bf U}_{l-1,M}^{(\theta,-l,i)})
  (r,x+z)
  }
  {\rhoh(t,r)}
  \left(
  1 ,
  \tfrac{ 
 z
  }{r-t}
  \right)
  \right)
  \right]
  \Bigg|_{r=\uniform^{(\theta, l,i)}_t, z=W_{\uniform^{(\theta, l,i)}_t}^{(\theta,l,i)}- W^{(\theta, l, i)}_{t}}
  \right].
  \end{split}
  \end{equation}
This,  Item~\eqref{properties_approx:item5}, Item~\eqref{properties_approx:item2}, and the fact that 
for all $\theta \in \Theta$, $l, i\in \N$ it holds that
$(\unif^{(\theta,l,i, \vartheta)})_{\vartheta \in \Theta}$, $(W^{(\theta,l,i, \vartheta)})_{\vartheta \in \Theta})$,
$W^{\theta,l,i}$, and $\unif^{\theta,l,i}$ are independent,
and, e.g.,
%disintegration of a continuous random field
%and an independent random argument (see, e.g., 
Hutzenthaler et al.~\cite[Lemma 2.3]{HJKNW2018}
imply that for every $\theta \in \Theta$, $l,i\in \N$, $t\in [0,T)$, $x\in \R^d$ and every bounded $\mathcal B(\R^{d+1})/\mathcal B(\R)$-measurable $\psi \colon \R^{d+1}\to \R$ it holds that
\begin{equation}
\begin{split}
&\E\!\left[
\psi\!\left(
\tfrac{
  \funcF( {\bf U}_{l-1,M}^{(\theta,-l,i)})
  (\uniform^{(\theta, l,i)}_t,x+W_{\uniform^{(\theta, l,i)}_t}^{(\theta,l,i)}-W_t^{(\theta,l,i)})
  }
  {\rhoh(t,\uniform^{(\theta, l,i)}_t)}
  \left(
  1 ,
  \tfrac{ 
  W_{\uniform^{(\theta, l,i)}_t}^{(\theta,l,i)}- W^{(\theta, l, i)}_{t}
  }{ \uniform^{(\theta, l,i)}_t-t}
  \right)
  \right)
  \right]\\
  &=\E\!\left[
  \E\!\left[
\psi\!\left(
\tfrac{
  \funcF( {\bf U}_{l-1,M}^{(\theta,l,i)})
  (r,x+z)
  }
  {\rhoh(t,r)}
  \left(
  1 ,
  \tfrac{ 
 z
  }{r-t}
  \right)
  \right)
  \right]
  \Bigg|_{r=\uniform^{(\theta, l,i)}_t, z=W_{\uniform^{(\theta, l,i)}_t}^{(\theta,l,i)}- W^{(\theta, l, i)}_{t}}
  \right]\\
  &=\E\!\left[
\psi\!\left(
\tfrac{
  \funcF( {\bf U}_{l-1,M}^{(\theta,l,i)})
  (\uniform^{(\theta, l,i)}_t,x+W_{\uniform^{(\theta, l,i)}_t}^{(\theta,l,i)}-W_t^{(\theta,l,i)})
  }
  {\rhoh(t,\uniform^{(\theta, l,i)}_t)}
  \left(
  1 ,
  \tfrac{ 
  W_{\uniform^{(\theta, l,i)}_t}^{(\theta,l,i)}- W^{(\theta, l, i)}_{t}
  }{ \uniform^{(\theta, l,i)}_t-t}
  \right)
  \right)
  \right].
  \end{split}
  \end{equation}
  This establishes Item~\eqref{properties_approx:item6}.
The proof of Lemma~\ref{properties_approx} is thus completed.
\end{proof}

\begin{lemma}[Approximations are integrable]\label{lem:b38}
Assume Setting \ref{s:full.discretization},
let $p\in (1,\infty)$,
$M \in \N$, $x \in \R^d$,
and assume 
for all $q\in [1,p)$, $t\in[0,T)$ that
\begin{equation}\label{eq:int_cond_int}
\int_0^1
  \frac{1}{
  s^{\frac{q}{2}}
\left[\rho(s)\right]^{q-1}  
}
  \,ds+
  \sup_{s\in [t,T)}
  \E\!\left[
  \left|
  \big(\funcF(0)\big)
  (s,x+W_{s}^0-W_t^0)
  \right|^{q}
  \right]<\infty.
\end{equation}
Then 
\begin{enumerate}[(i)]
\item\label{item:approximations.integrable.i} 
it holds
for all
$\theta\in\Theta$, $q\in [1,\infty)$,
$\nu \in \{1,2,\ldots, d+1\}$ that
 \begin{equation}  \begin{split}
    \sup_{u\in(0,T]}\sup_{t\in[0,u)}\sup_{y\in\R^d}
    \E\!\left[\left|\left(g(y+W^{\theta}_u-W^{\theta}_t)-g(y)\right)
    \big\langle
    \uvec_\nu,
  \big(
  1 ,
  \tfrac{ 
  W^{\theta}_{u}- W^{\theta}_{t}
  }{ u - t }
  \big)
 \big \rangle
  \right|^q\right]<\infty,
  \end{split}     \end{equation}
\item\label{item:approximations.integrable.ii} 
it holds
for all
$n\in\N_0$, $\theta\in\Theta$, $q\in [1,p)$, $t\in [0,T)$, $\nu \in \{1,2,\ldots, d+1\}$ that
\begin{equation}\label{eq:approx.int.ii}
\begin{split}
\sup_{s\in[t,T)}
\left\{
\E\!\left[ \left| 
    ({\bf U}_{n,M}^{\theta}(s,x+W_{s}^\theta-W_{t}^\theta))_\nu
  \right|^q \right] 
  +
\E\!\left[
  \tfrac{
  \left|
  \big(\funcF({\bf U}_{n,M}^{\theta})\big)
  (\uniform^{\theta}_s,x+W_{\uniform^{\theta}_s}^{\theta}-W_t^{\theta})
  \right|^q
  }
  {\left[\rhoh(s,\uniform^{\theta}_s)\right]^q}
  \left|
  \big\langle \uvec_\nu,
  \big(
  1 ,
  \tfrac{ 
  W_{\uniform^{\theta }_s}^{\theta}- W^{\theta}_{s}
  }{ \uniform^{\theta}_s-s}
  \big)
  \big \rangle
  \right|^q
  \right]
  \right\}
  <\infty,
\end{split}     
\end{equation}
and
\item\label{item:approximations.integrable.iii} 
it holds for all  $n\in\N$, $\theta \in \Theta$, $t\in[0,T)$, that
  \begin{equation}  \begin{split}\label{eq:discreteFeynmanKac}
    \E\!\left[{\bf U}_{n,M}^{\theta}(t,x)\right]&=\E\!\left[
    g(x+W_T^{\theta}-W_t^{\theta})
     \Big(
     1 ,
     \tfrac{ 
     W^{\theta}_{T}- W^{\theta}_{t}
     }{ T - t }
     \Big)
    \right]
    \\&\qquad
    +
    \E\!\left[
    \left(\funcF( {\bf U}_{n-1,M,Q}^{\theta})\right)
    \!(\uniform^{\theta}_t,x+W_{\uniform^{\theta}_t}^{\theta}-W_t^{\theta})
     \Big(
     1 ,
     \tfrac{ 
     W^{\theta}_{\uniform^{\theta}_t}- W^{\theta}_{t}
     }{ {\uniform^{\theta}_t - t} }
     \Big)
    \right].
  \end{split}     \end{equation}
\end{enumerate}
\end{lemma}
\begin{proof}[Proof of Lemma~\ref{lem:b38}]
%First, we prove Item~\eqref{item:approximations.integrable.i}.
The Cauchy-Schwarz inequality,
the Lipschitz property~\eqref{eq:fLipschitz} of $g$,
Jensen's inequality,
and
the scaling property of Brownian motion
yield
for all
$\theta\in\Theta$, $q\in [1,\infty)$,
$\nu \in \{1,2,\ldots, d+1\}$ that
  \begin{equation}  \begin{split}
    &\sup_{u\in(0,T]}\sup_{t\in[0,u)}\sup_{y\in\R^d}
    \left(\E\!\left[\left|\left(g(y+W^{\theta}_u-W^{\theta}_t)-g(y)\right)
    \big\langle \uvec_\nu,
  \big(
  1 ,
  \tfrac{ 
  W^{\theta}_{u}- W^{\theta}_{t}
  }{ u - t }
  \big) \big \rangle
  \right|^q\right]\right)^2
  \\&
  \leq
    \sup_{u\in(0,T]}\sup_{t\in[0,u)}\sup_{y\in\R^d}
    \left(
    \E\!\left[\left|g(y+W^{\theta}_u-W^{\theta}_t)-g(y)\right|^{2q}
    \right]
    \E\!\left[\left|
    \big\langle \uvec_\nu,
  \big(
  1 ,
  \tfrac{ 
  W^{\theta}_{u}- W^{\theta}_{t}
  }{ u - t }
  \big)\rangle  \right|^{2q}\right]
  \right)
  \\&
  \leq
    \sup_{u\in(0,T]}\sup_{t\in[0,u)}
    \left(
    \E\!\left[\bigg|\sum_{i=1}^d K_i|W_u^{\theta,i}-W_t^{\theta,i}|\bigg|^{2q}
    \right]
    \E\!\left[
    1+
  \left|
  \tfrac{ 
  W^{\theta,1}_{u}- W^{\theta,1}_{t}
  }{ u - t }
  \right|^{2q}\right]
  \right)
  \\&
  \leq
    \sup_{u\in(0,T]}\sup_{t\in[0,u)}
    \left(
   d^{2q-1}\left(\smallsum_{i=1}^{ d}(K_i)^{2q}\right)
   \tfrac{(u-t)^q}{ T^{q}} \E[|W_T^{0,1}|^{2q}]
   \Big(
    1+
    \tfrac{(u-t)^{q}}{(u-t)^{2q}T^q} \E[|W_T^{0,1}|^{2q}]
    \Big)
  \right)
  \\&
  =
   d^{2q-1}\left(\smallsum_{i=1}^{ d}(K_i)^{2q}\right)
   \E[|W_T^{0,1}|^{2q}]
   \Big(
    1+
    \tfrac{1}{T^{2q}} \E[|W_T^{0,1}|^{2q}]
    \Big)
  <\infty.
  \end{split}     \end{equation}
This proves Item~\eqref{item:approximations.integrable.i}.
%Next, we prove Item~\eqref{item:approximations.integrable.ii}.
Next observe that
the fact that $W^{\theta}$, $\uniform^{\theta}$, and $U^{\theta}$ are
independent and continuous random fields
(see Lemma~\ref{properties_approx}),
%disintegration of a continuous random field
%and an independent random argument (see, e.g., 
Hutzenthaler et al.~\cite[Lemma 2.3]{HJKNW2018}),
H\"older's inequality 
(applied with the conjugate numbers
$\frac{p+q}{2q},\frac{p+q}{p-q}\in (1,\infty)$),
and
the scaling property of Brownian motion
demonstrate that
for all $t\in[0,T)$, $s\in[t,T)$, $n \in \N_0$, $\theta\in\Theta$, $q\in [1,p)$,
$\nu\in \{1,2,\ldots,d+1\}$ it holds that 
\begin{equation}
\begin{split}\label{eq:prepa}
  &
\E\!\left[
  \tfrac{
  \left|
  \big(\funcF({\bf U}_{n,M}^{\theta})\big)
  (\uniform^{\theta}_s,x+W_{\uniform^{\theta}_s}^{\theta}-W_t^{\theta})
  \right|^q
  }
  {\left[\rhoh(s,\uniform^{\theta}_s)\right]^q}
  \left|
  \big\langle \uvec_\nu,
  \big(
  1 ,
  \tfrac{ 
  W_{\uniform^{\theta }_s}^{\theta}- W^{\theta}_{s}
  }{ \uniform^{\theta}_s-s}
  \big)\big\rangle
  \right|^q
  \right]
  \\
  &=
\int_0^1
\E\!\left[
  \tfrac{
  \left|
  \big(\funcF({\bf U}_{n,M}^{\theta})\big)
  (s+u(T-s),x+W_{s+u(T-s)}^{\theta}-W_t^{\theta})
  \right|^q
  }
  {\left[\rhoh(s,s+u(T-s))\right]^q}
  \left|
   \big\langle \uvec_\nu,
  \big(
  1 ,
  \tfrac{ 
  W_{s+u(T-s)}^{\theta}- W^{\theta}_{s}
  }{ s+u(T-s)-s}
  \big)\big\rangle
  \right|^q
  \right]
  \rho(u)
  \,du
  \\
  &
  \leq
  \int_0^1
  \tfrac{\sup_{z\in[t,T)}
  \left(
  \E\!\left[
  \left|
  \big(\funcF({\bf U}_{n,M}^{\theta})\big)
  (z,x+W_{z}^\theta-W_t^\theta)
  \right|^{\frac{p+q}{2}}
  \right]
  \right)^{\frac{2q}{p+q}}
  (T-s)^q}{[\rho(u)]^{q-1}}
  \left(
  1+
  \left(
  \E\!\left[
  \left|
  \tfrac{W^{\theta,1}_{s+u(T-s)}-W^{\theta,1}_s}{u(T-s)}
  \right|^{\frac{q(p+q)}{p-q}}
  \right]
  \right)^{\frac{p-q}{p+q}}
  \right)
  \,du
  \\
  &
  =
  \left(
  \sup_{z\in [t,T)}
  \E\!\left[
  \left|
  \big(\funcF({\bf U}_{n,M}^{\theta})\big)
  (z,x+W_{z}^\theta-W_t^\theta)
  \right|^{\frac{p+q}{2}}
  \right]
  \right)^{\frac{2q}{p+q}}
  \int_0^1
  \tfrac{(T-s)^{q}}{
\left[\rho(u)\right]^{q-1}  
}
  \left(
  1+
  \tfrac{(u(T-s))^{\frac{q}{2}}}
  {(u(T-s))^q T^{\frac{q}{2}}}
  \left(
  \E\!\left[
  |
  W^{0,1}_T
  |^{\frac{q(p+q)}{p-q}}
  \right]
  \right)^{\frac{p-q}{p+q}}
  \right)
  \,du
  \\
  &
  \leq
  \left(
  \sup_{z\in [t,T)}
  \E\!\left[
  \left|
  \big(\funcF({\bf U}_{n,M}^{\theta})\big)
  (z,x+W_{z}^\theta-W_t^\theta)
  \right|^{\frac{p+q}{2}}
  \right]
  \right)^{\frac{2q}{p+q}}
  \left(
  T^{q}+
  \left(
  \E\!\left[
  |
  W^{0,1}_T
  |^{\frac{q(p+q)}{p-q}}
  \right]
  \right)^{\frac{p-q}{p+q}}
  \right)
  \int_0^1
  \tfrac{u^{-\frac{q}{2}}}{
\left[\rho(u)\right]^{q-1}  
}
  \,du.
\end{split}     
\end{equation}
Now we prove~\eqref{eq:approx.int.ii}
by induction on $n\in\N_0$.
In the base case $n=0$, the fact that for all $\theta\in\Theta$
it holds that $U^\theta_{0,M} = 0$,
\eqref{eq:prepa},
and~\eqref{eq:int_cond_int}
ensure that 
for all $\theta\in\Theta$, $q\in [1,p)$, $t\in[0,T)$, $\nu \in \{1,2,\ldots, d+1\}$ it holds that
\begin{equation}
\begin{split} \label{eq:b38:eqa}
&\sup_{s\in[t,T)}
\left\{
\E\!\left[ \left| 
  ({\bf U}_{0,M}^{\theta}(s,x+W_{s}^\theta-W_t^{\theta})_\nu
\right|^q \right] 
+
\E\!\left[
\tfrac{
\left|
\big(\funcF({\bf U}_{0,M}^{\theta})\big)
(\uniform^{\theta}_s,x+W_{\uniform^{\theta}_s}^{\theta}-W_t^{\theta})
\right|^q
}
{\left[\rhoh(s,\uniform^{\theta}_s)\right]^q}
\left|
\big \langle \uvec_\nu,
\big(
1 ,
\tfrac{ 
W_{\uniform^{\theta }_s}^{\theta}- W^{\theta}_{s}
}{ \uniform^{\theta}_s-s}
\big)\big\rangle
\right|^q
\right]
\right\}\\
&=
\sup_{s\in[t,T)}
\left\{
\E\!\left[
\tfrac{
\left|
\big(\funcF(0)\big)
(\uniform^{\theta}_s,x+W_{\uniform^{\theta}_s}^{\theta}-W_t^{\theta})
\right|^q
}
{\left[\rhoh(s,\uniform^{\theta}_s)\right]^q}
\left|
\big \langle \uvec_\nu,
\big(
1 ,
\tfrac{ 
W_{\uniform^{\theta }_s}^{\theta}- W^{\theta}_{s}
}{ \uniform^{\theta}_s-s}
\big)\big\rangle 
\right|^q
\right]
\right\}
\\
&
\leq
\left(
\sup_{z\in [t,T)}
\E\!\left[
\left|
\big(\funcF(0)\big)
(z,x+W_{z}^0-W_t^0)
\right|^{\frac{p+q}{2}}
\right]
\right)^{\frac{2q}{p+q}}
\left(
T^q+
\left(
\E\!\left[
|
W^{0,1}_T
|^{\frac{q(p+q)}{p-q}}
\right]
\right)^{\frac{p-q}{p+q}}
\right)
\int_0^1
\tfrac{u^{-\frac{q}{2}}}{
\left[\rho(u)\right]^{q-1}  
}
\,du<\infty.
\end{split}     
\end{equation}
This establishes \eqref{eq:approx.int.ii} in the base case $n=0$. 
For the induction step $\N_0 \ni n-1 \rightsquigarrow n \in \N$ let $n\in \N$ and assume that 
for all $k \in \N_0 \cap [0, n)$, $\theta\in\Theta$, $q\in [1,p)$, $t\in [0,T)$, $\nu \in \{1,2,\ldots, d+1\}$ it holds that
\begin{equation}
\begin{split} \label{eq:b38:eq2}
\sup_{s\in[t,T)}
\left\{
\E\!\left[ \left| 
  ({\bf U}_{k,M}^{\theta}(s,x+W_{s}^\theta-W_{t}^\theta))_\nu
\right|^q \right] 
+
\E\!\left[
\tfrac{
\left|
\big(\funcF({\bf U}_{k,M}^{\theta})\big)
(\uniform^{\theta}_s,x+W_{\uniform^{\theta}_s}^{\theta}-W_t^{\theta})
\right|^q
}
{\left[\rhoh(s,\uniform^{\theta}_s)\right]^q}
\left|
\big \langle \uvec_\nu,
\big(
1 ,
\tfrac{ 
W_{\uniform^{\theta }_s}^{\theta}- W^{\theta}_{s}
}{ \uniform^{\theta}_s-s}
\big)\big \rangle
\right|^q
\right]
\right\}
<\infty.
\end{split}     
\end{equation}
Observe that 
\eqref{eq:def:U}
and 
Jensen's inequality
ensure that 
for all $\theta\in\Theta$, $q\in [1,p)$, $t\in [0,T)$, $s \in [t, T)$, $\nu \in \{1,2,\ldots, d+1\}$
it holds that 
\begin{equation}  \begin{split}\label{b38:eq33}
&\E\!\left[ \left| 
  ({\bf U}_{n,M}^{\theta}(s,x+W_{s}^\theta-W_{t}^\theta))_\nu
\right|^q \right] 
\leq (3+2n)^{q-1} |g(x)|^q
+ (3+2n)^{q-1} 
\E\!\left[ 
|
  g(x+W_{s}^\theta-W_{t}^\theta)
  -g(x)
|^q  
 \right] \\
 &
   +
\tfrac{(3+2n)^{q-1} }{M^n}\sum_{i=1}^{M^n}
 \E\!\left[ 
\left|g(x+W_{s}^\theta-W_{t}^\theta+W^{(\theta,0,-i)}_T-W^{(\theta,0,-i)}_s)-g(x+W_{s}^\theta-W_{t}^\theta))\right|^q
\left|
\big \langle \uvec_\nu,
\big(
1 ,
\tfrac{ 
W^{(\theta, 0, -i)}_{T}- W^{(\theta, 0, -i)}_{s}
}{ T - s }
\big)\big\rangle  
\right|^q
\right]
\\
& +  \sum_{l=0}^{n-1}\tfrac{(3+2n)^{q-1}}{M^{n-l}}\sum_{i=1}^{M^{n-l}}
\E\!\vast[
\tfrac{
\left|
\big(\funcF({\bf U}_{l,M}^{(\theta,l,i)})\big)
(\uniform^{(\theta, l,i)}_s,x+W_s^{\theta}-W_t^{\theta}+W_{\uniform^{(\theta, l,i)}_s}^{(\theta,l,i)}-W_s^{(\theta,l,i)})
\right|^q
}
{\left[\rhoh(s,\uniform^{(\theta, l,i)}_s)\right]^q}
\left|
\big \langle \uvec_\nu,
\big(
1 ,
\tfrac{ 
W_{\uniform^{(\theta, l,i)}_s}^{(\theta,l,i)}- W^{(\theta, l, i)}_{s}
}{ \uniform^{(\theta, l,i)}_s-s}
\big)\big\rangle
\right|^q
\vast]
\\
&+  \sum_{l=1}^{n-1}\tfrac{(3+2n)^{q-1}}{M^{n-l}}\sum_{i=1}^{M^{n-l}}
\E\!\vast[
\tfrac{
\left|
\big(\funcF( {\bf U}_{l-1,M}^{(\theta,-l,i)})\big)
(\uniform^{(\theta, l,i)}_s,x+W_s^{\theta}-W_t^{\theta}+W_{\uniform^{(\theta, l,i)}_s}^{(\theta,l,i)}-W_s^{(\theta,l,i)})
\right|^q
}
{\left[\rhoh(s,\uniform^{(\theta, l,i)}_s)\right]^q}
\left|
\big \langle \uvec_\nu,
\big(
1 ,
\tfrac{ 
W_{\uniform^{(\theta, l,i)}_s}^{(\theta,l,i)}- W^{(\theta, l, i)}_{s}
}{ \uniform^{(\theta, l,i)}_s-s}
\big)\big \rangle
\right|^q
\vast].
\end{split}     \end{equation}
This,
Hutzenthaler et al.~\cite[Corollary 2.5]{HJKNW2018}
together with Lemma~\ref{properties_approx},
Item~\eqref{item:approximations.integrable.i},
and
the induction hypothesis~\eqref{eq:b38:eq2}
yield 
for all $\theta\in\Theta$, $q\in [1,p)$, $t\in [0,T)$, $\nu \in \{1,2,\ldots, d+1\}$
that 
\begin{equation}  \begin{split}\label{b38:eq3}
&\sup_{s\in[t,T)}\E\!\left[ \left| 
  ({\bf U}_{n,M}^{\theta}(s,x+W_{s}^\theta-W_{t}^\theta))_\nu
\right|^q \right] 
\leq (3+2n)^{q-1} |g(x)|^q
\\&\quad+
(6+2n)^{q} 
\sup_{u\in(0,T]}
\sup_{s\in[0,u)}
\sup_{y\in\R^d}
\sup_{j\in\{1,2,\ldots,d+1\}}
 \E\!\left[ 
\left|g(y+W^{0}_u-W^{0}_s)-g(y))\right|^q
\left|
\big \langle \uvec_j,
\big(
1 ,
\tfrac{ 
W^{0}_{u}- W^{0}_{s}
}{ u - s }
\big)\big\rangle
\right|^q
\right]
\\
&\quad +  \sum_{l=0}^{n-1}(6+2n)^{q}
\sup_{s\in[t,T)}\E\!\left[
\tfrac{
\left|
\big(\funcF({\bf U}_{l,M}^{0})\big)
(\uniform^{0}_s,x+W_{\uniform^{0}_s}^{0}-W_t^{0})
\right|^q
}
{\left[\rhoh(s,\uniform^{0}_s)\right]^q}
\left|
\big \langle \uvec_\nu,
\big(
1 ,
\tfrac{ 
W_{\uniform^{0}_s}^{0}- W^{0}_{s}
}{ \uniform^{0}_s-s}
\big)\big \rangle
\right|^q
\right]<\infty.
\end{split}     \end{equation}
Furthermore, Jensen's inequality,
the Lipschitz property~\eqref{eq:fLipschitz} of $f$,
\eqref{b38:eq3},
and
assumption~\eqref{eq:int_cond_int}
yield
that for all $\theta\in\Theta$, $q\in [1,p)$, $t\in [0,T)$
it holds that
\begin{equation}
\begin{split}
 & \sup_{s\in [t,T)} 
 \left\{
 \E\!\left[
  \left|
  \big(\funcF({\bf U}_{n,M}^{\theta})\big)
  (s,x+W_{s}^\theta-W_t^{\theta})
  \right|^{q}
  \right]
  \right\}
  \\
  &
  \le 
  2^{q-1}
   \sup_{s\in [t,T)} 
  \left(
  \E\!\left[
  \left|
  \big(\funcF({\bf U}_{n,M}^{\theta})
  -\funcF(0)\big)
  (s,x+W_{s}^\theta-W_t^{\theta})
  \right|^{q}
  \right]
  +
  \E\!\left[
  \left|
  \big(
  \funcF(0)\big)
  (s,x+W_{s}^\theta-W_t^{\theta})
  \right|^{q}
  \right]
  \right)
   \\
  &
  \le 
  2^{q-1}
   \sup_{s\in [t,T)}
  \left(
  \E\!\left[
  \left|
  \sum_{\nu=1}^{d+1}
  L_\nu
  {\bf U}_{n,M}^{\theta,\nu}  (s,x+W_{s}^\theta-W_t^{\theta})
  \right|^{q}
  \right]
  +
  \E\!\left[
  \left|
  \big(
  \funcF(0)\big)
  (s,x+W_{s}^\theta-W_t^{\theta})
  \right|^{q}
  \right]
  \right)
  \\
  &
\leq 
  (4d)^q
  \sum_{\nu=1}^{d+1}(L_\nu)^q
   \sup_{s\in [t,T)}
  \E\!\left[
  \left|
 {\bf U}_{n,M}^{\theta,\nu}  (s,x+W_{s}^\theta-W_t^{\theta})
  \right|^{q}
  \right]
  +
  2^q
   \sup_{s\in [t,T)}
  \E\!\left[
  \left|
  \big(
  \funcF(0)\big)
  (s,x+W_{s}^\theta-W_t^{\theta})
  \right|^{q}
  \right]
  \\&
  <\infty.
  \end{split}
\end{equation}
This,~\eqref{eq:prepa},
and assumption~\eqref{eq:int_cond_int}
implies that
for all $\theta\in\Theta$, $q\in [1,p)$, $t\in [0,T)$, $\nu \in \{1,2,\ldots, d+1\}$ it holds that
\begin{equation}\label{eq:RHS_intq}
\begin{split}
&\sup_{s\in[t,T)}
\E\!\left[
  \tfrac{
  \left|
  \big(\funcF({\bf U}_{n,M}^{\theta})\big)
  (\uniform^{\theta}_s,x+W_{\uniform^{\theta}_s}^{\theta}-W_t^{\theta})
  \right|^q
  }
  {\left[\rhoh(s,\uniform^{\theta}_s)\right]^q}
  \bigg|
 \big \langle \uvec_\nu,
  \big(
  1 ,
  \tfrac{ 
  W_{\uniform^{\theta }_s}^{\theta}- W^{\theta}_{s}
  }{ \uniform^{\theta}_s-s}
  \big)\big\rangle 
  \bigg|^q
  \right]
\\ &\leq
  \left(
  \sup_{s\in [t,T)}
  \E\!\left[
  \left|
  \big(\funcF({\bf U}_{n,M}^{\theta})\big)
  (s,x+W_{s}^\theta-W_t^\theta)
  \right|^{\frac{p+q}{2}}
  \right]
  \right)^{\frac{2q}{p+q}}
  \left(
  T^{q}+
  \left(
  \E\!\left[
  \left|
  W^{0,1}_T
  \right|^{\frac{q(p+q)}{p-q}}
  \right]
  \right)^{\frac{p-q}{p+q}}
  \right)
  \int_0^1
  \tfrac{u^{-\frac{q}{2}}}{
\left[\rho(u)\right]^{q-1}  
}
  \,du
  \\&<\infty.
\end{split}     
\end{equation}
This and \eqref{b38:eq3} finish the induction step.
Induction hence proves Item~\eqref{item:approximations.integrable.ii}. 

   Finally, we prove Item~\eqref{item:approximations.integrable.iii}.
   Note that \eqref{eq:def:U}, Item \eqref{item:approximations.integrable.ii},
   Corollary~2.5 in~\cite{HJKNW2018} together with
   the fact that 
   for all $n\in\N_0$ it holds that
  ${\bf U}_{n,M}^{\theta}$, $\theta\in\Theta$, are identically
distributed (see Item~\eqref{properties_approx:item1} in Lemma~\ref{properties_approx})
  and together with the fact that $(\uniform^{\theta},W^{\theta})$, $\theta\in\Theta$,
  are identically distributed,
a telescoping sum,
%disintegration of a continuous random field
%and an independent random argument (see, e.g., 
 Hutzenthaler et al.~\cite[Lemma 2.3]{HJKNW2018}),
and
the fact that for every $t\in[0,T)$, $\theta\in\Theta$
the distribution of $R^{\theta}_t$ has density $\rhoh(t,\cdot)$ with
respect to Lebesgue measure on $[t,T]$
yield that
  for all $n\in\N$, $\theta\in\Theta$, $t\in[0,T)$
  it holds that
  \begin{equation}  \begin{split}
  &\E\left[{\bf U}_{n,M}^{\theta}(t,x)\right]- \E\!\left[
    g(x+W_T^0-W_t^0)
     \Big(
     1 ,
     \tfrac{ 
     W^{0}_{T}- W^{0}_{t}
     }{ T - t }
     \Big)
    \right]
  \\
  &=\sum_{l=0}^{n-1}
  \sum_{i=1}^{M^{n-l}}\tfrac{1}{M^{n-l}}\E\left[
  \tfrac{
  \big(\funcF({\bf U}_{l,M}^{(\theta,l,i)})\big)
  (\uniform_t^{(\theta,l,i)},x+W_{\uniform_t^{(\theta,l,i)}}^{(\theta,l,i)}-W_t^{(\theta,l,i)})}
  {
  \rhoh(t,\uniform_t^{(\theta,l,i)})
  }
  \Big(
  1 ,
  \tfrac{ 
  W^{(\theta, l, i)}_{\uniform_t^{(\theta,l,i)}}- W^{(\theta, l, i)}_{t}
  }{ \uniform_t^{(\theta,l,i)} - t }
  \Big)\right]\\
  &\qquad
  -\sum_{l=1}^{n-1}
  \sum_{i=1}^{M^{n-l}}\tfrac{1}{M^{n-l}}\E\left[
  \tfrac{
  \big(\funcF( {\bf U}_{l-1,M}^{(\theta,-l,i)})\big)
  (\uniform_t^{(\theta,l,i)},x+W_{\uniform_t^{(\theta,l,i)}}^{(\theta,l,i)}-W_t^{(\theta,l,i)})}
  {
  \rhoh(t,\uniform_t^{(\theta,l,i)})
  }
  \Big(
  1 ,
  \tfrac{ 
  W^{(\theta, l, i)}_{\uniform_t^{(\theta,l,i)}}- W^{(\theta, l, i)}_{t}
  }{ \uniform_t^{(\theta,l,i)} - t }
  \Big)\right]\\
  &=\sum_{l=0}^{n-1}
  \left(
  \E\left[
  \tfrac{
  \big(\funcF({\bf U}_{l,M}^{\theta})\big)
  (\uniform_t^{\theta},x+W_{\uniform_t^{\theta}}^{\theta}-W_t^{\theta})}
  {
  \rhoh(t,\uniform_t^{\theta})
  }
  \Big(
  1 ,
  \tfrac{ 
  W^{\theta}_{\uniform_t^{\theta}}- W^{\theta}_{t}
  }{ \uniform_t^{\theta} - t }
  \Big)\right]
  -
  \1_{\N}(l)\E\left[
  \tfrac{
  \big(\funcF( {\bf U}_{l-1,M}^{\theta})\big)
  (\uniform_t^{\theta},x+W_{\uniform_t^{\theta}}^{\theta}-W_t^{\theta})}
  {
  \rhoh(t,\uniform_t^{\theta})
  }
  \Big(
  1 ,
  \tfrac{ 
  W^{\theta}_{\uniform_t^{\theta}}- W^{\theta}_{t}
  }{ \uniform_t^{\theta} - t }
  \Big)\right]
  \right)
  \\
  &
  =
  \E\left[
  \tfrac{
  \big(\funcF({\bf U}_{n-1,M}^{\theta})\big)
  (\uniform_t^{\theta},x+W_{\uniform_t^{\theta}}^{\theta}-W_t^{\theta})}
  {
  \rhoh(t,\uniform_t^{\theta})
  }
  \Big(
  1 ,
  \tfrac{ 
  W^{\theta}_{\uniform_t^{\theta}}- W^{\theta}_{t}
  }{ \uniform_t^{\theta} - t }
  \Big)\right]\\
  &
  =
  \int_t^T\E\left[
  \tfrac{
  \big(\funcF({\bf U}_{n-1,M}^{\theta})\big)
  (s,x+W_{s}^{\theta}-W_t^{\theta})}
  {
  \rhoh(t,s)
  }
  \Big(
  1 ,
  \tfrac{ 
  W^{\theta}_{s}- W^{\theta}_{t}
  }{ s - t }
  \Big)\right]\P(\uniform_t^{\theta}\in ds)
  \\
& = \int_t^T\E\!\left[
    \left(\funcF( {\bf U}_{n-1,M,Q}^{\theta})\right)\!(s,x+W_s^\theta-W_t^\theta)
     \Big(
     1 ,
     \tfrac{ 
     W^{\theta}_{s}- W^{\theta}_{t}
     }{ {s - t} }
     \Big)
    \right]\,ds.
\end{split}     \end{equation}
  This establishes Item~\eqref{item:approximations.integrable.iii}.
  The proof of Lemma \ref{lem:b38} is thus completed.
\end{proof}

\subsection{Error analysis for MLP approximations}\label{subsec:error_analysis}

\begin{lemma}[Recursive bound for global error]\label{l:estimate.L2error}
  Assume Setting~\ref{s:full.discretization},
  let $p\in (1,\infty)$,
$M \in \N$,
and assume 
for all $q\in [1,p)$, $t\in[0,T)$, $x\in \R^d$ that
\begin{equation}\label{eq:approximations.integrable.assump2}
\int_0^1
  \frac{1}{
  s^{\frac{q}{2}}
\left[\rho(s)\right]^{q-1}  
}
  \,ds+
  \sup_{s\in [t,T)}
  \E\!\left[
  \left|
  \big(\funcF(0)\big)
  (s,x+W_{s}^0-W_t^0)
  \right|^{q}
  \right]<\infty.
\end{equation}
Then it holds
for all  $n\in\N$, $t\in[0,T)$, $x\in\R^d$, $\nu_0 \in \{1,2,\ldots, d+1\}$
 that
  \begin{equation}  \begin{split}\label{eq:estimate.L2error}
   &     \left\| {\bf U}_{{n},M}^{0,\nu_0}({t},x)-{\bf u}_{\nu_0}(t,x)\right\|_{L^2(\P;\R)}
     \\&
    \leq 
    \sum_{j=0}^{n-1}
    \sum_{\substack{\nu_1,\nu_2,\ldots,\nu_{j+1} \in \\ \{1,2,\ldots,d+1\}}}
    \tbinom{n-1}{j}
   \tfrac{\mathrm{1}_{\{1\}}(\nu_{j+1})2^j\left[\prod_{i=1}^j L_{\nu_i}\right]}{\sqrt{M^{n-j}}}   
    \Big\|\left(g(x+W^0_{T}-W^0_{t})-g(x+W^0_{S(j,t)}-W^0_{t})\right)\\
    &  \qquad \qquad \qquad \qquad \qquad
    \cdot  
    \big\langle \uvec_{\nu_j},
    \big(   
   1,
     \tfrac{ 
     W^{0}_{T}- W^{0}_{\ES(j,t)}
     }{T-\ES(j,t) } 
     \big)\big \rangle 
   \prod_{i=1}^{j}
   \tfrac{1}{\rhoh(\ES(i-1,t),\ES(i,t))}
   \big\langle \uvec_{\nu_{i-1}},
     \big(   
   1,
     \tfrac{ 
     W^{0}_{\ES(i,t)}- W^{0}_{\ES(i-1,t)}
     }{\ES(i,t)-\ES(i-1,t) } 
     \big)\big\rangle 
   \Big\|_{L^2(\P;\R)}\\
   &+
    \sum_{j=0}^{n-1}
    \sum_{\substack{\nu_1,\nu_2,\ldots,\nu_{j+1} \in \\ \{1,2,\ldots,d+1\}}}
    \tbinom{n-1}{j}
   \tfrac{\mathrm{1}_{\{1\}}(\nu_{j+1})2^j\left[\prod_{i=1}^j L_{\nu_i}\right]}{\sqrt{M^{n-j}}}   
    \Big\|(\funcF (0))(\ES(j+1,t),x+W^0_{\ES(j+1,t)}-W^0_t)\\
    &\qquad \qquad \qquad \qquad \qquad \qquad \qquad \qquad \qquad \qquad
    \cdot  
   \prod_{i=1}^{j+1}
   \tfrac{1}{\rhoh(\ES(i-1,t),\ES(i,t))}
   \big\langle \uvec_{\nu_{i-1}},
     \big(   
   1,
     \tfrac{ 
     W^{0}_{\ES(i,t)}- W^{0}_{\ES(i-1,t)}
     }{\ES(i,t)-\ES(i-1,t) } 
     \big)\big\rangle 
   \Big\|_{L^2(\P;\R)}
   \\
   &+
    \sum_{j=0}^{n-1}
    \sum_{\substack{\nu_1,\nu_2,\ldots,\nu_{j+1} \in \\ \{1,2,\ldots,d+1\}}}
    \tbinom{n-1}{j}
   \tfrac{2^j\left[\prod_{i=1}^{j+1} L_{\nu_i}\right]}{\sqrt{M^{n-j-1}}}   
    \Big\|{\bf u}_{\nu_{j+1}}(\ES(j+1,t),x+W^0_{\ES(j+1,t)}-W^0_t)\\
    &\qquad \qquad \qquad \qquad \qquad \qquad \qquad \qquad \qquad \qquad
    \cdot  
   \prod_{i=1}^{j+1}
   \tfrac{1}{\rhoh(\ES(i-1,t),\ES(i,t))}
   \big \langle \uvec_{\nu_{i-1}},
     \big(   
   1,
     \tfrac{ 
     W^{0}_{\ES(i,t)}- W^{0}_{\ES(i-1,t)}
     }{\ES(i,t)-\ES(i-1,t) } 
     \big)\big\rangle
   \Big\|_{L^2(\P;\R)}.
    \end{split}
    \end{equation}
\end{lemma}
\begin{proof}[Proof of Lemma~\ref{l:estimate.L2error}]
 First, we analyze the \emph{Monte Carlo error}.
 Item~\eqref{item:approximations.integrable.i}
 of Lemma~\ref{lem:b38},  Item~\eqref{item:approximations.integrable.ii}
 of Lemma~\ref{lem:b38}, and Item~\eqref{properties_approx:item6} of Lemma~\ref{properties_approx}
 imply that all summands on the right-hand side of~\eqref{eq:def:U}
 are integrable.
 Lemma~\ref{properties_approx} yields that
 the summands in~\eqref{eq:def:U} are
 pairwise independent. 
 Then~\eqref{eq:def:U}
 and the fact that 
 the summation formula for variances of pairwise independent, integrable
 random variables
%%%%%% % pairwise independence and integrability yields that products of different
%%%%%% % (centered) summands are in L^1 with vanishing expectation.
%%%%%% % Thus we can apply linearity after the binomial formula:
%%%%%% % E[(X+Y+Z)^2]=E[(X^2+Y^2+Z^2+2XY+2XZ+2YZ)]=E[X^2]+E[Y^2]+E[Z^2]
%%%%%% % +2E[XY]+2E[XZ]+2E[YZ]=E[X^2]+E[Y^2]+E[Z^2]=Var(X)+Var(Y)+Var(Z)
 imply that
 for all  $m\in\N$, $x\in\R^d$, $t\in[0,T)$,
 $\nu \in \{1,2,\ldots,d+1\}$ it holds that
  \begin{equation} \label{eq:var_approx1} \begin{split}
    &\Var\!\left(
       {\bf U}_{m,M}^{0,\nu}(t,x)
     \right)
     =\tfrac{1}{M^m}
     \Var\!\left(\left(g(x+W_T^{0}-W_t^0)-g(x)\right)
     \big \langle \uvec_\nu,
     \big(
     1 ,
     \tfrac{ 
     W^{0}_{T}- W^{0}_{t}
     }{ T - t }
     \big)\big\rangle 
    \right)
     \\&+\sum_{l=0}^{m-1}\tfrac{1}{M^{m-l}}
     \Var\!     
     \left(
\tfrac{     
     \left(\funcF( {\bf U}_{l,M}^{(0,l,1)})-\1_{\N}(l)\funcF( {\bf U}_{l-1,M}^{(0,-l,1)})\right)\!
     (\uniform_t^{(0,l,1)},x+W_{\uniform_t^{(0,l,1)}}^{(0,l,1)}-W_t^{(0,l,1)})
     }
      {\rhoh(t,\uniform^{(0, l,1)}_t)}
      \Big \langle \uvec_\nu,
     \Big(
     1 ,
     \tfrac{ 
     W_{\uniform_t^{(0,l,1)}}^{(0,l,1)}-W_t^{(0,l,1)}
     }{ {\uniform_t^{(0,l,1)} - t} }
     \Big)\Big\rangle 
     \right)
   \\&
   \leq \tfrac{1}{M^m}
     \E\!\left[\left|
     \left(g(x+W_T^{0}-W_s^0)-g(x)\right)
     \big \langle \uvec_\nu,
     \big(
     1 ,
     \tfrac{ 
     W^{0}_{T}- W^{0}_{s}
     }{ T - s }
     \big)\big\rangle 
    \right|^2\right]
     \\&+\sum_{l=0}^{m-1}\tfrac{1}{M^{m-l}}
     \E\!\left[\left|
\tfrac{     
     \left(\funcF( {\bf U}_{l,M}^{(0,l,1)})-\1_{\N}(l)\funcF( {\bf U}_{l-1,M}^{(0,-l,1)})\right)\!
     (\uniform_t^{(0,l,1)},x+W_{\uniform_t^{(0,l,1)}}^{(0,l,1)}-W_t^{(0,l,1)})
     }
      {\rhoh(t,\uniform^{(0, l,1)}_t)}
      \Big \langle \uvec_\nu,
     \Big(
     1 ,
     \tfrac{ 
     W_{\uniform_t^{(0,l,1)}}^{(0,l,1)}-W_t^{(0,l,1)}
     }{ {\uniform_t^{(0,l,1)} - t} }
     \Big)\Big\rangle 
     \right|^2\right]
%      \\&
%   =\tfrac{1}{M^m}
%     \E\!\left[\left|
%     \left(g(x+W_T^{0}-W_s^0)-g(x)\right)
%     \Big(
%     1 ,
%     \tfrac{ 
%     W^{0}_{T}- W^{0}_{s}
%     }{ T - s }
%     \Big)_{\nu}
%    \right|^2\right]
%     \\&+\sum_{l=0}^{m-1}\tfrac{(T-t)}{M^{m-l}}
%     \E\!\left[\int_t^T\left|\left(\funcF( {\bf U}_{l,M}^{(0,l,1)})-\1_{\N}(l)\funcF( {\bf U}_{l-1,M}^{(0,-l,1)})\right)\!
%     (s,x+W_{s}^{(0,l,1)}-W_t^{(0,l,1)})\,
%     \Big(
%     1 ,
%     \tfrac{ 
%     W_{s}^{(0,l,1)}-W_t^{(0,l,1)}
%     }{ {s - t} }
%     \Big)_{\nu}
%     \right|^2\,ds \right].
  \end{split}     \end{equation}
 Combining this, the triangle inequality, and \eqref{eq:fLipschitz} 
 yields that
  for all 
  $m\in\N$, $x\in\R^d$, $t\in[0,T)$, $\nu \in \{1,2,\ldots,d+1\}$
  it holds that
  \begin{equation}  \begin{split}
   &     \left\| {\bf U}_{m,M}^{0,\nu}(t,x)-\E\!\left[{\bf U}_{m,M}^{0,\nu}(t,x)\right]\right\|_{L^2(\P;\R)}
   =
   \left(
      \Var\!\left(
       {\bf U}_{m,M}^{0,\nu}(t,x)
     \right)
     \right)^{\nicefrac{1}{2}}
 \\&
    \leq
     \tfrac{1}{\sqrt{M^m}}
     \left\|
     \left(g(x+W_T^{0}-W_t^0)-g(x)\right)
     \big \langle \uvec_\nu,
     \big(
     1 ,
     \tfrac{ 
     W^{0}_{T}- W^{0}_{t}
     }{ T - t }
     \big)\big\rangle
     \right\|_{L^2(\P;\R)}
    \\&
     +
    \sum_{l=0}^{m-1}
    \tfrac{1}{\sqrt{M^{m-l}}}
    \left\|
    \tfrac{
     \left(\funcF( {\bf U}_{l,M}^{(0,l,1)})-\1_{\N}(l)\funcF( {\bf U}_{l-1,M}^{(0,-l,1)})\right)\!
     (\uniform_t^{(0,l,1)},x+W_{\uniform_t^{(0,l,1)}}^{(0,l,1)}-W_t^{(0,l,1)})
     }
      {\rhoh(t,\uniform^{(0, l,1)}_t)}
      \Big \langle \uvec_\nu,
     \Big(
     1 ,
     \tfrac{ 
     W_{\uniform_t^{(0,l,1)}}^{(0,l,1)}-W_t^{(0,l,1)}
     }{ {\uniform_t^{(0,l,1)} - t} }
     \Big)\Big\rangle
    \right\|_{L^2(\P;\R)}   
 \\&
    \leq
     \tfrac{1}{\sqrt{M^m}}
     \left(
     \left\|
     \left(g(x+W_T^{0}-W_t^0)-g(x)\right)
     \big \langle \uvec_\nu,
     \big(
     1 ,
     \tfrac{ 
     W^{0}_{T}- W^{0}_{t}
     }{ T - t }
     \big)\big\rangle
     \right\|_{L^2(\P;\R)}
    +
    \left\|
    \tfrac{
     \left(\funcF( 0)\right)\!
     (\uniform^{0}_t,x+W_{\uniform_t^{0}}^0-W_t^0)
     }
      {\rhoh(t,\uniform^{0}_t)}
      \Big \langle \uvec_\nu,
     \Big(
     1 ,
     \tfrac{ 
     W^{0}_{\uniform_t^{0}}- W^{0}_{t}
     }{ {{\uniform_t^{0}} - t} }
     \Big)\Big \rangle
    \right\|_{L^2(\P;\R)}\right)
    \\&
     +
    \sum_{l=1}^{m-1}
     \tfrac{1}{\sqrt{M^{m-l}}}
    \left\|
    \sum_{\nu_1=1}^{d+1}L_{\nu_1}
    \tfrac{
     \left|
	\left(     
     {\bf U}_{l,M}^{(0,l,1),\nu_1}-{\bf U}_{l-1,M}^{(0,-l,1),\nu_1}\right)\!
     (\uniform^{(0, l,1)}_t,x+W_{\uniform^{(0, l,1)}_t}^{(0,l,1)}-W_t^{(0,l,1)}) 
     \right|
     }
      {\rhoh(t,\uniform^{(0, l,1)}_t)}
     \left|
     \Big \langle \uvec_\nu,
     \Big(
     1 ,
     \tfrac{ 
     W^{(0, l,1)}_{\uniform^{(0, l,1)}_t}- W^{(0, l,1)}_{t}
     }{ {\uniform^{(0, l,1)}_t - t} }
     \Big)\Big\rangle 
     \right|
    \right\|_{L^2(\P;\R)}.
    \end{split}
    \end{equation}
    This and the triangle inequality ensure that
  for all 
  $m\in\N$, $x\in\R^d$, $t\in[0,T)$, $\nu \in \{1,2,\ldots,d+1\}$
  it holds that
    \begin{equation}\label{eq:final.MCerror}
    \begin{split}
     &     \left\| {\bf U}_{m,M}^{0,\nu}(t,x)-\E\!\left[{\bf U}_{m,M}^{0,\nu}(t,x)\right]\right\|_{L^2(\P;\R)}\\
    &
    \leq
      \tfrac{1}{\sqrt{M^m}}
     \left(
     \left\|
     \left(g(x+W_T^{0}-W_t^0)-g(x)\right)
     \big \langle \uvec_\nu,
     \big(
     1 ,
     \tfrac{ 
     W^{0}_{T}- W^{0}_{t}
     }{ T - t }
     \big)\big\rangle 
     \right\|_{L^2(\P;\R)}
    +
    \left\|
    \tfrac{
     \left(\funcF( 0)\right)\!
     (\uniform^{0}_t,x+W_{\uniform_t^{0}}^0-W_t^0)
     }
      {\rhoh(t,\uniform^{0}_t)}
      \Big \langle \uvec_\nu,
     \Big(
     1 ,
     \tfrac{ 
     W^{0}_{\uniform_t^{0}}- W^{0}_{t}
     }{ {{\uniform_t^{0}} - t} }
     \Big)\Big\rangle 
    \right\|_{L^2(\P;\R)}\right)
    \\&
     +
    \sum_{l=1}^{m-1}
     \sum_{\nu_1=1}^{d+1}
     \tfrac{L_{\nu_1}}{\sqrt{M^{m-l}}}
     \left\|
	\tfrac{     
     \left({\bf U}_{l,M}^{0,\nu_1}-{\bf u}\right)\!
     (\uniform_t^0,x+W_{\uniform_t^0}^0-W_t^0)
     }
     {
     \rhoh(t,\uniform_t^0)
     }
     \Big \langle \uvec_\nu,
     \Big(
     1 ,
     \tfrac{ 
     W^{0}_{\uniform_t^0}- W^{0}_{t}
     }{ {\uniform_t^0 - t} }
     \Big)\Big\rangle 
    \right\|_{L^2(\P;\R)}
      \\&
     +
    \sum_{l=1}^{m-1}
     \sum_{\nu_1=1}^{d+1}
     \tfrac{L_{\nu_1}}{\sqrt{M^{m-l}}}
    \left\|
    \tfrac{
     \left({\bf U}_{l-1,M}^{0,\nu_1}-{\bf u}\right)\!
     (\uniform_t^0,x+W_{\uniform_t^0}^0-W_t^0)
     }
     {\rhoh(t,\uniform_t^0)}
     \Big \langle \uvec_\nu,
     \Big(
     1 ,
     \tfrac{ 
     W^{0}_{\uniform_t^0}- W^{0}_{t}
     }{ {\uniform_t^0 - t} }
     \Big)\Big\rangle 
    \right\|_{L^2(\P;\R)}
    \\&
    =
     \tfrac{1}{\sqrt{M^m}}
     \left(
     \left\|
     \left(g(x+W_T^{0}-W_t^0)-g(x)\right)
     \big \langle \uvec_\nu,
     \big(
     1 ,
     \tfrac{ 
     W^{0}_{T}- W^{0}_{t}
     }{ T - t }
     \big)\big\rangle 
     \right\|_{L^2(\P;\R)}
    +
    \left\|
    \tfrac{
     \left(\funcF( 0)\right)\!
     (\uniform^{0}_t,x+W_{\uniform_t^{0}}^0-W_t^0)
     }
      {\rhoh(t,\uniform^{0}_t)}
      \Big \langle \uvec_\nu,
     \Big(
     1 ,
     \tfrac{ 
     W^{0}_{\uniform_t^{0}}- W^{0}_{t}
     }{ {{\uniform_t^{0}} - t} }
     \Big)\Big\rangle 
    \right\|_{L^2(\P;\R)}\right)
    \\&
     +
    \sum_{l=0}^{m-1}
    \sum_{\nu_1=1}^{d+1} \tfrac{ L_{\nu_1}}{\sqrt{M^{m-l-1}}}
    \left(
    \tfrac{\1_{(0,m)}(l)}{\sqrt{M}}+\1_{(-1,m-1)}(l)
    \right)
    \left\|
	\tfrac{     
     \left( {\bf U}_{l,M}^{0,\nu_1}-{\bf u}\right)\!
     (\uniform_t^0,x+W_{\uniform_t^0}^0-W_t^0)
     }
     {
     \rhoh(t,\uniform_t^0)
     }
     \Big \langle \uvec_\nu,
     \Big(
     1 ,
     \tfrac{ 
     W^{0}_{\uniform_t^0}- W^{0}_{t}
     }{ {\uniform_t^0 - t} }
     \Big)\Big\rangle 
    \right\|_{L^2(\P;\R)}
    .
  \end{split}     \end{equation}
  Next we analyze the \emph{time discretization error}.
Item~\eqref{item:approximations.integrable.ii} of Lemma~\ref{lem:b38} ensures that
  for all  $m\in\N$, $t\in[0,T)$, $x\in\R^d$
   it holds that
  \begin{equation}  \begin{split}\label{eq:discreteFeynmanKac2}
    \E\!\left[{\bf U}_{m,M}^{0}(t,x)
    -g(x+W_T^0-W_t^0)
     \Big(
     1 ,
     \tfrac{ 
     W^{0}_{T}- W^{0}_{t}
     }{ T - t }
     \Big)
    \right]
    =
    \E\!\left[
    \int_t^T
    \left(\funcF( {\bf U}_{m-1,M}^{0})\right)\!(s,x+W_s^0-W_t^0)
     \Big(
     1 ,
     \tfrac{ 
     W^{0}_{s}- W^{0}_{t}
     }{ {s - t} }
     \Big)\,
     ds
    \right].
  \end{split}     \end{equation}
%  Item \eqref{item:approximations.integrable.iiiFC} of Lemma \ref{l:nonlinear.FK.formula} proves that 
%  for all $t\in[0,T)$, $x\in\R^d$ it holds
%  that
%  \begin{equation}  \begin{split}\label{eq:BEL2}
%      {\bf u}^{\infty}(t,x)-\E\!\left[g(x+W_T^0-W_t^0)
%     \Big(
%     1 ,
%     \tfrac{ 
%     W^{0}_{T}- W^{0}_{t}
%     }{ T - t }
%     \Big)
%      \right]
%      &=\E\!\left[
%      \int_t^{T}(\funcF({\bf u}^{\infty}))(t,x+W_s^0-W_t^0)
%     \Big(
%     1 ,
%     \tfrac{ 
%     W^{0}_{s}- W^{0}_{t}
%     }{ s- t }
%     \Big)
%      \,ds
%      \right].
%    \end{split}     \end{equation}
This,~\eqref{eq:feynmankacuinfty},
linearity together with~\eqref{eq:feynmankacintegrability}
and with Item~\eqref{item:approximations.integrable.i}
in Lemma~\ref{lem:b38},
%the fact that  
%$({\bf U}_{m,M,Q}^{0})_{m\in \N}$ and $W^0$ are independent,
and Jensen's inequality
show
  for all  $m\in\N$, $t\in[0,T)$, $x\in\R^d$, $\nu \in \{1,2,\ldots,d+1\}$
  that
  \begin{equation}  \begin{split}\label{eq:final.Timeerror}
    &\left|\E\!\left[{\bf U}_{m,M}^{0,\nu}(t,x)\right]
            -{\bf u}_{\nu}(t,x)\right|
  \\&=
    \left|
    \E\!\left[
    \int_t^T
    \left(\funcF( {\bf U}_{m-1,M}^{0})-\funcF({\bf u})\right)\!(s,x+W_s^0-W_t^0)
    \big\langle \uvec_\nu,
     \big(
     1 ,
     \tfrac{ 
     W^{0}_{s}- W^{0}_{t}
     }{ {s - t} }
     \big)\big\rangle\,
     ds
     \right]
     \right|
 \\&\leq  
 \sum_{\nu_1=1}^{d+1}L_{\nu_1}
    \E\!\left[
    \int_t^T
    \left| 
    \left({\bf U}_{m-1,M}^{0,\nu_1}-{\bf u}_{\nu_1}\right)\!(s,x+W_s^0-W_t^0)
    \right|
     \left|
 \big\langle \uvec_\nu,     
     \big(
     1 ,
     \tfrac{ 
     W^{0}_{s}- W^{0}_{t}
     }{ {s - t} }
     \big)\big \rangle \right|\,
     ds
     \right]
     \\& =
    \sum_{\nu_1=1}^{d+1}L_{\nu_1}
    \E\!\left[
    \int_t^T
    \tfrac{
    \left| 
   \left( 
    {\bf U}_{m-1,M}^{0,\nu_1}-{\bf u}_{\nu_1}\right)\!(s,x+W_s^0-W_t^0)
    \right|
    }
    {\rhoh(t,s)}
     \left|
 \big\langle \uvec_\nu,     
     \big(
     1 ,
     \tfrac{ 
     W^{0}_{s}- W^{0}_{t}
     }{ {s - t} }
     \big)\big\rangle \right|
     \rhoh(t,s)\,
     ds
     \right]
     \\& =
    \sum_{\nu_1=1}^{d+1}L_{\nu_1}
    \E\!\left[
    \tfrac{
    \left| 
   \left( 
    {\bf U}_{m-1,M}^{0,\nu_1}-{\bf u}_{\nu_1}\right)\!
    (\uniform_t^0,x+W_{\uniform_t^0}^0-W_t^0)
    \right|
    }
    {\rhoh(t,\uniform_t^0)}
     \left|
 \Big\langle \uvec_\nu,     
     \Big(
     1 ,
     \tfrac{ 
     W^{0}_{\uniform_t^0}- W^{0}_{t}
     }{ {\uniform_t^0 - t} }
     \Big)\Big\rangle \right|
     \right]
     \\& \leq 
    \sum_{\nu_1=1}^{d+1}L_{\nu_1}
    \left\|
    \tfrac{
   \left( 
    {\bf U}_{m-1,M}^{0,\nu_1}-{\bf u}_{\nu_1}\right)\!
    (\uniform_t^0,x+W_{\uniform_t^0}^0-W_t^0)
    }
    {\rhoh(t,\uniform_t^0)}
     \Big\langle \uvec_\nu,
     \Big(
     1 ,
     \tfrac{ 
     W^{0}_{\uniform_t^0}- W^{0}_{t}
     }{ {\uniform_t^0 - t} }
     \Big)\Big\rangle 
     \right\|_{L^2(\P;\R)}.
  \end{split}     \end{equation}
  In the next step we combine the established bounds for the Monte Carlo error and for the time discretization error
  to obtain a bound for the
  \emph{global error}. More formally, observe that 
  \eqref{eq:final.MCerror} and \eqref{eq:final.Timeerror} ensure that 
  for all  $m\in\N$, $t\in[0,T)$, $x\in\R^d$, $\nu \in \{1,2,\ldots, d+1\}$
  it holds that
  \begin{equation}  \begin{split}\label{eq:global.estimate}
   &     \left\| {\bf U}_{m,M}^{0,\nu}(t,x)-{\bf u}_{\nu}(t,x)\right\|_{L^2(\P;\R)}\\
   & \leq \left\| {\bf U}_{m,M}^{0,\nu}(t,x)-\E\!\left[{\bf U}_{m,M}^{0,\nu}(t,x)\right]\right\|_{L^2(\P;\R)} 
   +\left|\E\!\left[{\bf U}_{m,M}^{0,\nu}(t,x)\right]
            -{\bf u}_\nu(t,x)\right|
\\&
    \leq 
     \tfrac{1}{\sqrt{M^m}}
     \left\|
     \left(g(x+W_T^{0}-W_t^0)-g(x)\right)
      \big\langle \uvec_\nu,
     \big(
     1 ,
     \tfrac{ 
     W^{0}_{T}- W^{0}_{t}
     }{ T - t }
     \big)\big\rangle 
     \right\|_{L^2(\P;\R)}\\
     &\quad
     +
    \tfrac{1}{\sqrt{M^m}}
    \left\|
    \tfrac{
     \left(\funcF( 0)\right)\!
     (\uniform^{0}_t,x+W_{\uniform_t^{0}}^0-W_t^0)
     }
      {\rhoh(t,\uniform^{0}_t)}
       \Big\langle \uvec_\nu,
     \Big(
     1 ,
     \tfrac{ 
     W^{0}_{\uniform_t^{0}}- W^{0}_{t}
     }{ {{\uniform_t^{0}} - t} }
     \Big)\Big\rangle
    \right\|_{L^2(\P;\R)}
    \\&
    \quad +
    \sum_{l=0}^{m-1}
    \sum_{\nu_1=1}^{d+1} \tfrac{ L_{\nu_1}}{\sqrt{M^{m-l-1}}}
    \left(
    \tfrac{\1_{(0,m)}(l)}{\sqrt{M}}+\1_{(-1,m-1)}(l)
    \right)
    \left\|
	\tfrac{     
     \left({\bf U}_{l,M}^{0,\nu_1}-{\bf u}_{\nu_1}\right)\!
     (\uniform_t^0,x+W_{\uniform_t^0}^0-W_t^0)
     }
     {
     \rhoh(t,\uniform_t^0)
     }
      \Big\langle \uvec_\nu,
     \Big(
     1 ,
     \tfrac{ 
     W^{0}_{\uniform_t^0}- W^{0}_{t}
     }{ {\uniform_t^0 - t} }
     \Big)\Big\rangle 
    \right\|_{L^2(\P;\R)}
    \\&
    \quad+\sum_{\nu_1=1}^{d+1}L_{\nu_1}
    \left\|
    \tfrac{
   \left( 
    {\bf U}_{m-1,M}^{0,\nu_1}-{\bf u}_{\nu_1}\right)\!
    (\uniform_t^0,x+W_{\uniform_t^0}^0-W_t^0)
    }
    {\rhoh(t,\uniform_t^0)}
     \Big\langle \uvec_\nu,
     \Big(
     1 ,
     \tfrac{ 
     W^{0}_{\uniform_t^0}- W^{0}_{t}
     }{ {\uniform_t^0 - t} }
     \Big)\Big\rangle 
     \right\|_{L^2(\P;\R)}
     \\&
    \leq 
     \tfrac{1}{\sqrt{M^m}}
     \left\|
     \left(g(x+W_T^{0}-W_t^0)-g(x)\right)
      \big\langle \uvec_\nu,
     \big(
     1 ,
     \tfrac{ 
     W^{0}_{T}- W^{0}_{t}
     }{ T - t }
     \big)\big\rangle 
     \right\|_{L^2(\P;\R)}\\
     &\quad
    + \tfrac{1}{\sqrt{M^m}}
    \left\|
    \tfrac{
     \left(\funcF( 0)\right)\!
     (\uniform^{0}_t,x+W_{\uniform_t^{0}}^0-W_t^0)
     }
      {\rhoh(t,\uniform^{0}_t)}
       \Big\langle \uvec_\nu,
     \Big(
     1 ,
     \tfrac{ 
     W^{0}_{\uniform_t^{0}}- W^{0}_{t}
     }{ {{\uniform_t^{0}} - t} }
     \Big)\Big\rangle 
    \right\|_{L^2(\P;\R)}
    \\&\quad 
     +
    \sum_{\nu_1=1}^{d+1} \tfrac{ L_{\nu_1}}{\sqrt{M^{m-1}}}
    \left\|
	\tfrac{     
     {\bf u}_{\nu_1}
     (\uniform_t^0,x+W_{\uniform_t^0}^0-W_t^0)
     }
     {
     \rhoh(t,\uniform_t^0)
     }
      \Big\langle \uvec_\nu,
     \Big(
     1 ,
     \tfrac{ 
     W^{0}_{\uniform_t^0}- W^{0}_{t}
     }{ {\uniform_t^0 - t} }
     \Big)\Big\rangle 
    \right\|_{L^2(\P;\R)}
    \\&\quad 
     +
    \sum_{l=1}^{m-1}
    \sum_{\nu_1=1}^{d+1} \tfrac{ 2 L_{\nu_1}}{\sqrt{M^{m-l-1}}}
    \left\|
	\tfrac{     
     \left( {\bf U}_{l,M}^{0,\nu_1}-{\bf u}_{\nu_1}\right)\!
     (\uniform_t^0,x+W_{\uniform_t^0}^0-W_t^0)
     }
     {
     \rhoh(t,\uniform_t^0)
     }
      \Big\langle \uvec_\nu,
     \Big(
     1 ,
     \tfrac{ 
     W^{0}_{\uniform_t^0}- W^{0}_{t}
     }{ {\uniform_t^0 - t} }
     \Big)\Big\rangle 
    \right\|_{L^2(\P;\R)}.
  \end{split}     \end{equation}
We next iterate this inequality. More precisely, we show that 
it holds
for all $n,k\in\N$, $t\in[0,T)$, $x\in\R^d$, $\nu_0 \in \{1,2,\ldots, d+1\}$ that
\begin{equation}  \begin{split}\label{eq:estimate.L2error2}
   &     \left\| {\bf U}_{{n},M}^{0,\nu_0}({t},x)-{\bf u}_{\nu_0}(t,x)\right\|_{L^2(\P;\R)}
     \\&
    \leq 
    \sum_{j=0}^{k-1}
    \sum_{\substack{l_1,l_2,\ldots,l_{j+1}\in\N,\\ l_1<l_2<\ldots<l_{j}<l_{j+1}=n}}
    \sum_{\substack{\nu_1,\nu_2,\ldots,\nu_{j+1} \in \\ \{1,2,\ldots,d+1\}}}
   \tfrac{\mathrm{1}_{\{1\}}(\nu_{j+1})2^j\left[\prod_{i=1}^j L_{\nu_i}\right]}{\sqrt{M^{n-j}}}   
    \Big\|\left(g(x+W^0_{T}-W^0_{t})-g(x+W^0_{S(j,t)}-W^0_{t})\right)\\
    &  \qquad \qquad \qquad \qquad \qquad \qquad
    \cdot  
     \big\langle \uvec_{\nu_j},
    \big(   
   1,
     \tfrac{ 
     W^{0}_{T}- W^{0}_{\ES(j,t)}
     }{T-\ES(j,t) } 
     \big)\big\rangle 
   \prod_{i=1}^{j}
   \tfrac{1}{\rhoh(\ES(i-1,t),\ES(i,t))}
    \big\langle \uvec_{\nu_{i-1}},
     \big(   
   1,
     \tfrac{ 
     W^{0}_{\ES(i,t)}- W^{0}_{\ES(i-1,t)}
     }{\ES(i,t)-\ES(i-1,t) } 
     \big)\big\rangle 
   \Big\|_{L^2(\P;\R)}\\
   &\quad+
    \sum_{j=0}^{k-1}
    \sum_{\substack{l_1,l_2,\ldots,l_{j+1}\in\N,\\ l_1<l_2<\ldots<l_{j}<l_{j+1}=n}}
    \sum_{\substack{\nu_1,\nu_2,\ldots,\nu_{j+1} \in \\ \{1,2,\ldots,d+1\}}}
   \tfrac{\mathrm{1}_{\{1\}}(\nu_{j+1})2^j\left[\prod_{i=1}^j L_{\nu_i}\right]}{\sqrt{M^{n-j}}}   
    \Big\|(\funcF (0))(\ES(j+1,t),x+W^0_{\ES(j+1,t)}-W^0_t)\\
    &\qquad \qquad \qquad \qquad \qquad \qquad \qquad \qquad \qquad \qquad
    \cdot  
   \prod_{i=1}^{j+1}
   \tfrac{1}{\rhoh(\ES(i-1,t),\ES(i,t))}
   \big\langle \uvec_{\nu_{i-1}},
     \big(   
   1,
     \tfrac{ 
     W^{0}_{\ES(i,t)}- W^{0}_{\ES(i-1,t)}
     }{\ES(i,t)-\ES(i-1,t) } 
     \big)\big\rangle 
   \Big\|_{L^2(\P;\R)}
   \\
   &\quad +
    \sum_{j=0}^{k-1}
    \sum_{\substack{l_1,l_2,\ldots,l_{j+1}\in\N,\\ l_1<l_2<\ldots<l_{j}<l_{j+1}=n}}
    \sum_{\substack{\nu_1,\nu_2,\ldots,\nu_{j+1} \in \\ \{1,2,\ldots,d+1\}}}
   \tfrac{2^j\left[\prod_{i=1}^{j+1} L_{\nu_i}\right]}{\sqrt{M^{n-j-1}}}   
    \Big\|{\bf u}_{\nu_{j+1}}(\ES(j+1,t),x+W^0_{\ES(j+1,t)}-W^0_t)\\
    &\qquad \qquad \qquad \qquad \qquad \qquad \qquad \qquad \qquad \qquad
    \cdot  
   \prod_{i=1}^{j+1}
   \tfrac{1}{\rhoh(\ES(i-1,t),\ES(i,t))}
   \big\langle \uvec_{\nu_{i-1}},
     \big(   
   1,
     \tfrac{ 
     W^{0}_{\ES(i,t)}- W^{0}_{\ES(i-1,t)}
     }{\ES(i,t)-\ES(i-1,t) } 
     \big)\big\rangle 
   \Big\|_{L^2(\P;\R)}\\
   &\quad+
   \sum_{\substack{l_1,l_2,\ldots,l_{k}\in\N,\\ l_1<l_2<\ldots<l_{k}<n}}
   \sum_{\substack{\nu_1,\nu_2,\ldots,\nu_{k} \in \\ \{1,2,\ldots,d+1\}}}
   \tfrac{2^k\left[\prod_{i=1}^{k} L_{\nu_i}\right]}{\sqrt{M^{n-k-l_1}}}  
    \Big\|
    \left( {\bf U}_{l_1,M}^{0,\nu_k}-{\bf u}_{\nu_k}\right)\!
     (\ES(k,t),x+W_{\ES(k,t)}^0-W_t^0)\\
     &\qquad \qquad \qquad \qquad \qquad \qquad \qquad \qquad \qquad \qquad
    \cdot  
   \prod_{i=1}^{k}
   \tfrac{1}{\rhoh(\ES(i-1,t),\ES(i,t))}
   \big\langle \uvec_{\nu_{i-1}},
     \big(   
   1,
     \tfrac{ 
     W^{0}_{\ES(i,t)}- W^{0}_{\ES(i-1,t)}
     }{\ES(i,t)-\ES(i-1,t) } 
     \big)\big\rangle 
   \Big\|_{L^2(\P;\R)}.
    \end{split}
    \end{equation}
  We prove~\eqref{eq:estimate.L2error2} by induction on $k\in\N$.
  The base case $k=1$ follows immediately from~\eqref{eq:global.estimate}.
  For the induction step $\N\ni k\rightsquigarrow k+1\in\N$ let $k\in\N$ and assume that~\eqref{eq:estimate.L2error2}
  holds for $k$.
  Inequality~\eqref{eq:global.estimate}
  and independence of $(W^0,\unif^{(m)},{\bf U}^0_{m,M})_{m\in\N_0}$
 (see Item~\eqref{properties_approx:item3} in Lemma~\ref{properties_approx})
  yield that
  for all  $l_1\in\N$, $x\in\R^d$, $t\in [0,T)$, $\nu_0,\nu_1,\ldots, \nu_{k} \in \{1,2,\ldots, d+1\}$
  it holds that
  \begin{equation}
  \begin{split}
 & \Big\|
    \left( {\bf U}_{l_1,M}^{0,\nu_k}-{\bf u}_{\nu_k}\right)\!
     (\ES(k,t),x+W_{\ES(k,t)}^0-W_t^0)
   \prod_{i=1}^{k}
   \tfrac{1}{\rhoh(\ES(i-1,t),\ES(i,t))}
   \big\langle \uvec_{\nu_{i-1}},
     \big(   
   1,
     \tfrac{ 
     W^{0}_{\ES(i,t)}- W^{0}_{\ES(i-1,t)}
     }{\ES(i,t)-\ES(i-1,t) } 
     \big)\big\rangle 
   \Big\|_{L^2(\P;\R)}\\
   &=
   \Big\|
    \left\|\left( {\bf U}_{l_1,M}^{0,\nu_k}-{\bf u}_{\nu_k}\right)\!
     (s,y)\right\|_{L^2(\P;\R)}
     \bigg|_{(s,y)=(\ES(k,t),x+W_{\ES(k,t)}^0-W_t^0)}
   \prod_{i=1}^{k}
   \tfrac{1}{\rhoh(\ES(i-1,t),\ES(i,t))}
   \big\langle \uvec_{\nu_{i-1}},
     \big(   
   1,
     \tfrac{ 
     W^{0}_{\ES(i,t)}- W^{0}_{\ES(i-1,t)}
     }{\ES(i,t)-\ES(i-1,t) } 
     \big)\big\rangle 
   \Big\|_{L^2(\P;\R)}\\
   &\leq 
   \tfrac{1}{\sqrt{M^{l_{1}}}}
   \left\|\left(g(x+W^0_{T}-W^0_{t})-g(x+W^0_{\ES(k,t)}-W^0_{t})\right)
   \big\langle \uvec_{\nu_{k}},
     \big(   
   1 ,
     \tfrac{ 
     W^{0}_{T}- W^{0}_{\ES(k,t)}
     }{ T-\ES(k,t) } 
     \big)
     \big\rangle
  \prod_{i=1}^{k}
   \tfrac{
  \big\langle \uvec_{\nu_{i-1}}, 
   \big(   
   1,
     \tfrac{ 
     W^{0}_{\ES(i,t)}- W^{0}_{\ES(i-1,t)}
     }{\ES(i,t)-\ES(i-1,t) } 
     \big)\big\rangle }
    {\rhoh(\ES(i-1,t),\ES(i,t))}
   \right\|_{L^2(\P;\R)}
   \\&
   +
   \tfrac{1}{\sqrt{M^{l_{1}}}}
   \left\|\left(\funcF(0)\right)(\ES(k+1,t),x+W_{\ES(k+1,t)}-W_{t}^0)
   \prod_{i=1}^{k+1}
   \tfrac{
 \big\langle \uvec_{\nu_{i-1}},  
   \big(   
   1,
     \tfrac{ 
     W^{0}_{\ES(i,t)}- W^{0}_{\ES(i-1,t)}
     }{\ES(i,t)-\ES(i-1,t) } 
     \big)\big\rangle}{\rhoh(\ES(i-1,t),\ES(i,t))}
   \right\|_{L^2(\P;\R)}
    \\&
   +
   \sum_{\nu_{k+1}=1}^{d+1}\tfrac{L_{\nu_{k+1}}}{\sqrt{M^{l_{1}-1}}}
   \left\|{\bf u}_{\nu_{k+1}}(\ES(k+1,t),x+W_{\ES(k+1,t)}-W_{t}^0)
   \prod_{i=1}^{k+1}
   \tfrac{
\big\langle \uvec_{\nu_{i-1}},   
   \big(   
   1,
     \tfrac{ 
     W^{0}_{\ES(i,t)}- W^{0}_{\ES(i-1,t)}
     }{\ES(i,t)-\ES(i-1,t) } 
     \big)\big\rangle}{\rhoh(\ES(i-1,t),\ES(i,t))}
   \right\|_{L^2(\P;\R)}
   \\&
   +
   \sum_{l_0=1}^{l_1-1}\sum_{\nu_{k+1}=1}^{d+1}\tfrac{2L_{\nu_{k+1}}}{\sqrt{M^{l_{1}-l_0-1}}}
   \left\|({\bf U}^{0,\nu_{k+1}}_{l_0,M}-{\bf u}_{\nu_{k+1}})(\ES(k+1,t),x+W_{\ES(k+1,t)}-W_{t}^0)
   \prod_{i=1}^{k+1}
   \tfrac{
\big\langle \uvec_{\nu_{i-1}},   
   \big(   
   1,
     \tfrac{ 
     W^{0}_{\ES(i,t)}- W^{0}_{\ES(i-1,t)}
     }{\ES(i,t)-\ES(i-1,t) } 
     \big)\big\rangle}{\rhoh(\ES(i-1,t),\ES(i,t))}
   \right\|_{L^2(\P;\R)}.
  \end{split}
  \end{equation}
  This and the induction hypothesis
complete the induction step $\N \ni k \rightsquigarrow k+1 \in \N$. Induction hence 
establishes~\eqref{eq:estimate.L2error2}. 
 Applying \eqref{eq:estimate.L2error2} with $k=n$ yields for all
$n\in\N$, $t\in[0,T)$, $x\in\R^d$, $\nu_0 \in \{1,2,\ldots, d+1\}$
  that
  \begin{equation}  \begin{split}
   &     \left\| {\bf U}_{{n},M}^{0,\nu_0}({t},x)-{\bf u}_{\nu_0}(t,x)\right\|_{L^2(\P;\R)}
     \\&
    \leq 
    \sum_{j=0}^{n-1}
    \sum_{\substack{l_1,l_2,\ldots,l_{j+1}\in\N,\\ l_1<l_2<\ldots<l_{j}<l_{j+1}=n}}
    \sum_{\substack{\nu_1,\nu_2,\ldots,\nu_{j+1} \in \\ \{1,2,\ldots,d+1\}}}
   \tfrac{\mathrm{1}_{\{1\}}(\nu_{j+1})2^j\left[\prod_{i=1}^j L_{\nu_i}\right]}{\sqrt{M^{n-j}}} 
    \Big\|\left(g(x+W^0_{T}-W^0_{t})-g(x+W^0_{S(j,t)}-W^0_{t})\right)\\
    &   \qquad \qquad \qquad \qquad \qquad
    \cdot  
    \big\langle \uvec_{\nu_{j}},
    \big(   
   1,
     \tfrac{ 
     W^{0}_{T}- W^{0}_{\ES(j,t)}
     }{T-\ES(j,t) } 
     \big)\big\rangle 
   \prod_{i=1}^{j}
   \tfrac{1}{\rhoh(\ES(i-1,t),\ES(i,t))}
   \big\langle \uvec_{\nu_{i-1}},
     \big(   
   1,
     \tfrac{ 
     W^{0}_{\ES(i,t)}- W^{0}_{\ES(i-1,t)}
     }{\ES(i,t)-\ES(i-1,t) } 
     \big)\big\rangle
   \Big\|_{L^2(\P;\R)}\\
   &\quad+
    \sum_{j=0}^{n-1}
    \sum_{\substack{l_1,l_2,\ldots,l_{j+1}\in\N,\\ l_1<l_2<\ldots<l_{j}<l_{j+1}=n}}
    \sum_{\substack{\nu_1,\nu_2,\ldots,\nu_{j+1} \in \\ \{1,2,\ldots,d+1\}}}
   \tfrac{\mathrm{1}_{\{1\}}(\nu_{j+1})2^j\left[\prod_{i=1}^j L_{\nu_i}\right]}{\sqrt{M^{n-j}}}   
    \Big\|(\funcF (0))(\ES(j+1,t),x+W^0_{\ES(j+1,t)}-W^0_t)\\
    &\qquad \qquad \qquad \qquad \qquad \qquad \qquad \qquad \qquad \qquad
    \cdot  
   \prod_{i=1}^{j+1}
   \tfrac{1}{\rhoh(\ES(i-1,t),\ES(i,t))}
   \big\langle \uvec_{\nu_{i-1}},
     \big(   
   1,
     \tfrac{ 
     W^{0}_{\ES(i,t)}- W^{0}_{\ES(i-1,t)}
     }{\ES(i,t)-\ES(i-1,t) } 
     \big)\big\rangle 
   \Big\|_{L^2(\P;\R)}
   \\
   &\quad+
    \sum_{j=0}^{n-1}
    \sum_{\substack{l_1,l_2,\ldots,l_{j+1}\in\N,\\ l_1<l_2<\ldots<l_{j}<l_{j+1}=n}}
    \sum_{\substack{\nu_1,\nu_2,\ldots,\nu_{j+1} \in \\ \{1,2,\ldots,d+1\}}}
   \tfrac{2^j\left[\prod_{i=1}^{j+1} L_{\nu_i}\right]}{\sqrt{M^{n-j-1}}}   
    \Big\|{\bf u}_{\nu_{j+1}}(\ES(j+1,t),x+W^0_{\ES(j+1,t)}-W^0_t)\\
    &\qquad \qquad \qquad \qquad \qquad \qquad \qquad \qquad \qquad \qquad
    \cdot  
   \prod_{i=1}^{j+1}
   \tfrac{1}{\rhoh(\ES(i-1,t),\ES(i,t))}
   \big\langle \uvec_{\nu_{i-1}},
     \big(   
   1,
     \tfrac{ 
     W^{0}_{\ES(i,t)}- W^{0}_{\ES(i-1,t)}
     }{\ES(i,t)-\ES(i-1,t) } 
     \big)\big\rangle 
   \Big\|_{L^2(\P;\R)}.
    \end{split}
    \end{equation}
This and the fact that for all $n\in \N$ and $j\in \{0,1,\ldots, n-1\}$ it holds that
\begin{equation}
 \sum_{\substack{l_1,l_2,\ldots,l_{j+1}\in\N,\\ l_1<l_2<\ldots<l_{j}<l_{j+1}=n}}1=\binom{n-1}{j}.
\end{equation}
proves \eqref{eq:estimate.L2error}.
  This finishes the proof of Lemma \ref{l:estimate.L2error}.
\end{proof}

\begin{proposition}[Global approximation error]\label{thm:rate}
Assume Setting~\ref{s:full.discretization},
let $t \in [0,T)$, $x\in \R^d$, $\nu_0\in \{1,2,\ldots,d+1\}$,
 $M,n\in \N$, $p\in [2,\infty)$,
$\alpha\in (\frac{p-2}{2(p-1)},\frac{p}{2(p-1)})$,
$\beta=\frac{\alpha}{2}-\frac{(1-\alpha)(p-2)}{2p}$,
  $C\in\R$ satisfy that
 \begin{equation}
 C=\max\left\{1,
  2(T-t)^\frac{1}{2}|\Gamma(\tfrac{p}{2})|^{\frac{1}{p}}(1-\alpha)^{\frac{1}{p}-1}
   \max\{1,\|L\|_1\}
 \max\left\{
      (T-t)^\frac{1}{2}, 2^\frac{1}{2}|\Gamma(\tfrac{p+1}{2})|^{\frac{1}{p}}\pi^{-\frac{1}{2p}}
      \right\}
\right\},
 \end{equation}
% let $\beta\in (0,\frac{\alpha}{2}]$ satisfy that ,
% assume that 
%   \begin{equation}\label{eq:approximations.integrable.assump3}
%\sup_{y\in \R^d, s\in [0,T)}
%    \frac{\left| 
%  \big(\funcF(0)\big)
%  (s,y)\right|}
%  {1+\|y\|^q_1}<\infty,
%\end{equation} 
and assume for all $s\in (0,1)$ that
$\rho(s)=\frac{1-\alpha}{s^\alpha}$.
 Then
  \begin{equation}  \begin{split}\label{eq:ub_error_thm}
   &     \big\| {\bf U}_{{n},M}^{0,\nu_0}({t},x)-{\bf u}_{\nu_0}(t,x)\big\|_{L^2(\P;\R)}
\leq
\tfrac{1}{4}\left[1+\tfrac{pn}{2}\right]^{\frac{1}{8}}M^{-\frac{n}{2}}(2C)^{n}
\exp\big(\tfrac{1}{8}+\beta M^{\frac{1}{2\beta}}\big)
%\\ &
\bigg[2C^{-1}\sqrt{ \max\{T-t,3\}}\|K\|_1
\\&
+
   \sup_{s\in [t,T)}
 \left\|
   (\funcF (0))(s,x+W^0_{s}-W^0_t)
   \right\|_{L^{\frac{2p}{p-2}}(\P;\R)}
   +  \sqrt{M}
\sup_{s\in [t,T)}\max_{i\in\{1,2,\ldots,d+1\}}
   \left\|
{\bf u}_i(s,x+W^0_{s}-W^0_t)
%\right\|_{L^{\frac{2p}{p-2}}(\P;\R)}
\right\|_{L^{\frac{2p}{p-2}}(\P;\R)}      
\bigg]
.
    \end{split}
    \end{equation}  
\end{proposition}
\begin{proof}[Proof of Proposition~\ref{thm:rate}]
Throughout this proof
let $C_1\in [0,\infty)$ satisfy
\begin{equation}
C_1=
      \max\left\{
      T-t, 2|\Gamma(\tfrac{p+1}{2})|^{\frac{2}{p}}\pi^{-\frac{1}{p}}
      \right\}
  (T-t)|\Gamma(\tfrac{p}{2})|^{\frac{2}{p}}(1-\alpha)^{\frac{2}{p}-2}.
\end{equation}
Without loss of generality we assume that
   $\sup_{s\in [t,T)} \left\| \left | (\funcF (0))(s,x+W^0_{s}-W^0_t) \right | \right\|_{L^{\frac{2p}{p-2}}(\P;\R)}<\infty$
   (otherwise the assertion is trivial).
It follows from \eqref{eq:fLipschitz} and the triangle inequality that
  for all $\nu\in \{1,2,\ldots,d+1\}$,  $s\in [0,T)$ it holds that
  \begin{equation}\label{eq:090119}
  \begin{split}
  &
  \sum_{\alpha=1}^{d}K_\alpha
    \left\|
    (W^{0,\alpha}_T-W^{0,\alpha}_{s})
\big \langle \uvec_\nu,    
         \big(   
   1,
     \tfrac{ 
     W^{0}_{T}- W^{0}_{s}
     }{ {T-s} } 
     \big)\big\rangle 
   \right\|_{L^2(\P;\R)}\\
   &=
   \sum_{\alpha=1}^{d}K_\alpha\left(\sqrt{T-s}\1_{\{1\}}(\nu)+\tfrac{\1_{[2,\infty)}(\nu)}{T-s}\|(W^{0,\alpha}_T-W^{0,\alpha}_s)(W^{0,\nu-1}_T-W^{0,\nu-1}_s)\|_{L^2(\P;\R)}\right)\\
   &=\sqrt{T-s}\|K\|_1\1_{\{1\}}(\nu)+\tfrac{\1_{[2,\infty)}(\nu)}{T-s}\left(K_{\nu-1}\|(W^{0,\nu-1}_T-W^{0,\nu-1}_s)^2\|_{L^2(\P;\R)}+ \sum_{\alpha\in \{1,2,\ldots,d\}\setminus \{\nu-1\}}K_\alpha \|W^{0,\alpha}_T-W^{0,\alpha}_s\|_{L^2(\P;\R)}^2\right)\\
   &=\sqrt{T-s}\|K\|_1\1_{\{1\}}(\nu)+\1_{[2,\infty)}(\nu)\left(\sqrt{3}K_{\nu-1}+ \sum_{\alpha\in \{1,2,\ldots,d\}\setminus \{\nu-1\}}K_\alpha\right)\\
   &\leq \max\{\sqrt{T-s},\sqrt{3}\}\|K\|_1.
   \end{split}
  \end{equation}
  Note that the fact that $p\in [2,\infty)$ and the fact that $\alpha\in \left(\frac{p-2}{2(p-1)},\frac{p}{2(p-1)}\right)$ imply that $\frac{p}{2}\in [\alpha(p-1),\alpha(p-1)+1]$ and that 
  $\alpha \in (0,1)$. 
  This, \eqref{eq:090119}, \eqref{eq:fLipschitz}, and Corollary \ref{cor:ub.it.int.lp} (with $p=2$ in the notation of  Corollary~\ref{cor:ub.it.int.lp}) show that
  for all $j\in \{0,1,\ldots,n-1\}$, $\nu_1,\nu_2,\ldots,\nu_j \in \{1,2,\ldots,d+1\}$ 
  it holds that
\begin{equation}\label{eq:estimate.L2error.g}
\begin{split}
&\left\|\left(g(x+W^0_{T}-W^0_{t})-g(x+W^0_{S(j,t)}-W^0_{t})\right) 
\big \langle \uvec_{\nu_j},    
    \big(   
   1,
     \tfrac{ 
     W^{0}_{T}- W^{0}_{\ES(j,t)}
     }{T-\ES(j,t) } 
     \big)\big\rangle 
   \prod_{i=1}^{j}
   \tfrac{1}{\rhoh(\ES(i-1,t),\ES(i,t))}
   \big \langle \uvec_{\nu_{i-1}},   
     \big(   
   1,
     \tfrac{ 
     W^{0}_{\ES(i,t)}- W^{0}_{\ES(i-1,t)}
     }{\ES(i,t)-\ES(i-1,t) } 
     \big)\big\rangle
   \right\|_{L^2(\P;\R)}\\
   &\le 
   \left\|
   \left(
   \sum_{\alpha=1}^d K_\alpha 
   \left|
   W^{0,\alpha}_T-W^{0,\alpha}_{\ES(j,t)}
   \right|
   \right) 
      \big \langle \uvec_{\nu_{j}},
    \big(   
   1,
     \tfrac{ 
     W^{0}_{T}- W^{0}_{\ES(j,t)}
     }{T-\ES(j,t) } 
     \big)\big\rangle 
   \prod_{i=1}^{j}
   \tfrac{1}{\rhoh(\ES(i-1,t),\ES(i,t))}
      \big \langle \uvec_{\nu_{i-1}},
     \big(   
   1,
     \tfrac{ 
     W^{0}_{\ES(i,t)}- W^{0}_{\ES(i-1,t)}
     }{\ES(i,t)-\ES(i-1,t) } 
     \big)\big\rangle
   \right\|_{L^2(\P;\R)}\\
   &= 
   \left\|
   \left\|
   \left(
   \sum_{\alpha=1}^d K_\alpha 
   \left|
   W^{0,\alpha}_T-W^{0,\alpha}_{s}
   \right|
   \right) 
      \big \langle \uvec_{\nu_{j}},
    \big(   
   1,
     \tfrac{ 
     W^{0}_{T}- W^{0}_{s}
     }{T-s } 
     \big)\big\rangle
     \right\|_{L^2(\P;\R)}
     \bigg|_{s=S(j,t)}
   \prod_{i=1}^{j}
   \tfrac{1}{\rhoh(\ES(i-1,t),\ES(i,t))}
      \big \langle \uvec_{\nu_{i-1}},
     \big(   
   1,
     \tfrac{ 
     W^{0}_{\ES(i,t)}- W^{0}_{\ES(i-1,t)}
     }{\ES(i,t)-\ES(i-1,t) } 
     \big)\big\rangle 
   \right\|_{L^2(\P;\R)}\\
    &\le 
   \left\|
   \left(
   \sum_{\alpha=1}^d K_\alpha
    \left\|
   \left(W^{0,\alpha}_T-W^{0,\alpha}_{s}\right)
      \big \langle \uvec_{\nu_{j}},
    \big(   
   1,
     \tfrac{ 
     W^{0}_{T}- W^{0}_{s}
     }{T-s } 
     \big)\big\rangle
     \right\|_{L^2(\P;\R)}
     \bigg|_{s=S(j,t)}
     \right)
   \prod_{i=1}^{j}
   \tfrac{1}{\rhoh(\ES(i-1,t),\ES(i,t))}
      \big \langle \uvec_{\nu_{i-1}},
     \big(   
   1,
     \tfrac{ 
     W^{0}_{\ES(i,t)}- W^{0}_{\ES(i-1,t)}
     }{\ES(i,t)-\ES(i-1,t) } 
     \big)\big\rangle 
   \right\|_{L^2(\P;\R)}\\
    &\le 
   \sqrt{ \max\{T-t,3\}}\|K\|_1
   \left\|
   \prod_{i=1}^{j}
   \tfrac{1}{\rhoh(\ES(i-1,t),\ES(i,t))}
      \big \langle \uvec_{\nu_{i-1}},
     \big(   
   1,
     \tfrac{ 
     W^{0}_{\ES(i,t)}- W^{0}_{\ES(i-1,t)}
     }{\ES(i,t)-\ES(i-1,t) } 
     \big)\big\rangle 
   \right\|_{L^2(\P;\R)}
   \\
   &\le 
    \sqrt{ \max\{T-t,3\}}\|K\|_1
        \left[
  \tfrac{(T-t)\max\left\{
      T-t, 2\frac{\Gamma(\frac{3}{2})}{\sqrt{\pi}}
      \right\}}{(1-\alpha)}
 \right]^{\frac{j}{2}}
(ej)^{\frac{1}{8}}
\left[
\tfrac{1}{\Gamma\left(j+1\right)}
\right]^{\frac{\alpha}{2}}.
\end{split}
\end{equation}
The facts that $\Gamma\left(\frac{p+1}{2}\right)\ge \Gamma\left( \frac{3}{2}\right)=\frac{\sqrt{\pi}}{2}$, that $\frac{2(p-1)}{p}\ge 1$, and that $\alpha\le 1$ prove that
\begin{equation}\label{eq:lb.C1}
  \tfrac{(T-t)\max\left\{
      T-t, 2\frac{\Gamma(\frac{3}{2})}{\sqrt{\pi}}
      \right\}}{(1-\alpha)}
\le C_1.
\end{equation}
Moreover, the fact that $p\ge 2$ ensures that for all $j\in \N_0$ it holds that
$(ej)^{\frac{1}{8}}\le \left[e\left(\tfrac{pj}{2}+1\right) \right]^{\frac{1}{8}}$ and that
$\Gamma(j+1)^{\frac{\alpha}{2}}\ge \Gamma(j+1)^{\beta}$. This together with
\eqref{eq:estimate.L2error.g} and \eqref{eq:lb.C1}
proves that for all $j\in \{0,1,\ldots,n-1\}$, $\nu_1,\nu_2,\ldots,\nu_j \in \{1,2,\ldots,d+1\}$ 
  it holds that
\begin{equation}\label{eq:estimate.L2error.g2}
\begin{split}
&\left\|\left(g(x+W^0_{T}-W^0_{t})-g(x+W^0_{S(j,t)}-W^0_{t})\right) 
   \big \langle \uvec_{\nu_{j}},
    \big(   
   1,
     \tfrac{ 
     W^{0}_{T}- W^{0}_{\ES(j,t)}
     }{T-\ES(j,t) } 
     \big)\big\rangle
   \prod_{i=1}^{j}
   \tfrac{1}{\rhoh(\ES(i-1,t),\ES(i,t))}
      \big \langle \uvec_{\nu_{i-1}},
     \big(   
   1,
     \tfrac{ 
     W^{0}_{\ES(i,t)}- W^{0}_{\ES(i-1,t)}
     }{\ES(i,t)-\ES(i-1,t) } 
     \big)\big\rangle
   \right\|_{L^2(\P;\R)}\\
      &\le 
    \sqrt{ \max\{T-t,3\}}\|K\|_1   
 C_1^{\frac{j}{2}}
\left[e\left(\tfrac{pj}{2}+1\right) \right]^{\frac{1}{8}}
\left[
\tfrac{1}{\Gamma\left(j+1\right)}
\right]^{\beta}.
\end{split}
\end{equation}
Corollary \ref{cor:ub.it.int.lp} and the facts that $p\ge 2$ and $\alpha>\frac{p-2}{2(p-1)}$ prove that
\begin{equation}\label{eq:ub.it.int.lp.c}
\begin{split}
&\left\|
\prod_{i=1}^{j+1}
   \tfrac{1}{\rhoh(\ES(i-1,t),\ES(i,t))}
      \big \langle \uvec_{\nu_{i-1}},
     \big(   
   1,
     \tfrac{ 
     W^{0}_{\ES(i,t)}- W^{0}_{\ES(i-1,t)}
     }{\ES(i,t)-\ES(i-1,t) } 
     \big)\big\rangle
   \right\|_{L^p(\P;\R)}
     \\
      &\le 
        \left[
      \max\left\{
      T-t, 2\tfrac{\Gamma\!\left(\frac{p+1}{2}\right)^{\frac{2}{p}}}{\pi^{\frac{1}{p}}}
      \right\}
  \tfrac{(T-t)\Gamma(\frac{p}{2})^{\frac{2}{p}}}{(1-\alpha)^{\frac{2(p-1)}{p}}
(\frac{p}{2})^{\frac{2(\alpha(p-1)-\frac{p}{2}+1)}{p}} }\right]^{\frac{j+1}{2}}
\left[e^{\frac{p}{2}}\left(\tfrac{pj}{2}+1\right) \right]^{\frac{1}{2p^2}}
\left[
\tfrac{\Gamma\left(\frac{2}{p}\right)}{\Gamma\left(1+j+\frac{2}{p}\right)}
\right]^{\frac{\alpha}{2}-\frac{(1-\alpha)(p-2)}{2p}}\\
&\le 
        \left[
      \max\left\{
      T-t, 2\tfrac{\Gamma\!\left(\frac{p+1}{2}\right)^{\frac{2}{p}}}{\pi^{\frac{1}{p}}}
      \right\}
  \tfrac{(T-t)\Gamma(\frac{p}{2})^{\frac{2}{p}}}{(1-\alpha)^{\frac{2(p-1)}{p}}
 }\right]^{\frac{j+1}{2}}
\left[e\left(\tfrac{pj}{2}+1\right) \right]^{\frac{1}{8}}
\left[
\tfrac{1}{\Gamma\left(1+j+\frac{2}{p}\right)}
\right]^{\frac{\alpha}{2}-\frac{(1-\alpha)(p-2)}{2p}}\\
&=
C_1^{\frac{j+1}{2}}
\left[e\left(\tfrac{pj}{2}+1\right) \right]^{\frac{1}{8}}
\left[
\tfrac{1}{\Gamma\left(1+j+\frac{2}{p}\right)}
\right]^{\beta}
  \end{split}
\end{equation}
This together with H\"older's inequality and independence of $(\mathfrak{r}^{(n)})_{n\in\N}$ and $W^0$ proves that for all
   $j\in \{0,1,\ldots,n-1\}$, $\nu_1,\nu_2,\ldots,\nu_{j+1} \in \{1,2,\ldots,d+1\}$ it holds that
\begin{equation}\label{eq:estimate.L2error.F0}
\begin{split}
 &\left\|(\funcF (0))(\ES(j+1,t),x+W^0_{\ES(j+1,t)}-W^0_t)
   \prod_{i=1}^{j+1}
   \tfrac{1}{\rhoh(\ES(i-1,t),\ES(i,t))}
      \big \langle \uvec_{\nu_{i-1}},
     \big(   
   1,
     \tfrac{ 
     W^{0}_{\ES(i,t)}- W^{0}_{\ES(i-1,t)}
     }{\ES(i,t)-\ES(i-1,t) } 
     \big)\big\rangle
   \right\|_{L^2(\P;\R)}\\
   &\leq
   \left\|
   (\funcF (0))(\ES(j+1,t),x+W^0_{\ES(j+1,t)}-W^0_t)
   \right\|_{L^{\frac{2p}{p-2}}(\P;\R)}
   \left\|
   \prod_{i=1}^{j+1}
   \tfrac{1}{\rhoh(\ES(i-1,t),\ES(i,t))}
      \big \langle \uvec_{\nu_{i-1}},
     \Big(   
   1,
     \tfrac{ 
     W^{0}_{\ES(i,t)}- W^{0}_{\ES(i-1,t)}
     }{\ES(i,t)-\ES(i-1,t) } 
     \big)\big\rangle
   \right\|_{L^p(\P;\R)}\\
   &=
   %\sup_{s\in [t,T)}
   \left(\int_t^T \E\Big[\Big\|(\funcF(0))(s,x+W^0_s-W^0_t)\Big\|^{\frac{2p}{p-2}}\Big]\P\Big(S(j+1,t)\in ds\Big)\right)^{\frac{p-2}{2p}}
%   \left\|
%   (\funcF (0))(s,x+W^0_{s}-W^0_t)
%   \right\|_{L^{\frac{2p}{p-2}}(\P;\R)}
  \\&\qquad\cdot
   \left\|
   \prod_{i=1}^{j+1}
   \tfrac{1}{\rhoh(\ES(i-1,t),\ES(i,t))}
      \big \langle \uvec_{\nu_{i-1}},
     \big(   
   1,
     \tfrac{ 
     W^{0}_{\ES(i,t)}- W^{0}_{\ES(i-1,t)}
     }{\ES(i,t)-\ES(i-1,t) } 
     \big)\big\rangle
   \right\|_{L^p(\P;\R)}\\
   &\le 
   \sup_{s\in [t,T)}
    \left\|
   (\funcF (0))(s,x+W^0_{s}-W^0_t)
   \right\|_{L^{\frac{2p}{p-2}}(\P;\R)}
\left[e\left(\tfrac{pj}{2}+1\right) \right]^{\frac{1}{8}}
C_1^{\frac{j+1}{2}}
\left[
\tfrac{1}{\Gamma\left(1+j+\frac{2}{p}\right)}
\right]^{\beta}
\end{split}
\end{equation}  
and, analogously,
  \begin{equation}\label{eq:estimate.L2error.uinfty}
  \begin{split}
  &\left\|{\bf u}_{\nu_{j+1}}(\ES(j+1,t),x+W^0_{\ES(j+1,t)}-W^0_t)
   \prod_{i=1}^{j+1}
   \tfrac{1}{\rhoh(\ES(i-1,t),\ES(i,t))}
      \big \langle \uvec_{\nu_{i-1}},
     \big(   
   1,
     \tfrac{ 
     W^{0}_{\ES(i,t)}- W^{0}_{\ES(i-1,t)}
     }{\ES(i,t)-\ES(i-1,t) } 
     \big)\big\rangle
   \right\|_{L^2(\P;\R)}
%   \\ &\le 
%\left\|
%\left({\bf u}(\ES(j+1,t),x+W^0_{\ES(j+1,t)}-W^0_t)\right)_{\nu_{j+1}}
%\right\|_{L^{\frac{2p}{p-2}}(\P;\R)}
%   \left\|
%   \prod_{i=1}^{j+1}
%   \tfrac{1}{\rhoh(\ES(i-1,t),\ES(i,t))}
%     \Big(   
%   1,
%     \tfrac{ 
%     W^{0}_{\ES(i,t)}- W^{0}_{\ES(i-1,t)}
%     }{\ES(i,t)-\ES(i-1,t) } 
%     \Big)_{\nu_{i-1}}
%   \right\|_{L^p(\P;\R)}\\
%   &\le 
%\sup_{s\in [t,T)}
%\left\|
%\left(
%{\bf u}(s,x+W^0_{s}-W^0_t)
%\right)_{\nu_{j+1}}
%\right\|_{L^{\frac{2p}{p-2}}(\P;\R)}
%   \left\|
%   \prod_{i=1}^{j+1}
%   \tfrac{1}{\rhoh(\ES(i-1,t),\ES(i,t))}
%     \Big(   
%   1,
%     \tfrac{ 
%     W^{0}_{\ES(i,t)}- W^{0}_{\ES(i-1,t)}
%     }{\ES(i,t)-\ES(i-1,t) } 
%     \Big)_{\nu_{i-1}}
%   \right\|_{L^p(\P;\R)}
\\ &\leq 
\sup_{s\in [t,T)}
   \left\|
{\bf u}_{\nu_{j+1}}(s,x+W^0_{s}-W^0_t)
\right\|_{L^{\frac{2p}{p-2}}(\P;\R)}
   \left[e\left(\tfrac{pj}{2}+1\right) \right]^{\frac{1}{8}}
C_1^{\frac{j+1}{2}}
\left[
\tfrac{1}{\Gamma\left(1+j+\frac{2}{p}\right)}
\right]^{\beta}.
  \end{split}
  \end{equation}
  Next we apply Lemma~\ref{l:estimate.L2error}. To this end note that 
  %the fact that
   $\sup_{s\in [t,T)} \left\| \left | (\funcF (0))(s,x+W^0_{s}-W^0_t) \right | \right\|_{L^{\frac{2p}{p-2}}(\P;\R)}<\infty$
%   \eqref{eq:approximations.integrable.assump3}
   and $\tfrac{2p}{p-2}\geq 2$
   ensure that
   for all $r\in [1,2)$ it holds that
  \begin{equation}
  \sup_{s\in [t,T)}
  \E\!\left[
  \left|
  \big(\funcF(0)\big)
  (s,x+W_{s}^0-W_t^0)
  \right|^{r}
  \right]<\infty.
  \end{equation}
  Moreover, it holds for all $r\in [1,2)$ that
\begin{equation}
\int_0^1
  \frac{1}{
  s^{\frac{r}{2}}
\left[\rho(s)\right]^{r-1}  
}
  \,ds= \frac{1}{(1-\alpha)^{r-1}}\int_0^1 s^{-(r(\frac{1}{2}-\alpha)+\alpha)}\, ds <\infty.
\end{equation}  
Combing Lemma~\ref{l:estimate.L2error}, \eqref{eq:estimate.L2error.g2}, 
\eqref{eq:estimate.L2error.F0}, and \eqref{eq:estimate.L2error.uinfty} proves
  that
  \begin{equation}  \begin{split}\label{eq:glob.error.aux1}
   &     \left\| {\bf U}_{{n},M}^{0,\nu_0}({t},x)-{\bf u}_{\nu_0}(t,x)\right\|_{L^2(\P;\R)}
     \\&
    \leq 
    \sqrt{ \max\{T-t,3\}}\|K\|_1
    \sum_{j=0}^{n-1}
    \sum_{\substack{\nu_1,\nu_2,\ldots,\nu_{j+1} \in\\ \{1,2,\ldots,d+1\}}}
    \tbinom{n-1}{j}
   \tfrac{\mathrm{1}_{\{1\}}(\nu_{j+1})\left[e\left(\tfrac{pj}{2}+1\right) \right]^{\frac{1}{8}}2^j
C_1^{\frac{j}{2}}   
   \left[\prod_{i=1}^j L_{\nu_i}\right]}{\sqrt{M^{n-j}}}   
\left[
\tfrac{1}{\Gamma\left(j+1\right)}
\right]^{\beta}\\
   &+
   \sup_{s\in [t,T)}
   \left\|
   (\funcF (0))(s,x+W^0_{s}-W^0_t)
   \right\|_{L^{\frac{2p}{p-2}}(\P;\R)}
    \sum_{j=0}^{n-1}
    \sum_{\substack{\nu_1,\nu_2,\ldots,\nu_{j+1} \in\\ \{1,2,\ldots,d+1\}}}
    \tbinom{n-1}{j}
   \tfrac{\mathrm{1}_{\{1\}}(\nu_{j+1})\left[e\left(\tfrac{pj}{2}+1\right) \right]^{\frac{1}{8}} 2^j C_1^{\frac{j+1}{2}} \left[\prod_{i=1}^j L_{\nu_i}\right]}{\sqrt{M^{n-j}}   
   \left(\Gamma\left(1+j+\frac{2}{p}\right)\right)^{\!\beta}}
%\left[
%\tfrac{1}{\Gamma\left(1+j+\frac{2}{p}\right)}
%\right]^{\beta}
   \\
   &+
    \sum_{j=0}^{n-1}
    \sum_{\substack{\nu_1,\nu_2,\ldots,\nu_{j+1} \in\\ \{1,2,\ldots,d+1\}}}
    \tbinom{n-1}{j}
\sup_{s\in [t,T)}
   \left\|
{\bf u}_{\nu_{j+1}}(s,x+W^0_{s}-W^0_t)
\right\|_{L^{\frac{2p}{p-2}}(\P;\R)}
   \tfrac{\left[e\left(\tfrac{pj}{2}+1\right) \right]^{\frac{1}{8}} 2^j C_1^{\frac{j+1}{2}} \left[\prod_{i=1}^{j+1} L_{\nu_i}\right]}{\sqrt{M^{n-j-1}}
   \left(\Gamma\left(1+j+\frac{2}{p}\right)\right)^{\!\beta}} .
%\left[
%\tfrac{1}{}
%\right]^{\beta}.
    \end{split}
    \end{equation}
Observe that for all $j\in \N_0$ it holds that
  \begin{equation}
  \left[
    \sum_{\substack{\nu_1,\nu_2,\ldots,\nu_{j+1} \in \\ \{1,2,\ldots,d+1\}}}
    \mathrm{1}_{\{1\}}(\nu_{j+1})
    \prod_{i=1}^{j}L_{\nu_i}
    \right]=\|L\|_1^j.
  \end{equation}
  Combining this with \eqref{eq:glob.error.aux1}
  %, \eqref{eq:ub.it.int}, and \eqref{eq:ub.it.int2} 
proves
  that
   \begin{equation}  \begin{split}\label{eq:glob.error.aux2}
   &     \left\| {\bf U}_{{n},M}^{0,\nu_0}({t},x)-{\bf u}_{\nu_0}(t,x)\right\|_{L^2(\P;\R)}
     \\&
    \leq 
    \sqrt{ \max\{T-t,3\}}\|K\|_1
    \sum_{j=0}^{n-1}
   \tfrac{\left[e\left(\tfrac{pj}{2}+1\right) \right]^{\frac{1}{8}}2^j
C_1^{\frac{j}{2}}   
   \|L\|_1^j}{\sqrt{M^{n-j}}}   
    \tbinom{n-1}{j}
\left[
\tfrac{1}{\Gamma\left(j+1\right)}
\right]^{\beta}\\
   &\quad+
   \sup_{s\in [t,T)}
   \left\|
   (\funcF (0))(s,x+W^0_{s}-W^0_t)
   \right\|_{L^{\frac{2p}{p-2}}(\P;\R)}
    \sum_{j=0}^{n-1}
   \tfrac{ \left[e\left(\tfrac{pj}{2}+1\right) \right]^{\frac{1}{8}}
   2^j\|L\|_1^jC_1^{\frac{j+1}{2}}}{\sqrt{M^{n-j}}}   
    \tbinom{n-1}{j}
\left[
\tfrac{1}{\Gamma\left(1+j+\frac{2}{p}\right)}
\right]^{\beta}
   \\
   &\quad+
\sup_{s\in [t,T),i\in\{1,2,\ldots,d+1\}}
   \left\|
{\bf u}_i(s,x+W^0_{s}-W^0_t)
\right\|_{L^{\frac{2p}{p-2}}(\P;\R)}
    \sum_{j=0}^{n-1}
   \tfrac{\left[e\left(\tfrac{pj}{2}+1\right) \right]^{\frac{1}{8}}
   2^j\|L\|_1^{j+1} C_1^{\frac{j+1}{2}}}{\sqrt{M^{n-j-1}}}     
   \tbinom{n-1}{j}
\left[
\tfrac{1}{\Gamma\left(1+j+\frac{2}{p}\right)}
\right]^{\beta}
    \end{split}
    \end{equation}
    This and the facts that $2\sqrt{C_1}\|L\|_1\le 2\sqrt{C_1}\max\{1,\|L\|_1\}\le C$, $p\ge 2$, and $C\ge 1$ imply that
    \begin{equation}  \begin{split}\label{eq:glob.error.aux3}
   &     \left\| {\bf U}_{{n},M}^{0,\nu_0}({t},x)-{\bf u}_{\nu_0}(t,x)\right\|_{L^2(\P;\R)}
     \\&
    \leq 
    \sqrt{ \max\{T-t,3\}}\|K\|_1
    \sum_{j=0}^{n-1}
   \tfrac{\left[e\left(\tfrac{pj}{2}+1\right) \right]^{\frac{1}{8}}
   C^j 
   }{\sqrt{M^{n-j}}}  
    \tbinom{n-1}{j} 
\left[
\tfrac{1}{\Gamma\left(j+1\right)}
\right]^{\beta}\\
   &\quad +
  \frac{1}{2}
   \sup_{s\in [t,T)}
  \left\|
   (\funcF (0))(s,x+W^0_{s}-W^0_t)
   \right\|_{L^{\frac{2p}{p-2}}(\P;\R)}
    \sum_{j=0}^{n-1}
   \tfrac{ \left[e\left(\tfrac{pj}{2}+1\right) \right]^{\frac{1}{8}}
   C^{j+1}}{\sqrt{M^{n-j}}}   
      \tbinom{n-1}{j}
\left[
\tfrac{1}{\Gamma\left(j+1\right)}
\right]^{\beta}
   \\
   &\quad +
\frac{1}{2}   
\sup_{s\in [t,T),i\in\{1,2,\ldots,d+1\}}
   \left\|
{\bf u}_i(s,x+W^0_{s}-W^0_t)
\right\|_{L^{\frac{2p}{p-2}}(\P;\R)}
    \sum_{j=0}^{n-1}
   \tfrac{\left[e\left(\tfrac{pj}{2}+1\right) \right]^{\frac{1}{8}}
   C^{j+1}}{\sqrt{M^{n-j-1}}}     
    \tbinom{n-1}{j}
\left[
\tfrac{1}{\Gamma\left(j+1\right)}
\right]^{\beta}\\
&
\leq
\tfrac{\left[e\left(\tfrac{pn}{2}+1\right) \right]^{\frac{1}{8}}
C^{n-1}
}{\sqrt{M^{n-1}}}
\left[
\sum_{j=0}^{n-1}
    \tbinom{n-1}{j} 
\tfrac{(\sqrt{M})^j}{\Gamma\left(j+1\right)^\beta}
\right]
\bigg[
\tfrac{\sqrt{ \max\{T-t,3\}}\|K\|_1}{\sqrt{M}}
+
 \tfrac{C
   \sup_{s\in [t,T)}
 \left\|
   (\funcF (0))(s,x+W^0_{s}-W^0_t)
   \right\|_{L^{\frac{2p}{p-2}}(\P;\R)}}{2\sqrt{M}} 
\\&\qquad\qquad\ \qquad\qquad \qquad\qquad\qquad\qquad \qquad\qquad
   +
   \tfrac{C
\sup_{s\in [t,T),i\in\{1,2,\ldots,d+1\}}
   \left\|
{\bf u}_i(s,x+W^0_{s}-W^0_t)
%\right\|_{L^{\frac{2p}{p-2}}(\P;\R)}
\right\|_{L^{\frac{2p}{p-2}}(\P;\R)}}{2}      
\bigg].
    \end{split}
    \end{equation}
     Next note that for all $j\in \N_0$ it holds that
    \begin{equation}\label{eq:geogamma}
    \frac{(\sqrt{M})^j}{\Gamma(j+1)^{\beta}}
    =
    \left(
   \frac{(M^{\frac{1}{2\beta}})^{j}}{\Gamma(j+1)}   
   \right)^{\beta}   
\le 
    \left(
    \sum_{k=0}^\infty
   \frac{(M^{\frac{1}{2\beta}})^k}{\Gamma(k+1)}   
   \right)^{\beta}   
   =\exp\left(\beta M^{\frac{1}{2\beta}}\right)
    \end{equation}
    and that
    \begin{equation}\label{eq:sumbinom}
    \sum_{j=0}^{n-1}
    \tbinom{n-1}{j} = 2^{n-1}.
    \end{equation}
Combining \eqref{eq:glob.error.aux2}, \eqref{eq:geogamma}, and \eqref{eq:sumbinom} shows that
 \begin{equation}  \begin{split}\label{eq:glob.error.aux4}
   &     \left\| {\bf U}_{{n},M}^{0,\nu_0}({t},x)-{\bf u}_{\nu_0}(t,x)\right\|_{L^2(\P;\R)}
\leq
\tfrac{\left[e\left(\tfrac{pn}{2}+1\right) \right]^{\frac{1}{8}}
(2C)^{n-1}
\exp\left(\beta M^{\frac{1}{2\beta}}\right)
}{\sqrt{M^{n-1}}}
%\\ &
\cdot
\bigg[
\tfrac{\sqrt{ \max\{T-t,3\}}\|K\|_1}{\sqrt{M}}
\\&\quad\quad
+
 \tfrac{C
   \sup_{s\in [t,T)}
 \left\|
   (\funcF (0))(s,x+W^0_{s}-W^0_t)
   \right\|_{L^{\frac{2p}{p-2}}(\P;\R)}}{2\sqrt{M}} 
   +
   \tfrac{C
\sup_{s\in [t,T),i\in\{1,2,\ldots,d+1\}}
   \left\|
{\bf u}_i(s,x+W^0_{s}-W^0_t)
%\right\|_{L^{\frac{2p}{p-2}}(\P;\R)}
\right\|_{L^{\frac{2p}{p-2}}(\P;\R)}}{2}      
\bigg]
.
    \end{split}
    \end{equation}    
  This completes the proof of Proposition~\ref{thm:rate}.
  \end{proof}
  
  \section{Regularity analysis for solutions of certain differential equations}\label{sec:ex_sol}

The error analysis in Subsection~\ref{subsec:error_analysis} above provides upper bounds for the approximation errors of the MLP approximations in~\eqref{eq:def:U}. The established upper bounds contain certain norms of the unknown exact solutions of the PDEs which we intend to approximate; see, e.g., the right-hand side of \eqref{eq:ub_error_thm} in Proposition~\ref{thm:rate} in Subsection~\ref{subsec:error_analysis} above for details. 
In Lemma~\ref{upper_exact} below we establish suitable upper bounds for these norms of the unknown exact solutions of the PDEs which we intend to approximate. In our proof of Lemma~\ref{upper_exact} we employ certain a priori
estimates for solutions of BSDEs which we establish in the essentially well-known result in Lemma~\ref{upper_exact} below (see, e.g., El Karoui et al.\ \cite[Proposition~2.1 and Equation~(2.12)]{ElKarouiPengQuenez1997} for results related to Lemma~\ref{upper_exact} below).

\subsection{Regularity analysis for solutions of backward stochastic differential equations (BSDEs)}

%%%%
%
%
% Where is 'from viscosity solution to BSDE solution'? Maybe Ma and Zhang
\begin{lemma}\label{lem:grad_bound}
%For every
%topological space $(E,\mathcal E)$ let $\mathcal{B}(E)$
%be the Borel-sigma-algebra on $(E,\mathcal E)$,
Let $T\in(0,\infty)$,
$t\in[0,T]$,
$d\in\N$,
$L_1,L_2,\ldots,L_{d+1}\in [0,\infty)$,
let $\left\|\cdot\right\|_2\colon \R^d \to [0,\infty)$ be the $d$-dimensional Euclidean norm,
let $(\Omega,\mathcal{F},\P)$ be a probability space
with a normal filtration $(\mathbb{F}_t)_{t\in[t,T]}$,
%which satisfies
%$\{A\in\mathcal{F}\colon \P(A)=0\}\subseteq\mathbb{F}_0$
let $f,\tilde f \colon [t,T] \times \R \times \R^d \times\Omega \to \R$
be functions satisfying that for all $s\in [t,T]$ the function
$[t,s] \times \R \times \R^d \times\Omega \ni (u,y,z, \omega) \mapsto (f(u,y,z,\omega),\tilde f(u,y,z,\omega))\in \R^2$ is 
$(\mathcal B([t,s]) \otimes \mathcal{B}(\R) \otimes \mathcal{B}(\R^d)\otimes \mathbb{F}_s)$/$\mathcal{B}(\R^2)$ is measurable,
assume
that for all $s\in[t,T]$, $y,\tilde{y}\in\R$, $z,\tilde{z}\in\R^d$
it holds $\P$-a.s.\ that
\begin{equation}  \begin{split}\label{eq:fLipBSDE}
 |f(s,y,z)-f(s,\tilde{y},\tilde{z})|\leq L_1|y-\tilde{y}|+\sum_{j=1}^dL_{j+1}|z_j-\tilde{z}_j|,
\end{split}     \end{equation}
let $Y,\tilde{Y}\colon[t,T]\times\Omega\to\R$,
$W\colon[t,T]\times\Omega\to\R^d$
be $(\mathbb{F}_s)_{s\in[t,T]}$-adapted processes with continuous sample paths,
assume that $(W_{s+t}-W_{t})_{s\in[0,T-t]}$ is a standard Brownian motion,
let
$Z,\tilde{Z}\colon[t,T]\times\Omega\to\R^d$
be $(\mathbb{F}_s)_{s\in[t,T]}$-adapted $(\mathcal{B}([t,T])\otimes\mathcal{F})$/$\mathcal{B}(\R^d)$-measurable
processes,
assume that it holds $\P$-a.s.\ that
$\int_t^T |f(s,Y_s,Z_s)|\,ds<\infty$, $\int_t^T |\tilde f(s,\tilde Y_s,\tilde Z_s)|\,ds<\infty$, 
$\int_t^T\|Z_s\|_2^2\,ds<\infty$, and $\int_t^T\|\tilde Z_s\|_2^2\,ds<\infty$, 
assume that
for all $s\in[t,T]$ it holds $\P$-a.s.\ that
\begin{equation}  \begin{split}\label{eq:Y}
  Y_s&=Y_T+\int_s^T f(u,Y_u,Z_u)\,du-\int_s^T \big(Z_u\big)^T\,dW_u,\\
  \tilde{Y}_s&=\tilde{Y}_T+\int_s^T \tilde f(u,\tilde{Y}_u,\tilde{Z}_u)\,du-\int_s^T \big(\tilde{Z}_u\big)^T\,dW_u,
\end{split}     \end{equation}
and that $\E\!\left[\sup_{s\in [t,T]}Y^2_{s}\right]<\infty$ and $\E\!\left[\sup_{s\in [t,T]}\tilde{Y}_{s}^2\right]<\infty$.
  Then it holds $\P$-a.s.\ that
  \begin{equation}\label{eq:grad_bound}
  \left|Y_{t}-\tilde{Y}_{t}\right|\le
     e^{L_1(T-t)}\left(\|Y_T-\tilde{Y}_T\|_{L^{\infty}(\P;\R)}+(T-t)\sup_{s\in [t,T], y\in \R, z\in \R^d}\|f(s,y,z)-\tilde f(s,y,z)\|_{L^\infty(\P;\R)}\right).
  \end{equation}
\end{lemma}
\begin{proof}[Proof of Lemma~\ref{lem:grad_bound}]
Throughout the proof let
$
  \langle \cdot, \cdot \rangle \colon
    \R^d \times \R^d
  \to
  [0,\infty)
$
the function that satisfies 
for all $ v = ( v_1, \dots, v_d ) $, $ w = ( w_1, \dots, w_d ) \in \R^d $
that
$
  \langle
    v, w
  \rangle
    =
   \sum_{i=1}^d v_i w_i
$.
Without loss of generality we suppose that 
$\sup_{s\in [t,T], y\in \R, z\in \R^d}\|f(s,y,z)-\tilde f(s,y,z)\|_{L^\infty(\P;\R)}<\infty$.
  Throughout this proof let
  $\left\|\cdot\right\|\colon\R^d\to[0,\infty)$,
  be the Euclidean norm,
  let
  $\langle\cdot,\cdot\rangle\colon\R^d\times\R^d\to\R$
  be the Euclidean scalar product,
  let
  $A\colon[t,T]\times\Omega\to\R$,
  $B\colon[t,T]\times\Omega\to\R^d$
  be the functions which satisfy for all $s\in[t,T]$, $j\in\{1,2,\ldots,d\}$ that
 \begin{align}
  \label{eq:A}
    A_s&=\begin{cases}\frac{f(s,Y_s,\tilde{Z}_s)-f(s,\tilde{Y}_s,\tilde{Z}_s)}{Y_s-\tilde{Y}_s}&\text{if $Y_s\neq \tilde{Y}_s$}\\
                      0&\text{else}
        \end{cases},\\
  \label{eq:B}
    B_s(j)&=\begin{cases}\frac{f(s,Y_s,Z_s(1),\ldots,Z_s(j),\tilde{Z}_s(j+1)\ldots,\tilde{Z}_s(d))-f(s,Y_s,Z_s(1),\ldots,Z_s(j-1),\tilde{Z}_s(j)\ldots,\tilde{Z}_s(d))}{Z_s(j)-\tilde{Z}_s(j)}&\text{if $Z_s(j)\neq \tilde{Z}_s(j)$}\\
                      0&\text{else}
        \end{cases},
  \end{align}
  and let $\Gamma\colon[t,T]\times\Omega\to\R$ be the function that satisfies 
  for all $s\in [t,T]$ that
  $\Gamma_s=e^{\int_t^s A_r-\frac{\|B_r\|^2}{2} dr
    +\int_t^s\langle B_r,dW_r\rangle}$.
   It\^o's formula implies that for all $s\in [t,T]$ it holds $\P$-a.s.\ that
  \begin{equation}\label{eq:Gamma}
  \Gamma_s=1+\int_{t}^s \Gamma_rA_r\,dr+\int_{t}^s\Gamma_r \langle B_r,dW_r\rangle.
  \end{equation}
  Then It\^o's formula, \eqref{eq:Y}, and \eqref{eq:Gamma}
  yield that
   for all $u\in [t,T)$
  it holds $\P$-a.s.\ that
  \begin{equation}  \begin{split}\label{eq:Ito1}
    &\Gamma_u Y_u-Y_{t}
    =\int_{t}^u \Gamma_s
    \Big(-f\left(s,Y_s,Z_s\right)+A_sY_s+ \langle B_s,Z_s\rangle\Big)\,ds
    +\int_{t}^u\Gamma_sY_s \langle B_s,dW_s\rangle+\int_{t}^u\Gamma_s \langle Z_s,dW_s\rangle
  \end{split}     \end{equation}
  and that
  \begin{equation}  \begin{split}\label{eq:Ito2}
    &\Gamma_u \tilde{Y}_u-\tilde{Y}_t
    =\int_{t}^u \Gamma_s
    \Big(-\tilde f(s,\tilde{Y}_s,\tilde{Z}_s)+A_s\tilde{Y}_s+ \langle B_s,\tilde{Z}_s\rangle\Big)\,ds
    +\int_{t}^u\Gamma_s\tilde{Y}_s \langle B_s,dW_s\rangle+\int_{t}^u\Gamma_s \langle \tilde{Z}_s,dW_s\rangle.
  \end{split}     \end{equation}
  Next~\eqref{eq:A}, \eqref{eq:B}, and a telescoping sum imply
  for all $s\in[t,T]$ that
  \begin{equation}  \begin{split}\label{eq:next.telescope}
    &A_s(Y_s-\tilde{Y}_s)+ \langle B_s,Z_s-\tilde{Z}_s\rangle
    =
    A_s(Y_s-\tilde{Y}_s)+\sum_{j=1}^d B_s(j)\big(Z_s(j)-\tilde{Z}_s(j)\big)
    \\&
    =f(s,Y_s,\tilde{Z}_s)-f(s,\tilde{Y}_s,\tilde{Z}_s)
    +\sum_{j=1}^d\Big(f(s,Y_s,Z_s(1),\ldots,Z_s(j),\tilde{Z}_s(j+1),\ldots,\tilde{Z}_s(d))
    \\&\qquad\qquad-f(s,Y_s,Z_s(1),\ldots,Z_s(j-1),\tilde{Z}_s(j),\ldots,\tilde{Z}_s(d))\Big)
    \\&=f(s,Y_s,\tilde{Z}_s)-f(s,\tilde{Y}_s,\tilde{Z}_s)
    +f(s,Y_s,Z_s)-f(s,Y_s,\tilde{Z}_s)
    \\&
    =f(s,Y_s,Z_s)-f(s,\tilde{Y}_s,\tilde{Z}_s).
  \end{split}     \end{equation}
  This, ~\eqref{eq:Ito1}, and~\eqref{eq:Ito2}
  imply
  that for all $u\in [t,T]$ it holds $\P$-a.s.\ that
  \begin{equation}  \begin{split}\label{eq:Ito3}
    &Y_t-\tilde{Y}_t
    -\Gamma_u (Y_u-\tilde{Y}_u)
    +\int_t^u\Gamma_s(Y_s-\tilde{Y}_s)\langle B_s,dW_s\rangle
    +\int_t^u\Gamma_s\langle Z_s-\tilde{Z}_s,dW_s\rangle
    \\&=\int_t^u \Gamma_s
    \Big(f(s,Y_s,Z_s)-\tilde f(s,\tilde{Y}_s,\tilde{Z}_s)
    -\left(A_s(Y_s-\tilde{Y}_s)+\langle B_s,Z_s-\tilde{Z}_s\rangle\Big)\right)\,ds\\
    &=
    \int_t^u \Gamma_s
    \Big(f(s,\tilde Y_s,\tilde Z_s)-\tilde f(s,\tilde{Y}_s,\tilde{Z}_s)\Big)\,ds.
  \end{split}     \end{equation}
  For every $n\in \N$ let $\tau_n\colon \Omega \to [t,T)$ be the stopping time that satisfies
  \begin{equation}
  \tau_n=\inf\left(\left\{r\in [t,T]: \left|\int_t^r\Gamma_s(Y_s-\tilde{Y}_s)\langle B_s,dW_s\rangle\right|+\left|\int_t^r\Gamma_s\langle Z_s-\tilde{Z}_s,dW_s\rangle\right|\ge n \right\}\cup\{T-\tfrac{T}{n}\}\right).
  \end{equation}
 Taking conditional expectations in \eqref{eq:Ito3} 
implies for all $n\in \N$ that $\P$-a.s.\ it holds that
\begin{equation}  \label{eq:deltaIto3}
    Y_t-\tilde{Y}_t=\E\left[\Gamma_{\tau_n} (Y_{\tau_n}-\tilde{Y}_{\tau_n})
    + \int_t^{\tau_n}\Gamma_s
    \Big(f(s,\tilde Y_s,\tilde Z_s)-\tilde f(s,\tilde{Y}_s,\tilde{Z}_s)\Big)\,ds\big | \mathbb F_t\right].
    \end{equation}
    Next note that assumption~\eqref{eq:fLipBSDE} ensures that
   $\sup_{s\in[t,T]}\|B_s\|^2\leq \sum_{j=2}^{d+1}(L_j)^2$.    
This and an exponential martingale argument imply that
\begin{equation}\label{eq:2moment_exp_mart}
    \begin{split}
    \E\left[\left(e^{-\int_t^T\frac{\|B_s\|^2}{2} \,ds+\int_t^T\langle B_s,dW_s\rangle}\right)^2\right]
    &=
    \E\left[e^{\int_t^T \|B_s\|^2 \,ds}
    e^{-\int_t^T \frac{\|2B_s\|^2}{2} \,ds+\int_t^T\langle 2B_s,dW_s\rangle}\right]\\
&    \le 
    e^{(T-t)\sum_{j=2}^{d+1}(L_j)^2}
    \E\left[
    e^{-\int_t^T \frac{\|2B_s\|^2}{2} \,ds+\int_t^T\langle 2B_s,dW_s\rangle}\right]\\
    &
    = e^{(T-t)\sum_{j=2}^{d+1}(L_j)^2}.
    \end{split}
    \end{equation}
Note that by assumption~\eqref{eq:fLipBSDE} it holds $\P$-a.s.\ that 
$\sup_{s\in [t,T]}|A_s|\leq L_1$.
This, Doob's martingale inequality, and \eqref{eq:2moment_exp_mart}, 
show that
\begin{equation}
\begin{split}
\E\left[\sup_{s\in [t,T]}\Gamma_s^2\right]
&=
\E\left[\sup_{s\in [t,T]}e^{\int_t^s 2A_rdr}\left(e^{-\int_t^s\frac{\|B_r\|^2}{2} dr+\int_t^s\langle B_r,dW_r\rangle}\right)^2\right]\\
&
\le 
e^{2L_1(T-t)}\E\left[\sup_{s\in [t,T]}\left(e^{-\int_t^s\frac{\|B_r\|^2}{2} dr+\int_t^s\langle B_r,dW_r\rangle}\right)^2\right]\\
&\le 4 e^{2L_1(T-t)}\E\left[\left(e^{-\int_t^T\frac{\|B_r\|^2}{2} dr+\int_t^T\langle B_r,dW_r\rangle}\right)^2\right]
\le 4 e^{(T-t)\left(2L_1+\sum_{j=2}^{d+1}(L_j)^2\right)}.
\end{split}
\end{equation}  
    This together with the Cauchy-Schwarz inequality, and the assumptions that 
$\E\!\left[\sup_{s\in [t,T]}Y^2_{s}\right]<\infty$ and $\E\!\left[\sup_{s\in [t,T]}\tilde{Y}_{s}^2\right]<\infty$
    prove that
    \begin{equation}\label{eq:cs_aux}
 \E\!\left[\sup_{s\in [t,T]}|\Gamma_{s} (Y_{s}-\tilde{Y}_{s})|\right]
 \le \sqrt{\E\!\left[\sup_{s\in [t,T]}\Gamma_s^2\right]\E\!\left[\sup_{s\in [t,T]}(Y_{s}-\tilde{Y}_{s})^2\right]}<\infty.
\end{equation}
Moreover, it holds that
\begin{equation}
\E\left[
    \int_t^{T}\Gamma_s
    \Big|f(s,\tilde Y_s,\tilde Z_s)-\tilde f(s,\tilde{Y}_s,\tilde{Z}_s)\Big|\,ds \right]
    \le \E\left[
    \int_t^{T}\Gamma_s
    \,ds \right]\left[
    \sup_{s\in [t,T], y\in \R, z\in \R^d}\|f(s,y,z)-\tilde f(s,y,z)\|_{L^\infty(\P;\R)}
    \right]<\infty.
\end{equation}
This, \eqref{eq:cs_aux}, Lebesgue's dominated convergence theorem, \eqref{eq:deltaIto3}, continuity of $Y$, $\tilde{Y}$, $\Gamma$, and the 
fact that it holds $\P$-a.s.\ that $\lim_{n\to \infty}\tau_n=T$ ensure that $\P$-a.s.\ it holds that
\begin{equation}  \begin{split}
    Y_t-\tilde{Y}_t&=\lim_{n\to\infty}\E\left[\Gamma_{\tau_n} (Y_{\tau_n}-\tilde{Y}_{\tau_n})
    + \int_t^{\tau_n}\Gamma_s
    \Big(f(s,\tilde Y_s,\tilde Z_s)-\tilde f(s,\tilde{Y}_s,\tilde{Z}_s)\Big)\,ds\big | \mathbb F_t\right]\\&=
\E\left[\Gamma_{T} (Y_{T}-\tilde{Y}_{T})
    + \int_t^{T}\Gamma_s
    \Big(f(s,\tilde Y_s,\tilde Z_s)-\tilde f(s,\tilde{Y}_s,\tilde{Z}_s)\Big)\,ds\big | \mathbb F_t\right]
\end{split}     \end{equation}
  This, the triangle inequality, the fact that the process $[t,T]\times \Omega \ni (s,\omega) \mapsto e^{-\int_t^s\frac{\|B_r\|^2}{2} dr +\int_t^s\langle B_r,dW_r\rangle} \in (0,\infty)$
   is a martingale, and the fact
  $\sup_{s\in [t,T]}|A_s|\leq L_1$ yield 
  that $\P$-a.s.\ it holds that
  \begin{equation}
  \begin{split}
    \big|Y_t-\tilde{Y}_t\big|
    &\le\E\!\left[\Gamma_{T} |Y_T-\tilde{Y}_T|
    + \int_t^{T}\Gamma_s
    \Big|f(s,\tilde Y_s,\tilde Z_s)-\tilde f(s,\tilde{Y}_s,\tilde{Z}_s)\Big|\,ds
    \Big | \mathbb F_t\right]\\
    &
\le \E[\Gamma_T\big | \mathbb F_t]\,\|Y_T-\tilde{Y}_T\|_{L^{\infty}(\P;\R)}
+\E\!\left[
     \int_t^{T}\Gamma_s
   \, ds
    \Big | \mathbb F_t\right] \left[ 
    \sup_{s\in [t,T], y\in \R, z\in \R^d}\|f(s,y,z)-\tilde f(s,y,z)\|_{L^\infty(\P;\R)}
    \right]\\&
\le    
    e^{L_1(T-t)}\left(\|Y_T-\tilde{Y}_T\|_{L^{\infty}(\P;\R)}+(T-t)\sup_{s\in [t,T], y\in \R, z\in \R^d}\|f(s,y,z)-\tilde f(s,y,z)\|_{L^\infty(\P;\R)}\right).
    \end{split}
    \end{equation}
    This proves \eqref{eq:grad_bound}.
     The proof of Lemma~\ref{lem:grad_bound} is thus completed.  
\end{proof}

\subsection{Regularity analysis for solutions of partial differential equations (PDEs)}

\begin{lemma}[Upper bound for exact solution]
\label{upper_exact}
Let $T\in(0,\infty)$, $d\in\N$,
$\eta, L_0,L_1,\ldots,L_{d}$, $\mathfrak{L}_1, \mathfrak{L}_2, \ldots, \mathfrak{L}_d$, $K_1,K_2,\ldots,K_d\in \R$,
$f\in C( [0,T]\times\R^d\times \R \times \R^d, \R)$,
$g\in C(\R^d, \R)$,
$u = ( u(t,x) )_{ (t,x) \in [0,T] \times \R^d }\in C^{1,2}([0,T]\times\R^d,\R)$ ,
let $\left\|\cdot \right\|\colon \R^{d+1}\to [0,\infty)$ be a norm,
assume
for all
$t\in[0,T]$, $x=(x_1,x_2,\ldots, x_d)$, $\mathfrak{x}=(\mathfrak{x}_1,\mathfrak{x}_2,\ldots,\mathfrak{x}_d)$, $z=(z_1,z_2,\ldots, z_d)$, $\mathfrak{z}=(\mathfrak{z}_1,\mathfrak{z}_2,\ldots,\mathfrak{z}_d) \in\R^d$, $y,\mathfrak{y} \in \R$ that
\begin{equation}\label{eq:fLipschitz2}
|f(t,x,y,z)-f(t,\mathfrak{x},\mathfrak y,\mathfrak{z})|\le L_0|y-\mathfrak y|+\textstyle \sum_{j=1}^d \displaystyle (L_{j}|z_j-\mathfrak{z}_j|
+ \mathfrak L_{j}|x_j-\mathfrak{x}_j|),
\end{equation}
\begin{equation}\label{eq:gLipschitz2}
|g(x)-g(\mathfrak x)|\le \textstyle\sum_{i=1}^d \displaystyle K_i |x_i-\mathfrak x_i|, \qquad
| u(t,x) | \leq \eta \big[ 1 + \textstyle\sum_{ i=1 }^d\displaystyle | x_i | \big]^{ \eta },
\qquad 
 u(T,x)=g(x),
\end{equation}
%that $\inf_{\lambda\in(0,\infty)}\sup_{t\in[0,T],x\in\R^d}\tfrac{|u(t,x)|}{1+\|x\|_2^\lambda}<\infty$, 
%and
%for all $t\in (0,T)$, $x\in \R^d$ it holds that and that
\begin{equation}\label{eq:PDE}
\text{and}\qquad\big( \tfrac{ \partial }{ \partial t } u \big)( t, x ) + \tfrac{ 1 }{ 2 } ( \Delta_x u )( t, x ) + f\big( t, x, u(t,x), ( \nabla_x u )(t, x) \big) = 0,
\end{equation}
let
  $(\Omega,\mathcal{F},\P)$ be a probability space, and
  let $W\colon[0,T]\times\Omega\to\R^d$ be a standard Brownian motion.
Then
\begin{enumerate}[(i)]
\item \label{item:upper_exact_1} it holds for all $s\in[0,T)$, $x\in\R^d$ that
\begin{equation}  \begin{split}\label{eq:feynmankacintegrabilityExact}
&\E\!\left[\big\|g(x+W_{T-s})\big(1,\tfrac{W_{T-s}}{T-s} \big)
          \big\|
          \right]
      \\&\quad
      +\E\!\left[
      \int_s^{T}\big\|\big[
      f\big(t,x+W_{t-s},u(t,x+W_{t-s}),(\nabla_xu)(t,x+W_{t-s})\big)\big]
      \big(1,\tfrac{W_{t-s}}{t-s} \big)\big\|
      \,dt
      \right]<\infty,
\end{split}     \end{equation}
\item \label{item:upper_exact_2}it holds for all $s\in[0,T)$, $x\in\R^d$ that
\begin{equation}  \begin{split}\label{eq:feynmankacuinftyExact}
(u(s,x),(\nabla_x u)(s,x))&=\E\!\left[g(x+W_{T-s})\big(1,\tfrac{W_{T-s}}{T-s} \big)\right]
\\&\quad
      +\E\!\left[
      \int_s^{T}\left[
      f\big(t,x+W_{t-s},u(t,x+W_{t-s}),(\nabla_xu)(t,x+W_{t-s})\big)\right]
      \big(1,\tfrac{W_{t-s}}{t-s} \big)
      \,dt\right],
\end{split}     \end{equation}
\item \label{item:upper_exact_3}it holds for all $t\in[0,T]$, $x=(x_1,x_2,\ldots, x_d)$, $\mathfrak x=(\mathfrak{x}_1,\mathfrak{x}_2,\ldots, \mathfrak{x}_d)\in\R^d$ that
\begin{equation}\label{eq:grad_bound2}
\left| u(t,x)-u(t,\mathfrak{x})\right|\le e^{L_0(T-t)}\Big(\textstyle\sum_{j=1}^d\displaystyle(K_j+(T-t)\mathfrak L_{j})|x_j-\mathfrak x_j|\Big),
\end{equation}
\item \label{item:upper_exact_4}it holds for all $t\in(0,T)$, $x\in\R^d$, $i\in\{1,2,\ldots,d\}$ that
\begin{equation}\label{eq:grad_bound3}
\big| (\tfrac{\partial}{\partial x_i}u)(t,x)\big|\le e^{L_0(T-t)}
  (K_i+(T-t)\mathfrak{L}_{i}),
\end{equation}
and
  \item \label{item:upper_exact_5}
it holds for all $x\in\R^d$, $p\in[1,\infty)$ that
\begin{equation}  
\begin{split}\label{upper_exact:concl1}
  &\sup_{s\in[0,T]} \sup_{t\in[s,T]}\left(\E\!\left[\left|u(t,x +W_t-W_s)\right|^p\right]\right)^{\!\nicefrac{1}{p}}
  \\&
  \leq e^{L_0T} 
  \left[
  \sup_{s\in[0,T]}
  (\E[|g(x+W_s)|^p])^{\!\nicefrac{1}{p}}
 +T\sup_{s,t\in[0,T]}
 (\E[|f(t,x+W_s,0,0)|^p])^{\!\nicefrac{1}{p}}
 +Te^{L_0T}\textstyle\sum_{j=1}^d \displaystyle L_{j} (K_j +T\mathfrak{L_j}) \right].
\end{split}     
\end{equation}
\end{enumerate}
\end{lemma}
\begin{proof}[Proof of Lemma~\ref{upper_exact}]
Throughout this proof let $\left\|\cdot\right\|_2\colon \R^d \to [0,\infty)$ be the $d$-dimensional Euclidean norm.
Item \eqref{item:upper_exact_1} and item \eqref{item:upper_exact_2} follow from Lemma 4.2 in~\cite{HutzenthalerKruse2017}.
Next we prove item \eqref{item:upper_exact_3}.
  Throughout the proof of item \eqref{item:upper_exact_3} let $t\in[0,T)$, $x=(x_1,x_2,\ldots, x_d)$, $\mathfrak{x}=(\mathfrak{x}_1,\mathfrak{x}_2,\ldots,\mathfrak{x}_d) \in\R^d$,
  let
  $Y,\mathfrak{Y}\colon[t,T]\times\Omega\to\R$, let
  $Z,\mathfrak{Z}\colon[t,T]\times\Omega\to\R^d$
  be functions which satisfy for all $s\in[t,T]$ that
 \begin{align}
  \label{eq:X}
    Y_s&=u(s,x+W_s-W_t),\\
  \label{eq:Y2}
    \mathfrak{Y}_s&=u(s,\mathfrak x+W_s-W_t),\\
    Z_s&=(\nabla_x u)(s,x+W_s-W_t),\\
    \mathfrak{Z}_s&=(\nabla_x u)(s,\mathfrak x +W_s-W_t).
  \end{align}
Then It\^o's lemma yields that for all $s\in[t,T]$ it holds $\P$-a.s.\ that
\begin{equation}  \begin{split}\label{eq:dyn_Y}
  Y_T=Y_s-\int_s^T f(r,x+W_r-W_t,Y_r,Z_r)\,dr+\int_s^T \big(Z_r\big)^T\,dW_r
\end{split}     \end{equation}
and
\begin{equation}  \begin{split}\label{eq:dyn_tilde_Y}
  \mathfrak{Y}_T=\mathfrak{Y}_s-\int_s^T f(r,\mathfrak x+W_r-W_t,\mathfrak{Y}_r,\mathfrak{Z}_r)\,dr
  +\int_s^T \big(\mathfrak{Z}_r\big)^T\,dW_r.
\end{split}     \end{equation}
Next note that \eqref{eq:gLipschitz2} implies that there exists $\lambda \in (\frac{1}{2},\infty)$ such that
$\sup_{s\in[0,T],\xi\in\R^d}\tfrac{|u(s,\xi)|}{1+\|\xi\|_2^\lambda}<\infty$. For such a  $\lambda \in (\frac{1}{2},\infty)$ Doob's inequality implies that
\begin{equation}\label{eq:supY}
\begin{split}
\left \| \sup_{s\in [t,T]} |Y_s| \right\|_{L^2(\P;\R)} 
&=\left \| \sup_{s\in [t,T]} |u(s,x+W_s-W_t)| \right\|_{L^2(\P;\R)} \\
&\le \left[ \sup_{s\in[0,T],\xi \in\R^d}\tfrac{|u(s,\xi)|}{1+\|\xi\|_2^\lambda}\right] 
\left( 1+ \left \| \sup_{s\in [t,T]} \|x+W_s-W_t\|^\lambda_2 \right\|_{L^2(\P;\R)} \right)\\
&\le \left[ \sup_{s\in[0,T],\xi\in\R^d}\tfrac{|u(s,\xi)|}{1+\|\xi\|_2^\lambda}\right] 
\left( 1+ \left(\frac{2\lambda}{2\lambda-1} \right)^{1/\lambda}\left \| \|x+W_T-W_t\|^\lambda_2 \right\|_{L^2(\P;\R)} \right)<\infty
\end{split}
\end{equation}
and likewise
\begin{equation}\label{eq:suptildeY}
\left \| \sup_{s\in [t,T]} |\mathfrak Y_s| \right\|_{L^2(\P;\R)}<\infty.
\end{equation}
Moreover, \eqref{eq:fLipschitz2} implies that for all $s\in [t,T]$, $y,\mathfrak{y} \in \R$, $z=(z_1,z_2,\ldots, z_d)$, $\mathfrak{z}=(\mathfrak{z}_1,\mathfrak{z}_2,\ldots,\mathfrak{z}_d) \in\R^d$,  it holds $\P$-a.s.\ that
\begin{equation}\label{eq:delta_f}
 |f(s,x+W_t-W_s,y,z)-f(s,x+W_t-W_s,\mathfrak{y},\mathfrak{z})|\leq L_0|y-\mathfrak{y}|+\sum_{j=1}^dL_{j}|z_j-\mathfrak{z}_j|.
\end{equation}
and
\begin{equation}\label{eq:fLip_yz}
|f(s,x+W_s-W_t,y,z)-f(s,\mathfrak x +W_s-W_t,y,z)|\le 
 \sum_{j=1}^d\mathfrak{L}_{j}|x_j-\tilde x_j|.
\end{equation}
This, \eqref{eq:dyn_Y}, \eqref{eq:dyn_tilde_Y}, \eqref{eq:supY}, \eqref{eq:suptildeY},  %\eqref{eq:fLip_yz},
and Lemma \ref{lem:grad_bound}
prove that
\begin{equation}  \begin{split}
  &|u(t,x)-u(t,\mathfrak x)|=
  \big|Y_{t}-\mathfrak{Y}_{t}\big|\\
  &\le
     e^{L_0(T-t)}\Big(\|Y_T-\mathfrak{Y}_T\|_{L^{\infty}(\P;\R)}+(T-t)\sup_{s\in [t,T], y\in \R, z\in \R^d}\|f(s,x+W_s-W_t,y,z)- f(s,\mathfrak x +W_s-W_t,y,z)\|_{L^\infty(\P;\R)}\Big)\\
     &\le
     e^{L_0(T-t)}\Big(\|g(x+W_T-W_t)-g(\mathfrak x+W_T-W_t)\|_{L^{\infty}(\P;\R)}+(T-t) \sum_{j=1}^d\mathfrak L_{j}|x_j-\mathfrak x_j|\Big)
     \\
     &\le
     e^{L_0(T-t)}\Big(\sum_{j=1}^d(K_j+(T-t)\mathfrak{L}_{j})|x_j-\mathfrak x_j|\Big).
   \end{split}     \end{equation}  
This proves item \eqref{item:upper_exact_3}.
From item \eqref{item:upper_exact_3} it then follows for all $t\in[0,T]$, $x\in\R^d$, $i\in\{1,2,\ldots,d\}$
that
\begin{equation}  \begin{split}
  \left| (\tfrac{\partial}{\partial x_i}u)(t,x)\right|
  =
  \left| \lim_{\R\setminus\{0\}\ni h\to 0}\tfrac{u(t,x+h e_i)-u(t,x)}{h}\right|
  \le e^{L_0(T-t)}
  (K_i+(T-t)\mathfrak{L}_{i}).
\end{split}     \end{equation}
and this proves item \eqref{item:upper_exact_4}.
Next we prove item \eqref{item:upper_exact_5}.
For the rest of the proof let
$
  \funcF \colon 
  C([0,T)\times\R^d,\R^{1+d})
  \to
  C([0,T)\times\R^d,\R)
$, ${\bf u}\colon [0,T)\times \R^d \to \R^{d+1}$ be the functions which satisfies for all $t\in[0,T)$, $x\in\R^d$, ${\bf v}\in C([0,T)\times\R^d,\R^{1+d})$ that
$(F({\bf v}))(t,x)=f(t,x,{\bf v}(t,x))$ and ${\bf u}(t,x)=(u(t,x), (\nabla_x u)(t,x))$.
Item \eqref{item:upper_exact_2}, Tonelli's theorem, and the triangle inequality
prove
for all $s\in[0,T]$, $t\in[s,T]$, $x\in\R^d$, $p\in[1,\infty)$ that
\begin{equation}  \begin{split}
  &\|u(t,x+W_t-W_s)\|_{L^p(\P;\R)}
  \\&=
  \Big\| \int g(x+W_t-W_s+y)\P_{W_T-W_s}(dy)+\int_t^T \int (F({\bf u}))(v,x+W_t-W_s+y)\,\P_{W_v-W_t}(dy)\,dv\Big\|_{L^p(\P;\R)}
  \\&\leq
  \Big\| \int g(x+W_t-W_s+y)\P_{W_T-W_s}(dy)\Big\|_{L^p(\P;\R)}+\int_t^T \Big\| \int (F({\bf u}))(v,x+W_t-W_s+y)\,\P_{W_v-W_t}(dy)\Big\|_{L^p(\P;\R)}\,dv
\end{split}     \end{equation}
This,
Jensen's inequality
independence of Brownian increments
yield
for all $s\in[0,T]$, $t\in[s,T]$, $x\in\R^d$, $p\in[1,\infty)$ that
\begin{equation}  \begin{split}
  &\|u(t,x+W_t-W_s)\|_{L^p(\P;\R)}
  \leq
  \left(
   \int \int \Big|g(x+z+y)\Big|^p\P_{W_T-W_s}(dy)\,\P_{W_t-W_s}(dz)\right)^{\frac{1}{p}}
 \\&\qquad
   +\int_t^T \left(\int \int \Big|(F({\bf u}))(v,x+z+y)\Big|^p\,\P_{W_T-W_s}(dy)\,\P_{W_t-W_s}(dz)\right)^{\frac{1}{p}}\,dv.
  \\&=
  \left(
   \int \int \Big|g(x+z+y)\Big|^p\P_{(W_v-W_t,W_t-W_s)}(d(y,z))\right)^{\frac{1}{p}}
 \\&\qquad
   +\int_t^T \left(\int \int \Big|(F({\bf u}))(v,x+z+y)\Big|^p\,\P_{(W_v-W_t,W_t-W_s)}(d(y,z))\right)^{\frac{1}{p}}\,dv
 \\&=\|g(x+W_T-W_s)\|_{L^p(\P;\R)}
 +\int_t^T\Big\| (F({\bf u}))(v,x+W_v-W_s)\Big\|_{L^p(\P;\R)}\,dv
\end{split}     \end{equation}
This, the triangle inequality,
the global Lipschitz assumption~\ref{eq:fLipschitz2} of $f$,
and item \eqref{item:upper_exact_4}
show
for all $s\in[0,T]$, $x\in\R^d$, $p\in[1,\infty)$ that
\begin{equation}  \begin{split}\label{eq:upper.bound.u.Lip}
  &\sup_{t\in[s,T]}\|u(t,x+W_t-W_s)\|_{L^p(\P;\R)}
 \leq\|g(x+W_T-W_s)\|_{L^p(\P;\R)}
 +\int_s^T\Big\| (F(0))(v,x+W_v-W_s)\Big\|_{L^p(\P;\R)}\,dv
 \\&\qquad+L_0\int_s^T\Big\| { u}(v,x+W_v-W_s)\Big\|_{L^p(\P;\R)}\,dv
 +\sum_{j=1}^d L_{j}\int_s^T\Big\| \big(\tfrac{\partial}{\partial x_j}u\big)(v,x+W_v-W_s)\Big\|_{L^p(\P;\R)}\,dv
\\&\leq\|g(x+W_T-W_s)\|_{L^p(\P;\R)}
 +\int_s^T\Big\| (F(0))(v,x+W_v-W_s)\Big\|_{L^p(\P;\R)}\,dv
 \\&\qquad+L_0\int_s^T\sup_{t\in[v,T]}\Big\| { u}(t,x+W_t-W_s)\Big\|_{L^p(\P;\R)}\,dv
 +\sum_{j=1}^d e^{L_0T}T L_{j} (K_j +T\mathfrak{L_j})
\end{split}     \end{equation}
Note that there exists $\lambda \in (0,\infty)$ such that
$\sup_{s\in[0,T],\xi\in\R^d}\tfrac{|u(s,\xi)|}{1+\|\xi\|_2^\lambda}<\infty$. For such a  $\lambda \in (0,\infty)$ and for all $x\in\R^d$, $p\in[1,\infty)$ it holds that
\begin{equation}  \begin{split}
  \sup_{s\in[0,T],t\in[s,T]}\|u(t,x+W_t-W_s)\|_{L^p(\P;\R)}
  \leq \left[ \sup_{s\in[0,T],\xi \in\R^d}\tfrac{|u(s,\xi)|}{1+\|\xi\|_2^\lambda}\right] 
 \left(1+  \sup_{s\in[0,T],t\in[s,T]}\left \| \|x+W_t-W_s\|^\lambda_2 \right\|_{L^p(\P;\R)} \right)
  <\infty.
\end{split}     \end{equation}
This, \eqref{eq:upper.bound.u.Lip}, and Gronwall's inequality finally yield for all $x\in\R^d$, $p\in[1,\infty)$ that
\begin{equation}  \begin{split}
  &\sup_{s\in[0,T]}\sup_{t\in[s,T]}\|u(t,x+W_t-W_s)\|_{L^p(\P;\R)}
\\&
  \leq e^{L_0T}\sup_{s\in[0,T]}
  \|g(x+W_s)\|_{L^p(\P;\R)}
 +e^{L_0T}T\bigg(\sup_{s,t\in[0,T]}\big\| (F(0))(t,x+W_s)\big\|_{L^p(\P;\R)}
 +e^{L_0T}\sum_{j=1}^d L_{j} (K_j +T\mathfrak{L_j}) 
 \bigg).
\end{split}     \end{equation}
The proof of Lemma~\ref{upper_exact} is thus completed.
\end{proof}

\section{Overall complexity analysis for MLP approximation methods}\label{sec:rate}

In this section we combine the findings of Sections~\ref{sec:error_analysis} and~\ref{sec:ex_sol} to establish in Theorem~\ref{cor:comp_and_error2}
 below the main approximation result 
of this article; see also Corollary~\ref{cor:comp_and_error} and Corollary~\ref{cor:comp_and_error3} below. The i.i.d.\ random variables 
$\unif^\theta\colon \Omega \to (0,1)$, $\theta \in \Theta$,
appearing in the MLP approximation methods in
Corollary~\ref{cor:comp_and_error} (see \eqref{eq:def:Ucor} in Corollary~\ref{cor:comp_and_error}), 
Theorem~\ref{cor:comp_and_error2} (see \eqref{eq:def:Ucor2} in Theorem~\ref{cor:comp_and_error2}), 
and 
Corollary~\ref{cor:comp_and_error3} (see  \eqref{eq:def:Ucor3} in Corollary~\ref{cor:comp_and_error3}) are 
employed to approximate the time integrals in the semigroup 
formulations of the PDEs under consideration. 
One of the key ingredients of the MLP approximation methods, which we propose and analyze in this article, is the fact that the 
density of these i.i.d.\ random variables 
$\unif^\theta\colon \Omega \to (0,1)$, $\theta \in \Theta$,
is equal to the function 
$(0,1) \ni s \mapsto \alpha s^{\alpha-1}\in \R $
for some $\alpha \in ( 0, 1 )$, 
or equivalently, that these i.i.d.\ random variables satisfy for all 
$\theta \in \Theta$, $b \in (0,1)$ that $\P( \unif^{ \theta } \leq b ) = b^\alpha$ 
for some $\alpha \in (0,1)$. In particular, in contrast to 
previous MLP approximation methods studied in the scientific literature (see, e.g., \cite{HJKNW2018,hutzenthaler2019overcoming,beck2019overcoming}) it is crucial in this article 
to exclude the case where the random variables 
$\unif^\theta\colon \Omega \to (0,1)$, $\theta \in \Theta$,
are continuous uniformly distributed on $(0,1)$ (corresponding to the case $\alpha = 1$). To make this aspect more clear to the reader, 
we provide in Lemma~\ref{lem:uniform_dist} below an explanation 
why it is essential to exclude the continuous uniform 
distribution case $\alpha = 1$. 
Note that the random variable ${\bf U}\colon \Omega \to \R^{d+1}$ in Lemma~\ref{lem:uniform_dist}
coincides with a special case of the random fields in \eqref{eq:def:Ucor2} in Theorem~\ref{cor:comp_and_error2} (with $g_d(x)=0$, $f_d(s,x,y,z)=1$, $M=1$, $t=0$ for $s\in [0,T)$, $x,z \in \R^d$, $y\in \R$, $d\in \N$ in the notation of Theorem~\ref{cor:comp_and_error2}).

%In this section we combine the findings of Sections~\ref{sec:error_analysis} and~\ref{sec:ex_sol} to etablish in Theorem~\ref{cor:comp_and_error2} below the main approximation result of this article. Moreover, we show in Lemma~\ref{lem:uniform_dist} below that the conclusion of  Theorem~\ref{cor:comp_and_error2} does not hold true, if the random variables $\unif^\theta\colon \Omega \to (0,1)$, $\theta \in \Theta$, in Theorem~\ref{cor:comp_and_error2} are replaced by i.i.d.\ uniformly distributed random variables. Note that the random variable ${\bf U}\colon \Omega \to \R^{d+1}$ in Lemma~\ref{lem:uniform_dist} is a special case of \eqref{eq:def:Ucor2} in Theorem~\ref{cor:comp_and_error2} 

%In this section we combine the results of Section~\ref{sec:error_analysis} and Section~\ref{sec:ex_sol} and analyze the computational effort of the approximation method~\eqref{eq:def:U} to obtain the main result of this article, Corollary~\ref{cor:comp_and_error} below. In the formulation of Corollary~\ref{cor:comp_and_error}
%for every $n\in \N_0$ and every $M\in \N$ the natural number $\RN_{n,M}$ is an upper bound for the number of realizations of scalar random variables which are
% required to compute one realization of $U^0_{n,M}(0,0)$.
% 

\subsection{Quantitative complexity analysis for MLP approximation methods}

\begin{corollary}\label{cor:comp_and_error}
Let $\left\|\cdot\right \|_1\colon (\cup_{n\in \N}\R^n)\to \R$ and
$\left\|\cdot\right \|_\infty\colon (\cup_{n\in \N}\R^n)\to \R$ satisfy for 
all $n\in \N$, $x=(x_1,x_2,\ldots,x_n)\in \R^n$ that
$\|x\|_1=\sum_{i=1}^n |x_i|$ and
$\|x\|_\infty=\max_{i\in \{1,2,\ldots,n\}}|x_i|$,
let $T,\delta\in(0,\infty)$, $\eps \in (0,1]$, $d \in\N$,
$L=(L_0,L_1,\ldots,L_{d}) \in \R^{d+1}$, 
$K=(K_1,K_2,\ldots,K_d)$, $\mathfrak L=(\mathfrak L_1,\mathfrak L_2\ldots, \mathfrak L_d), \xi \in \R^d$,
 $p\in (2,\infty)$, $\alpha\in (\frac{p-2}{2(p-1)},\frac{p}{2(p-1)})$,
 $\beta=\frac{\alpha}{2}-\frac{(1-\alpha)(p-2)}{2p}\in (0,\frac{\alpha}{2})$,
$f\in C( [0,T]\times\R^d\times \R \times \R^d, \R)$,
$g\in C(\R^d, \R)$,
let
$u = ( u(t,x) )_{ (t,x) \in [0,T] \times \R^d }\in C^{1,2}([0,T]\times\R^d,\R)$ be an at most polynomially growing
function,
%$\inf_{\lambda\in(0,\infty)}\sup_{t\in[0,T],x\in\R^d}\frac{|u(t,x)|}{1+\|x\|_2^\lambda}<\infty$,
assume
for all
$t\in(0,T)$, $x=(x_1,x_2, \ldots,x_d)$, $\mathfrak x=(\mathfrak x_1,\mathfrak x_2, \ldots,\mathfrak x_d)$, $z=(z_1,z_2,\ldots,z_d)$, $\mathfrak{z}=(\mathfrak z_1, \mathfrak z_2, \ldots, \mathfrak z_d)\in\R^d$, $y,\mathfrak{y} \in \R$ that
\begin{equation}\label{eq:fLipschitz3}
|f(t,x,y,z)-f(t,\mathfrak x,\mathfrak y,\mathfrak{z})|\le L_0|y-\mathfrak y|+
\textstyle{
\sum_{j=1}^d}
\big( L_{j}|z_j-\mathfrak{z}_j|
+\mathfrak L_{j}|x_j-\mathfrak{x}_j|\big),
\end{equation}
\begin{equation}\label{eq:gLipschitz3}
|g(x)-g(\mathfrak  x)|\le \textstyle{\sum_{i=1}^d} K_i |x_i-\mathfrak x_i|, \qquad u(T,x) = g(x),
\end{equation}
\begin{equation}\label{eq:PDE2}
\text{and}\qquad\big( \tfrac{ \partial }{ \partial t } u \big)( t, x ) + \tfrac{ 1 }{ 2 } ( \Delta_x u )( t, x ) + f\big( t, x, u(t,x), ( \nabla_x u )(t, x) \big) = 0,
\end{equation}
let
$
  \funcF \colon 
  C([0,T)\times\R^d,\R^{1+d})
  \to
  C([0,T)\times\R^d,\R)
$ satisfy for all $t\in[0,T)$, $x\in\R^d$, ${\bf v}\in C([0,T)\times\R^d,\R^{1+d})$ that
$(F({\bf v}))(t,x)=f(t,x,{\bf v}(t,x))$,
let
$
  ( 
    \Omega, \mathcal{F}, \P 
  )
$
be a probability space,
let
$
  \Theta = \cup_{ n \in \N } \Z^n
$,
let
$
  Z^{ \theta } \colon \Omega \to \R^d 
$, 
$ \theta \in \Theta $,
be i.i.d.\ standard normal random variables,
let $\unif^\theta\colon \Omega\to(0,1)$, $\theta\in \Theta$, be 
i.i.d.\ random variables,
%$\mathcal F / \mathcal B((0,1))$-measurable,
%%whose distribution has density $\rho$, i.e.\
assume 
%for all $\theta\in \Theta$, $B\in\mathcal{B}((0,1))$
for all $b\in (0,1)$
%that $\P(\unif^0\le b)=\int_B\frac{1-\alpha}{s^\alpha} ds$, 
that
$\P(\unif^0\le b)=b^{1-\alpha}$,
assume that
$(Z^\theta)_{\theta \in \Theta}$ and
$(\unif^\theta)_{ \theta \in \Theta}$ are independent,
%is a family of independent random variables,
%let $\uniform^\theta\colon [0,T]\times \Omega\to [0,T]$, $\theta \in \Theta$, satisfy for all $t\in [0,T]$, $\theta \in \Theta$ that $\uniform^\theta _t = t+ (T-t)\unif^\theta$, let $\ES\colon \Omega \times \N_0 \times [0,T]\to [0,T]$ satisfy
%for all $t\in [0,T]$, $n\in \N_0$ that $\ES(0,t)=t$ and that $\ES(n+1,t)=\uniform^{(n)}_{\ES(n,t)}$,
let
$ 
  {\bf U}_{ n,M}^{\theta }
=
  (
  {\bf U}_{ n,M}^{\theta, 0},{\bf U}_{ n,M}^{\theta, 1},\ldots,{\bf U}_{ n,M}^{\theta, d}
  )
  \colon[0,T)\times\R^d\times\Omega\to\R^{1+d}
$,
$n,M\in\Z$, $\theta\in\Theta$,
satisfy
for all 
$
  n,M \in \N
$,
$ \theta \in \Theta $,
$ t\in [0,T)$,
$x \in \R^d$
that $
{\bf U}_{-1,M}^{\theta}(t,x)={\bf U}_{0,M}^{\theta}(t,x)=0$ and
\begin{equation}  \begin{split}\label{eq:def:Ucor}
  {\bf U}_{n,M}^{\theta}(t,x)
  &=
  \left(
    g(x)
    , 0
  \right)
  +
  \tfrac{1}{M^n}\textstyle \sum\limits_{i=1}^{M^n} \displaystyle \big(g(x+[T-t]^{1/2}Z^{(\theta,0,-i)})-g(x)\big)
  \big(
  1 ,  [T - t]^{-1/2}
  Z^{(\theta, 0, -i)}
  \big)
  \\
  &\quad +\textstyle\sum\limits_{l=0}^{n-1}\sum\limits_{i=1}^{M^{n-l}}\displaystyle \tfrac{(T-t)(\unif^{(\theta, l,i)})^\alpha}{(1-\alpha)M^{n-l}} \big(
  1 , [(T-t)\unif^{(\theta, l,i)}]^{-1/2}
  Z^{(\theta,l,i)}
  \big)\\
 & \quad 
\cdot \big[ 
 \big(\funcF({\bf U}_{l,M}^{(\theta,l,i)})-\1_{\N}(l)\funcF( {\bf U}_{l-1,M}^{(\theta,-l,i)})\big)
  (t+(T-t)\unif^{(\theta, l,i)},x+[(T-t)\unif^{(\theta, l,i)}]^{1/2}Z^{(\theta,l,i)})
  \big],
\end{split}     \end{equation}
let $(\RN_{n,M})_{(n,M)\in \Z^2}\subseteq\Z$ satisfy
for all $n,M \in \N$ that 
$\RN_{0,M}=0$
and 
\begin{align}
\label{c16}
  \RN_{ n,M}
  &\leq d M^n+\textstyle\sum\limits_{l=0}^{n-1}\displaystyle\left[M^{(n-l)}( d+1 + \RN_{ l, M}+ \1_{ \N }( l )  \RN_{ l-1, M })\right],
\end{align}
and let $C\in (0,\infty)$ satisfy that
%$C_1,C_2,C_3\in [0,\infty]$ be the extended real numbers given by
 \begin{equation}
  \begin{split}
 C&=\max\left\{\tfrac{1}{2},|\Gamma(\tfrac{p}{2})|^{\frac{1}{p}}(1-\alpha)^{\frac{1}{p}-1
 }\max\left\{
      T, \Gamma(\tfrac{p+1}{2})^{\frac{1}{p}}\pi^{-\frac{1}{2p}}\sqrt{2T}
      \right\}
      \max\{ 1, \| L \|_1 \}
  \right\}.
\end{split}     \end{equation}
Then there exists $N\in \N\cap [2,\infty)$ such that
 \begin{equation}
\sup_{ n \in \N \cap [N,\infty) } \Big[
 \big\| {\bf U}_{{n},\lfloor n^{2\beta} \rfloor}^{0,0}(0,\xi)-u(0,\xi)\big\|_{L^2(\P;\R)}+
\max_{i\in\{1,2,\ldots,d\}}
    \big\| {\bf U}_{{n},\lfloor n^{2\beta} \rfloor}^{0,i}(0,\xi)-( \tfrac{ \partial }{ \partial x_i } u )(0,\xi)\big\|_{L^2(\P;\R)}
 \Big]
 \le \eps
 \end{equation}
 and
 \begin{equation}  \begin{split}\label{eq:fin_cor}
 &\sum_{n=1}^{N}\RN_{n,\lfloor n^{2\beta} \rfloor}
\le d \varepsilon^{-(2+\delta)} 2^{3+\delta }
\bigg[1+
  \tfrac{\sqrt{ \max\{T,3\}}\|K\|_1}{\sqrt{\lfloor (N-1)^{2\beta} \rfloor}}
 +
 Ce^{L_0T}
 \sup_{s\in[0,T]}
  \|g(\xi+\sqrt{s}Z^{(0)})\|_{L^{\frac{2p}{p-2}}(\P;\R)} 
\\
   &
   \qquad
   +C
 e^{L_0T}(\|K\|_{\infty}+T\|\mathfrak L\|_\infty)
  +
  C
   \sup_{s,t\in [0,T)}
 \left\|
   (\funcF (0))(t,\xi+\sqrt{s}Z^{(0)})
   \right\|_{L^{\frac{2p}{p-2}}(\P;\R)} 
   \Big(\tfrac{1}{\sqrt{\lfloor (N-1)^{2\beta} \rfloor}}+Te^{L_0T}\Big)
 \\
 &\qquad +TCe^{2L_0T}\sum_{j=1}^d L_{j} (K_j +T\mathfrak{L_j})\bigg]^{2+\delta}  \cdot 
\left[
\sup_{n\in \N\cap [2,\infty)}
\left(
\tfrac{5^n \left[(n-1)^{2\beta} \left[e\left(\tfrac{p(n-1)}{2}+1\right) \right]^{\frac{1}{8}}
(4Ce^\beta)^{n-1}
\right]^{2+\delta}
}{(\lfloor (n-1)^{2\beta} \rfloor)^{\frac{\delta n}{2}}}
\right)
\right]<\infty.
 \end{split}     \end{equation}
\end{corollary}

\begin{proof}[Proof of Corollary~\ref{cor:comp_and_error}]
Throughout the proof let $\eps\in (0,1]$ and 
let $(\eta_{n,M})_{(n,M)\in \N^2}\subseteq \R$ satisfy for all $n,M\in \N$ that
\begin{equation}
\eta_{n,M}=
 \left\| {\bf U}_{{n},M}^{0,0}(0,\xi)-u(0,\xi)\right\|_{L^2(\P;\R)}+
\max_{i\in\{1,2,\ldots,d\}}
    \left\| {\bf U}_{{n},M}^{0,i}(0,\xi)-( \tfrac{ \partial }{ \partial x_i } u )(0,\xi)\right\|_{L^2(\P;\R)}.
\end{equation} 
First note that it follows from Proposition~\ref{thm:rate}, item~\eqref{item:upper_exact_1} of Lemma~\ref{upper_exact}, and the fact that the increments of a Brownian motion are normally distributed
that for all $M,n \in \N$ it holds that
\begin{equation}  \begin{split}  \label{eq:error.estimate}
\eta_{n,M}    &\leq 2\max\left\{ 
 \left\| {\bf U}_{{n},M}^{0,0}(0,\xi)-u(0,\xi)\right\|_{L^2(\P;\R)},
\max_{i\in\{1,2,\ldots,d\}}
    \left\| {\bf U}_{{n},M}^{0,i}(0,\xi)-( \tfrac{ \partial }{ \partial x_i } u )(0,\xi)\right\|_{L^2(\P;\R)}
 \right\}\\
 & \le 
    \tfrac{2\left[e\left(\tfrac{pn}{2}+1\right) \right]^{\frac{1}{8}}
(4C)^{n-1}
\exp\left(\beta M^{\frac{1}{2\beta}}\right)
}{\sqrt{M^{n-1}}}
\bigg[
\tfrac{\sqrt{ \max\{T,3\}}\|K\|_1}{\sqrt{M}}
+
 \tfrac{C
   \sup_{s,t\in [0,T)}
 \left\|
   (\funcF (0))(t,\xi+\sqrt{s}Z^{(0)})
   \right\|_{L^{\frac{2p}{p-2}}(\P;\R)}}{\sqrt{M}} 
   \\&\qquad 
   +
   C
\sup_{s,t\in [0,T)}\bigg[
\max\left\{ \left\|u(t,\xi+\sqrt{s}Z^{(0)})\right\|_{L^{\frac{2p}{p-2}}(\P;\R)} ,
\max_{i\in\{1,2,\ldots,d\}}
   \left\| (\tfrac{\partial}{\partial x_i} u)(t,\xi+\sqrt{s}Z^{(0)})
\right\|_{L^{\frac{2p}{p-2}}(\P;\R)}\right\}      
\bigg].
\end{split}     \end{equation}
This together with 
Lemma~\ref{upper_exact}
implies for all $M,n \in \N$ that
\begin{equation}  \begin{split}  \label{eq:error.estimate2}
   \eta_{n,M} 
   &\leq
    \tfrac{2\left[e\left(\tfrac{pn}{2}+1\right) \right]^{\frac{1}{8}}
(4C)^{n}
\exp\left(\beta M^{\frac{1}{2\beta}}\right)
}{\sqrt{M^{n-1}}}
\bigg[
\tfrac{\sqrt{ \max\{T,3\}}\|K\|_1}{\sqrt{M}}
+
C\Big(\tfrac{1}{\sqrt{M}}+Te^{L_0T}\Big)
   \sup_{s,t\in [0,T)}
 \left\|
   (\funcF (0))(t,\xi+\sqrt{s}Z^{(0)})
   \right\|_{L^{\frac{2p}{p-2}}(\P;\R)} 
   \\
   &\qquad
  +
    C
 e^{L_0T}(\|K\|_{\infty}+T\|\mathfrak L\|_\infty)+
Ce^{L_0T}
 \sup_{s\in[0,T]}
  \|g(\xi+\sqrt{s}Z^{(0)})\|_{L^{\frac{2p}{p-2}}(\P;\R)}
 +TCe^{2L_0T}\sum_{j=1}^d L_{j}  (K_j +T\mathfrak{L_j})
\bigg]
  .
\end{split}     \end{equation}
It follows from \eqref{eq:fLipschitz3} and \eqref{eq:gLipschitz3}
that 
 $\sup_{s,t\in [0,T)}
 \left\|
   (\funcF (0))(t,\xi+\sqrt{s}Z^{(0)})
   \right\|_{L^{\frac{2p}{p-2}}(\P;\R)}<\infty$ and
   $ \sup_{s\in[0,T]}
  \|g(\xi+\sqrt{s}Z^{(0)})\|_{L^{\frac{2p}{p-2}}(\P;\R)}<\infty$.
 This together with \eqref{eq:error.estimate2} proves that
%\begin{equation}
$
\limsup_{n\to \infty}\eta_{n, \lfloor n^{2\beta} \rfloor} =0
$.
%\end{equation} 
 Let $N\in \N$ be the natural number given by 
  \begin{equation}\label{eq:def_N}
  N=\min\left\{n\in \N \cap [2,\infty) \colon \sup_{m\in \N\cap [n,\infty)}
  \eta_{m, \lfloor m^{2\beta} \rfloor}
  \le \eps\right\}
  \end{equation}
  and let
  $\mathfrak{C}\in [0,\infty)$ be the real number given by
  \begin{equation}\label{eq:def_tilde_C}
  \begin{split}
  \mathfrak C&=2\bigg[1+
  \tfrac{\sqrt{ \max\{T,3\}}\|K\|_1}{\sqrt{\lfloor (N-1)^{2\beta} \rfloor}}
+
C
   \sup_{s,t\in [0,T)}
 \left\|
   (\funcF (0))(t,\xi+\sqrt{s}Z^{(0)})
   \right\|_{L^{\frac{2p}{p-2}}(\P;\R)} 
   \Big(\tfrac{1}{\sqrt{\lfloor (N-1)^{2\beta} \rfloor}}+Te^{L_0T}\Big)\\
   &\qquad
  +
    C
 e^{L_0T}(\|K\|_{\infty}+T\|\mathfrak L\|_\infty)+
Ce^{L_0T}
 \sup_{s\in[0,T]}
  \|g(\xi+\sqrt{s}Z^{(0)})\|_{L^{\frac{2p}{p-2}}(\P;\R)}
 +TCe^{2L_0T}\sum_{j=1}^d L_{j} (K_j +T\mathfrak{L_j}) \bigg].
  \end{split}
  \end{equation}
  If $N=2$ then it holds that
  \begin{equation}\label{eq:up_bd_eps_N2}
  \eps\le 1\le 4C\mathfrak C\left[e\left(\tfrac{p}{2}+1\right) \right]^{\frac{1}{8}}
e^\beta
=
\tfrac{\mathfrak C\left[e\left(\tfrac{p(N-1)}{2}+1\right) \right]^{\frac{1}{8}}
(4C)^{N-1}
\exp\left(\beta (\lfloor (N-1)^{2\beta} \rfloor)^{\frac{1}{2\beta}}\right)
}{\sqrt{(\lfloor (N-1)^{2\beta} \rfloor)^{N-2}}}.
  \end{equation}
  If $N>2$ 
 it follows from \eqref{eq:def_N}, \eqref{eq:error.estimate2} and \eqref{eq:def_tilde_C}  that
  \begin{equation}\label{eq:up_bd_eps}
  \eps < \eta_{N-1, \lfloor (N-1)^{2\beta} \rfloor}
  \le \tfrac{\mathfrak C\left[e\left(\tfrac{p(N-1)}{2}+1\right) \right]^{\frac{1}{8}}
(4C)^{N-1}
\exp\left(\beta (\lfloor (N-1)^{2\beta} \rfloor)^{\frac{1}{2\beta}}\right)
}{\sqrt{(\lfloor (N-1)^{2\beta} \rfloor)^{N-2}}}.
  \end{equation}
  Moreover,
\cite[Lemma 3.6]{HJKNW2018} implies that for all $n\in \N$ it holds that $\RN_{n,\lfloor n^{2\beta} \rfloor}\leq d(5\lfloor n^{2\beta} \rfloor)^n$. This implies that
\begin{equation}
\sum_{n=1}^{N}\RN_{n,\lfloor n^{2\beta} \rfloor}
\le d
\sum_{n=1}^{N}(5\lfloor n^{2\beta} \rfloor)^n 
\le d
\sum_{n=1}^{N}(5\lfloor N^{2\beta} \rfloor)^n 
=\frac{d (5\lfloor N^{2\beta})\rfloor)( (5\lfloor N^{2\beta} \rfloor)^{N} -1) }{5\lfloor N^{2\beta}\rfloor-1}
\le 2d (5\lfloor N^{2\beta} \rfloor)^{N}.
\end{equation}
Combining this with~\eqref{eq:up_bd_eps_N2} and~\eqref{eq:up_bd_eps} proves that
   \begin{equation}
    \begin{split}
\sum_{n=1}^{N}\RN_{n,\lfloor n^{2\beta} \rfloor}&\le 2d(5\lfloor N^{2\beta} \rfloor)^N=d(5\lfloor N^{2\beta} \rfloor)^N \varepsilon^{2+\delta}\varepsilon^{-(2+\delta)}
\\& 
\le 2 d \varepsilon^{-(2+\delta)} (5\lfloor N^{2\beta} \rfloor)^N 
\left(  \tfrac{\mathfrak C\left[e\left(\tfrac{p(N-1)}{2}+1\right) \right]^{\frac{1}{8}}
(4C)^{N-1}
\exp\left(\beta (\lfloor (N-1)^{2\beta} \rfloor)^{\frac{1}{2\beta}}\right)
}{\sqrt{(\lfloor (N-1)^{2\beta} \rfloor)^{N-2}}} \right)^{2+\delta}
\\ &
= 2d \varepsilon^{-(2+\delta)} \mathfrak{C}^{2+\delta}
\tfrac{5^N(\lfloor N^{2\beta} \rfloor)^N \left[\left[e\left(\tfrac{p(N-1)}{2}+1\right) \right]^{\frac{1}{8}}
(4C)^{N-1}
\exp\left(\beta (\lfloor (N-1)^{2\beta} \rfloor)^{\frac{1}{2\beta}}\right)\right]^{2+\delta}
}{(\lfloor (N-1)^{2\beta} \rfloor)^{\frac{(N-2)(2+\delta)}{2}}}
\\ &
\le 2d \varepsilon^{-(2+\delta)} \mathfrak{C}^{2+\delta}
\tfrac{5^N(\lfloor (N-1)^{2\beta} \rfloor)^N \left[\left[e\left(\tfrac{p(N-1)}{2}+1\right) \right]^{\frac{1}{8}}
(4C)^{N-1}
\exp\left(\beta (\lfloor (N-1)^{2\beta} \rfloor)^{\frac{1}{2\beta}}\right)\right]^{2+\delta}
}{(\lfloor (N-1)^{2\beta} \rfloor)^{\frac{(N-2)(2+\delta)}{2}}}
\\ &
= 2d \varepsilon^{-(2+\delta)} \mathfrak{C}^{2+\delta}
\tfrac{5^N \left[\lfloor (N-1)^{2\beta} \rfloor\left[e\left(\tfrac{p(N-1)}{2}+1\right) \right]^{\frac{1}{8}}
(4C)^{N-1}
\exp\left(\beta (\lfloor (N-1)^{2\beta} \rfloor)^{\frac{1}{2\beta}}\right)\right]^{2+\delta}
}{(\lfloor (N-1)^{2\beta} \rfloor)^{\frac{\delta N}{2}}}
\\ &
\le  2d \varepsilon^{-(2+\delta)} \mathfrak{C}^{2+\delta}
\tfrac{5^N \left[(N-1)^{2\beta} \left[e\left(\tfrac{p(N-1)}{2}+1\right) \right]^{\frac{1}{8}}
(4Ce^\beta)^{N-1}
\right]^{2+\delta}
}{(\lfloor (N-1)^{2\beta} \rfloor)^{\frac{\delta N}{2}}}
\\ &
\le  2d \varepsilon^{-(2+\delta)} \mathfrak{C}^{2+\delta}
\left[
\sup_{n\in \N\cap [2,\infty)}
\left(
\tfrac{5^n \left[(n-1)^{2\beta} \left[e\left(\tfrac{p(n-1)}{2}+1\right) \right]^{\frac{1}{8}}
(4Ce^\beta)^{n-1}
\right]^{2+\delta}
}{(\lfloor (n-1)^{2\beta} \rfloor)^{\frac{\delta n}{2}}}
\right)
\right]\\
&=
 2d \varepsilon^{-(2+\delta)} \mathfrak{C}^{2+\delta}
\left[
\sup_{n\in \N\cap [2,\infty)}
\left(
\tfrac{5 \left[(n-1)^{\frac{2\beta}{n}} \left[e\left(\tfrac{p(n-1)}{2}+1\right) \right]^{\frac{1}{8n}}
(4Ce^\beta)^{\frac{n-1}{n}}
\right]^{2+\delta}
}{(\lfloor (n-1)^{2\beta} \rfloor)^{\frac{\delta}{2}}}
\right)^n
\right]
<\infty
.
    \end{split}
\end{equation}  
This establishes \eqref{eq:fin_cor}. The proof of Corollary~\ref{cor:comp_and_error} is thus completed.   
\end{proof}

\subsection{Qualitative complexity analysis for MLP approximation methods}
\sloppy
\begin{theorem}\label{cor:comp_and_error2}
Let $T,\delta,\lambda \in (0,\infty)$, 
%$p\in (2,\infty)$, 
$\alpha \in (0,1)$,
%$\alpha\in (\frac{p-2}{2(p-1)},\frac{p}{2(p-1)})$,
$\beta \in (\max\{\frac{1-2\alpha}{1-\alpha},0\},1-\alpha)$,
 %$\beta=\frac{\alpha}{2}-\frac{(1-\alpha)(p-2)}{2p}\in (0,\frac{\alpha}{2})$, 
 let $f_d \in C( [0,T]\times\R^d\times \R \times \R^d,\R)$, $d\in \N$,
 let $g_d \in C( \R^d,\R)$, $d\in \N$,
 let $\xi_d=(\xi_{d,1}, \xi_{d,2}, \ldots, \xi_{d,d}) \in \R^d$, $d\in \N$,
 let $L_{d,i}\in \R$, $d,i \in \N$,
let
$u_d = ( u_d(t,x) )_{ (t,x) \in [0,T] \times \R^d }\in C^{1,2}([0,T]\times\R^d,\R)$, $d\in \N$,
be at most polynomially growing functions,
let
$
  \funcF_d \colon 
  C([0,T)\times\R^d,\R^{1+d})
  \to
  C([0,T)\times\R^d,\R)
$, $d\in \N$, be functions,
assume for all $d\in \N$,
$t\in[0,T)$, $x=(x_1,x_2, \ldots,x_d)$, $\mathfrak x=(\mathfrak x_1,\mathfrak x_2, \ldots,\mathfrak x_d)$, $z=(z_1,z_2,\ldots,z_d)$, $\mathfrak{z}=(\mathfrak z_1, \mathfrak z_2, \ldots, \mathfrak z_d)\in\R^d$, $y,\mathfrak{y} \in \R$, ${\bf v}\in C([0,T)\times \R^d,\R^{1+d})$ that
\begin{equation}\label{eq:fLipschitz4}
\max\{|f_d(t,x,y,z)-f_d(t,\mathfrak x,\mathfrak y,\mathfrak{z})|,|g_d(x)-g_d(\mathfrak  x)|\}\le
\textstyle{
\sum_{j=1}^d}L_{d,j}
\big(d^\lambda |x_j-\mathfrak{x}_j|+ |y-\mathfrak y|+|z_j-\mathfrak{z}_j|
\big),
\end{equation}
\begin{equation}\label{eq:PDE3}
\big( \tfrac{ \partial }{ \partial t } u_d \big)( t, x ) + \tfrac{ 1 }{ 2 } ( \Delta_x u_d )( t, x ) + f_d\big( t, x, u(t,x), ( \nabla_x u_d )(t, x) \big) = 0, \qquad u_d(T,x) = g_d(x),
\end{equation}
\begin{equation}
d^{-\lambda}(|g_d(0)|+|f_d(t,0,0,0)|+\max_{i\in \{1,2,\ldots,d\}}
  |\xi_{d,i}|)+\textstyle\sum_{i=1}^d \displaystyle L_{d,i}\le \lambda, \quad \text{and} \quad (F_d({\bf v}))(t,x)=f_d(t,x,{\bf v}(t,x)),
\end{equation}
let
$
  ( 
    \Omega, \mathcal{F}, \P 
  )
$
be a probability space,
let
$
  \Theta = \cup_{ n \in \N } \Z^n
$,
let
$
  Z^{d, \theta } \colon \Omega \to \R^d 
$, $d\in \N$,
$ \theta \in \Theta $,
be i.i.d.\ standard normal random variables,
let $\unif^\theta\colon \Omega\to(0,1)$, $\theta\in \Theta$, be 
i.i.d.\ random variables,
assume 
for all $b\in (0,1)$
that
$\P(\unif^0\le b)=b^{\alpha}$,
assume that
$(Z^{d,\theta})_{(d, \theta) \in \N \times \Theta}$ and
$(\unif^\theta)_{ \theta \in \Theta}$ are independent,
let
$ 
  {\bf U}_{ n,M}^{d,\theta }
=
  (
  {\bf U}_{ n,M}^{d,\theta, 0},{\bf U}_{ n,M}^{d,\theta, 1},\ldots,{\bf U}_{ n,M}^{d,\theta, d}
  )
  \colon[0,T)\times\R^d\times\Omega\to\R^{1+d}
$,
$n,M,d\in\Z$, $\theta\in\Theta$,
satisfy
for all 
$
  n,M,d \in \N
$,
$ \theta \in \Theta $,
$ t\in [0,T)$,
$x \in \R^d$
that $
{\bf U}_{-1,M}^{d,\theta}(t,x)={\bf U}_{0,M}^{d,\theta}(t,x)=0$ and
\begin{equation}  \begin{split}\label{eq:def:Ucor2}
 {\bf U}_{n,M}^{d,\theta}(t,x)
  &=
  \left(
    g_d(x)
    , 0
  \right)
  +
  \textstyle \sum\limits_{i=1}^{M^n} \displaystyle \tfrac{1}{M^n} \big(g_d(x+[T-t]^{1/2}Z^{(\theta,0,-i)})-g_d(x)\big)
  \big(
  1 , [T - t]^{-1/2}
  Z^{d,(\theta, 0, -i)}
  \big)
  \\
  &\quad +\textstyle\sum\limits_{l=0}^{n-1}\sum\limits_{i=1}^{M^{n-l}}\displaystyle \tfrac{(T-t)(\unif^{(\theta, l,i)})^{1-\alpha}}{\alpha M^{n-l}} \big(
  1 ,
   [(T-t)\unif^{(\theta, l,i)}]^{-1/2}
  Z^{d,(\theta,l,i)}
  \big)\\
 & \quad 
\cdot \big[ 
 \big(\funcF_d({\bf U}_{l,M}^{d,(\theta,l,i)})-\1_{\N}(l)\funcF_d( {\bf U}_{l-1,M}^{d,(\theta,-l,i)})\big)
  (t+(T-t)\unif^{(\theta, l,i)},x+[(T-t)\unif^{(\theta, l,i)}]^{1/2}Z^{d,(\theta,l,i)})
  \big],
\end{split}     \end{equation}
and let $\RN_{d,n,M}\in \Z$, $d,n,M\in \Z$, satisfy
for all $d,n,M \in \N$ that 
$\RN_{d,0,M}=0$
and 
\begin{align}
\label{c17}
  \RN_{d,n,M}
  &\leq d M^n+\textstyle\sum\limits_{l=0}^{n-1}\displaystyle\left[M^{(n-l)}( d+1 + \RN_{d, l, M}+ \1_{ \N }( l )  \RN_{d, l-1, M })\right].
\end{align}
Then there exist $c\in \R$ and $N=(N_{d,\eps})_{(d, \eps) \in \N \times (0,1]}\colon \N \times (0,1] \to \N$ such that
for all $d\in \N$, $\eps \in (0,1]$ it holds that 
$ \sum_{n=1}^{N_{d,\eps}}\RN_{d,n,\lfloor n^{\beta} \rfloor} \le c d^c \varepsilon^{-(2+\delta)}$ and
 \begin{equation}\label{eq:cor2_main}
\sup_{ n \in \N \cap [N_{d,\eps},\infty) } \Big[
 \E\big[|{\bf U}_{{n},\lfloor n^{\beta} \rfloor}^{d,0,0}(0,\xi_d)-u_d(0,\xi_d)|^2\big]+
\max_{i\in\{1,2,\ldots,d\}}
    \E\big[ |{\bf U}_{{n},\lfloor n^{\beta} \rfloor}^{d,0,i}(0,\xi_d)-( \tfrac{ \partial }{ \partial x_i } u_d )(0,\xi_d)|^2\big]
 \Big]^{\nicefrac 12}
 \le \eps.
 \end{equation}
\end{theorem}
\fussy  
  
 \begin{proof}[Proof of Theorem~\ref{cor:comp_and_error2}]
 Throughout this proof let $p=\frac{2\alpha}{\beta+2\alpha-1}$. 
  Note that the fact that $0<\beta<1-\alpha$ ensures that $p>\frac{2\alpha}{1-\alpha+2\alpha-1}=2$.
  Moreover, observe that the fact that $p=\frac{2\alpha}{\beta+2\alpha-1}$ demonstrates that 
\begin{equation}
\frac{\beta}{2}=\frac{1-\alpha}{2}-\frac{(1-(1-\alpha))(p-2)}{2p}.
\end{equation}  
 If $\alpha>\frac{1}{2}$, then it holds that $1-\alpha<\frac{p}{2(p-1)}$. Moreover, if $\alpha>\frac{1}{2}$ the fact that $\beta>0$ implies that 
 $p<\frac{2\alpha}{2\alpha-1}$ and hence
 $1-\alpha>\frac{p-2}{2(p-1)}$. 
 If $\alpha<\frac{1}{2}$, then it holds that $1-\alpha>\frac{p-2}{2(p-1)}$. 
  Moreover, if $\alpha<\frac{1}{2}$	 the fact that $\beta>\frac{1-2\alpha}{1-\alpha}$ implies that 
 $p<\frac{2\alpha}{\frac{1-2\alpha}{1-\alpha}+2\alpha-1}=\frac{1-\alpha}{\frac{1}{2}-\alpha}$ and hence
 $1-\alpha<\frac{p}{2(p-1)}$. 
 Furthermore, it holds that $\frac{p-2}{2(p-1)}<\frac{1}{2}<\frac{p}{2(p-1)}$. 
 To summarize, it holds that 
\begin{equation}
p\in (2,\infty),\quad 1-\alpha \in \bigg(\frac{p-2}{2(p-1)},\frac{p}{2(p-1)}\bigg), \quad \text{and} \quad 
\frac{\beta}{2}=\frac{1-\alpha}{2}-\frac{(1-(1-\alpha))(p-2)}{2p}.
\end{equation} 
  Next note that \eqref{eq:fLipschitz4} ensures for all $d\in \N$ that
  \begin{equation}\label{eq:cor2_1}
  \begin{split}
 &\sup_{s,t\in [0,T)}
 \left\|
   (\funcF_d (0))(t,\xi_d+\sqrt{s}Z^{d,(0)})
   \right\|_{L^{\frac{2p}{p-2}}(\P;\R)}\\
   &\le 
   \sup_{s,t\in [0,T)}
 \left\|
   (\funcF_d(0))(t,\xi_d+\sqrt{s}Z^{d,(0)})-(\funcF_d(0))(t,0)
   \right\|_{L^{\frac{2p}{p-2}}(\P;\R)}  
   +\sup_{t\in [0,T)}
 \left|
   (\funcF_d(0))(t,0)
   \right|
   \\
   &=
   \sup_{s,t\in [0,T)}
 \left\|f_d(t,\xi_d+\sqrt{s}Z^{d,(0)},0,0)-f_d(t,0,0,0)
   \right\|_{L^{\frac{2p}{p-2}}(\P;\R)}  
   +\sup_{t\in [0,T)}
 \left|
  f_d(t,0,0,0)
   \right|
 %   \\
%   &\le 
%   \sup_{s\in [0,T)}
% \left\|
%\sum_{j=1}^d
%\mathfrak L_{d,j}|\xi_{d,j}+\sqrt{s}Z^{d,j,(0)}|
%   \right\|_{L^{\frac{2p}{p-2}}(\P;\R)}  
%   +\sup_{t\in [0,T)}
% \left|
%  f_d(t,0,0,0)
%   \right|
    \\
   &\le 
   \bigg(d^\lambda
\sum_{j=1}^d
 L_{d,j}\big(|\xi_{d,j}|+\sqrt{T}\left\|Z^{1,(0)}
   \right\|_{L^{\frac{2p}{p-2}}(\P;\R)}\big)  \bigg)
   +\sup_{t\in [0,T)}
 \left|
  f_d(t,0,0,0)
   \right|
   \\
   &\le d^\lambda
   \bigg(\big[\max_{j\in \{1,2,\ldots,d\}}|\xi_{d,j}|\big]
   +\sqrt{T}\left\|Z^{1,(0)}
   \right\|_{L^{\frac{2p}{p-2}}(\P;\R)}\bigg)
   \bigg(
\sum_{j=1}^d
 L_{d,j} \bigg)
   +\sup_{t\in [0,T)}
 \left|
  f_d(t,0,0,0)
   \right|.
   \end{split}
  \end{equation}
 Moreover, \eqref{eq:fLipschitz4} proves for all $d\in \N$ that
  \begin{equation}\label{eq:cor2_2}
  \begin{split}
   &\sup_{s\in[0,T]}
  \|g_d(\xi_d+\sqrt{s}Z^{d,(0)})\|_{L^{\frac{2p}{p-2}}(\P;\R)}
  \le \sup_{s\in[0,T]}
  \|g_d(\xi_d+\sqrt{s}Z^{d,(0)})-g_d(0)\|_{L^{\frac{2p}{p-2}}(\P;\R)}+|g_d(0)|
  \\
  &
   \le 
  \bigg(d^\lambda 
  \sum_{j=1}^d L_{d,j}\big( |\xi_{d,j}|+\sqrt{T}\|Z^{1,(0)}
  \|_{L^{\frac{2p}{p-2}}(\P;\R)} \big)\bigg)
  +|g_d(0)|
  \\
  &
  \le d^\lambda
   \bigg(\big[\max_{j\in \{1,2,\ldots,d\}}|\xi_{d,j}|\big]
   +\sqrt{T}\left\|Z^{1,(0)}
   \right\|_{L^{\frac{2p}{p-2}}(\P;\R)}\bigg)
   \bigg(
\sum_{j=1}^d L_{d,j} \bigg)
  +|g_d(0)|.
  \end{split}
  \end{equation}
  Furthermore, it holds for all $d\in \N$ that
  \begin{equation}\label{eq:cor2_3}
  \max_{j\in \{1,2,\ldots, d\}}d^\lambda L_{j,d} \le d^\lambda \sum_{j=1}^d L_{j,d}, \quad \text{and} \quad \sum_{j=1}^dL_{d,j} (d^\lambda L_{d,j}+Td^\lambda L_{d,j})\le d^\lambda (T+1) \bigg(\sum_{j=1}^d L_{d,j}\bigg)^2.
  \end{equation}
  Combining \eqref{eq:cor2_1}, \eqref{eq:cor2_2}, and \eqref{eq:cor2_3} with Corollary~\ref{cor:comp_and_error} (applied with $\alpha=1-\alpha$, $\beta=\frac{\beta}{2}$, $p=\frac{1-\alpha}{\frac{\beta}{2}-\alpha+\frac{1}{2}}$, $L_0=\sum_{j=1}^dL_{d,j}$, $L_j=L_{d,j}$, $K_j=d^\lambda L_{d,j}$, $\mathfrak{L}_j=d^\lambda L_{d,j}$ for $j\in \{1,\ldots,d\}$, $d\in \N$ in the notation of Corollary~\ref{cor:comp_and_error}) and the fact that
  \begin{equation}
  \sup_{d\in \N}\left[\frac{1}{d^\lambda}\left(\max_{i\in \{1,2,\ldots,d\}}
  |\xi_{d,i}|+\sup_{t\in [0,T]}|f_d(t,0,0,0)|+|g_d(0)|\right)+\sum_{i=1}^d L_{d,i}\right]<\infty
  \end{equation}
   proves \eqref{eq:cor2_main}. The proof of Theorem~\ref{cor:comp_and_error2} is thus completed.   
  \end{proof}

\begin{lemma}\label{lem:uniform_dist}
Let $T\in (0,\infty)$, $d\in \N$, $\alpha \in (0,1]$, let $(\Omega, \mathcal F, \P)$ be a probability space, let $Z=(Z^1,Z^2,\ldots, Z^d)\colon \Omega \to \R^d$ be a standard normal random variable, let $\unif\colon \Omega \to (0,1)$ satisfy for all $b\in (0,1)$ that $\P(\unif\le b)=b^\alpha$, assume that $Z$ and $\unif$ are independent, and let ${\bf U}=({\bf U}^{0},{\bf U}^{1}, \ldots, {\bf U}^{d})\colon \Omega \to \R^{d+1}$ satisfy that
$
 {\bf U}
 =
 T\alpha^{ - 1 }\unif^{1-\alpha} \big(
  1 ,
   [T\unif]^{-1/2}
  Z
  \big).
$
Then 
\begin{enumerate}[(i)]
\item \label{item:uniform_1} it holds for all $i\in \{1,2,\ldots,d\}$ that
\begin{equation}
\E\big[|{\bf U}^i|^2\big]
=\frac{T}{\alpha}\int_0^1s^{-\alpha}\, ds = 
\begin{cases}
 \frac{T}{\alpha(1-\alpha)} & \colon \alpha < 1\\
\infty & \colon \alpha = 1 
\end{cases}
\end{equation}
and
\item  \label{item:uniform_2} it holds that ${\bf U}\in L^2(\P;\R^{d+1})$ if and only if $\alpha \in (0,1)$.
\end{enumerate}
\end{lemma}  
\begin{proof}[Proof of Lemma~\ref{lem:uniform_dist}]
Note that for all $i\in \{1,2,\ldots,d\}$ it holds that
\begin{equation}
\begin{split}
\E\big[|{\bf U}^i|^2\big]&=\frac{T}{\alpha^2}\E\big[\unif^{2(1-\alpha)-1}|Z^{i}|^2 \big]
=\frac{T}{\alpha}\int_0^1\E\big[s^{2(1-\alpha)-1}|Z^{i}|^2\big] s^{\alpha-1}\, ds
=\frac{T}{\alpha}\int_0^1s^{2(1-\alpha)-1+\alpha-1}\, ds\\&
=\frac{T}{\alpha}\int_0^1s^{-\alpha}\, ds= 
\begin{cases}
 \frac{T}{\alpha(1-\alpha)} & \colon \alpha < 1\\
\infty & \colon \alpha = 1 
\end{cases}.
\end{split}
\end{equation}
This proves item~\eqref{item:uniform_1}. Moreover, observe that item~\eqref{item:uniform_1} establishes item~\eqref{item:uniform_2}. The proof of Lemma~\ref{lem:uniform_dist} is thus completed.
\end{proof}

\sloppy
\begin{corollary}\label{cor:comp_and_error3}
Let $T,\delta,\lambda \in (0,\infty)$, 
 let $f_d \in C( \R \times \R^d,\R)$, $d\in \N$,
 let $g_d \in C( \R^d,\R)$, $d\in \N$,
 let $\xi_d=(\xi_{d,1}, \xi_{d,2}, \ldots, \xi_{d,d}) \in \R^d$, $d\in \N$,
 let $L_{d,i}\in \R$, $d,i \in \N$,
let
$u_d = ( u_d(t,x) )_{ (t,x) \in [0,T] \times \R^d }\in C^{1,2}([0,T]\times\R^d,\R)$, $d\in \N$,
be at most polynomially growing functions,
assume for all $d\in \N$,
$t\in[0,T)$, $x=(x_1,x_2, \ldots,x_d)$, $\mathfrak x=(\mathfrak x_1,\mathfrak x_2, \ldots,\mathfrak x_d)$, $z=(z_1,z_2,\ldots,z_d)$, $\mathfrak{z}=(\mathfrak z_1, \mathfrak z_2, \ldots, \mathfrak z_d)\in\R^d$, $y,\mathfrak{y} \in \R$ that
\begin{equation}\label{eq:fLipschitz5}
\max\{|f_d(y,z)-f_d(\mathfrak y,\mathfrak{z})|,|g_d(x)-g_d(\mathfrak  x)|\}\le
\textstyle{
\sum_{j=1}^d}L_{d,j}
\big(d^\lambda |x_j-\mathfrak{x}_j|+ |y-\mathfrak y|+|z_j-\mathfrak{z}_j|
\big),
\end{equation}
\begin{equation}\label{eq:PDE4}
\big( \tfrac{ \partial }{ \partial t } u_d \big)( t, x )=  ( \Delta_x u_d )( t, x ) + f_d\big( u(t,x), ( \nabla_x u_d )(t, x) \big),  \qquad u_d(0,x) = g_d(x),
\end{equation}
and
$
d^{-\lambda}(|g_d(0)|+|f_d(0,0)|+\max_{i\in \{1,2,\ldots,d\}}
  |\xi_{d,i}|)+\sum_{i=1}^d L_{d,i}\le \lambda,
$
let
$
  ( 
    \Omega, \mathcal{F}, \P 
  )
$
be a probability space,
let
$
  \Theta = \cup_{ n \in \N } \Z^n
$,
let
$
  Z^{d, \theta } \colon \Omega \to \R^d 
$, $d\in \N$,
$ \theta \in \Theta $,
be i.i.d.\ standard normal random variables,
let $\unif^\theta\colon \Omega\to(0,1)$, $\theta\in \Theta$, be 
i.i.d.\ random variables,
assume 
for all $b\in (0,1)$
that
$\P(\unif^0\le b)=\sqrt{b}$,
assume that
$(Z^{d,\theta})_{(d, \theta) \in \N \times \Theta}$ and
$(\unif^\theta)_{ \theta \in \Theta}$ are independent,
let
$ 
  {\bf U}_{ n,M}^{d,\theta }
=
  (
  {\bf U}_{ n,M}^{d,\theta, 0},{\bf U}_{ n,M}^{d,\theta, 1},\ldots,{\bf U}_{ n,M}^{d,\theta, d}
  )
  \colon(0,T]\times\R^d\times\Omega\to\R^{1+d}
$,
$n,M,d\in\Z$, $\theta\in\Theta$,
satisfy
for all 
$
  n,M,d \in \N
$,
$ \theta \in \Theta $,
$ t\in (0,T]$,
$x \in \R^d$
that $
{\bf U}_{-1,M}^{d,\theta}(t,x)={\bf U}_{0,M}^{d,\theta}(t,x)=0$ and
\begin{equation}  \begin{split}\label{eq:def:Ucor3}
 {\bf U}_{n,M}^{d,\theta}(t,x)
  &=
  \left(
    g_d(x)
    , 0
  \right)
  +
  \textstyle \sum\limits_{i=1}^{M^n} \displaystyle \tfrac{1}{M^n} \big(g_d(x+[2t]^{1/2}Z^{(\theta,0,-i)})-g_d(x)\big)
  \big(
  1 , [2t]^{-1/2}
  Z^{d,(\theta, 0, -i)}
  \big)
  \\
  &\quad +\textstyle\sum\limits_{l=0}^{n-1}\sum\limits_{i=1}^{M^{n-l}}\displaystyle \tfrac{2t [\unif^{(\theta, l,i)}]^{1/2}}{M^{n-l}}  \big[ 
 f_d\big({\bf U}_{l,M}^{d,(\theta,l,i)}(t(1-\unif^{(\theta, l,i)}),x+[2t\unif^{(\theta, l,i)}]^{1/2}Z^{d,(\theta,l,i)})\big)
\\
 & \quad  
 -\1_{\N}(l)f_d\big( {\bf U}_{l-1,M}^{d,(\theta,-l,i)}(t(1-\unif^{(\theta, l,i)}),x+[2t\unif^{(\theta, l,i)}]^{1/2}Z^{d,(\theta,l,i)})\big)
  \big]\big(
  1 ,
   [2t\unif^{(\theta, l,i)}]^{-1/2}
  Z^{d,(\theta,l,i)}
  \big),
\end{split}     \end{equation}
and let $\RN_{d,n,M}\in \Z$, $d,n,M\in \Z$, satisfy
for all $d,n,M \in \N$ that 
$\RN_{d,0,M}=0$
and 
\begin{align}
\label{c18}
  \RN_{d,n,M}
  &\leq d M^n+\textstyle\sum\limits_{l=0}^{n-1}\displaystyle\left[M^{(n-l)}( d+1 + \RN_{d, l, M}+ \1_{ \N }( l )  \RN_{d, l-1, M })\right].
\end{align}
Then there exist $c\in \R$ and $N=(N_{d,\eps})_{(d, \eps) \in \N \times (0,1]}\colon \N \times (0,1] \to \N$ such that
for all $d\in \N$, $\eps \in (0,1]$ it holds that 
$ \sum_{n=1}^{N_{d,\eps}}\RN_{d,n,\lfloor n^{1/4} \rfloor} \le c d^c \varepsilon^{-(2+\delta)}$ and
 \begin{equation}\label{eq:cor3_main}
\sup_{ n \in \N \cap [N_{d,\eps},\infty) } \Big[
 \E\big[|{\bf U}_{{n},\lfloor n^{1/4} \rfloor}^{d,0,0}(T,\xi_d)-u_d(T,\xi_d)|^2\big]+
\max_{i\in\{1,2,\ldots,d\}}
    \E\big[ |{\bf U}_{{n},\lfloor n^{1/4} \rfloor}^{d,0,i}(T,\xi_d)-( \tfrac{ \partial }{ \partial x_i } u_d )(T,\xi_d)|^2\big]
 \Big]^{\nicefrac 12}
 \le \eps.
 \end{equation}
\end{corollary}
\fussy   
\begin{proof}[Proof of Corollary~\ref{cor:comp_and_error3}]
Corollary~\ref{cor:comp_and_error3} is a direct consequence of Theorem~\ref{cor:comp_and_error2} (applied with $\alpha=\frac{1}{2}$, $\beta=\frac{1}{4}$, $T=2T$, $u_d(t,x)=u_d(T-\frac{t}{2},x)$, $f_d(t,x,y,z)=f_d(y,z)/2$ 
for $t\in [0,2T]$, $x,z\in \R^d$, $y\in \R$, $d\in \N$ in the notation of Theorem~\ref{cor:comp_and_error2}).
\end{proof}

\subsubsection*{Acknowledgement}
This project has been partially supported 
by the Deutsche Forschungsgesellschaft (DFG) via RTG 2131 {\it High-dimensional Phenomena in Probability -- Fluctuations and Discontinuity}
and via research grant HU 1889/6-1.
 
\bibliographystyle{acm}%{amsalpha}
%%\bibliography{../../Bib/bibfile}
%\bibliography{bibfile}

\begin{thebibliography}{10}

\bibitem{BeckBeckerCheridito2019}
{\sc Beck, C., Becker, S., Cheridito, P., Jentzen, A., and Neufeld, A.}
\newblock {Deep splitting method for parabolic PDEs}.
\newblock {\em arXiv:1907.03452\/} (2019).

\bibitem{Becketal2018}
{\sc Beck, C., Becker, S., Grohs, P., Jaafari, N., and Jentzen, A.}
\newblock Solving stochastic differential equations and {K}olmogorov equations
  by means of deep learning.
\newblock {\em arXiv:1806.00421\/} (2018).

\bibitem{BeckEJentzen2017}
{\sc Beck, C., E, W., and Jentzen, A.}
\newblock Machine learning approximation algorithms for high-dimensional fully
  nonlinear partial differential equations and second-order backward stochastic
  differential equations.
\newblock {\em Journal of Nonlinear Science 29}, 4 (2019), 1563--1619.

\bibitem{beck2019overcoming}
{\sc Beck, C., Hornung, F., Hutzenthaler, M., Jentzen, A., and Kruse, T.}
\newblock Overcoming the curse of dimensionality in the numerical approximation
  of {A}llen-{C}ahn partial differential equations via truncated full-history
  recursive multilevel {P}icard approximations.
\newblock {\em arXiv:1907.06729\/} (2019).

\bibitem{BeckerCheriditoJentzen2018arXiv}
{\sc {Becker}, S., {Cheridito}, P., and {Jentzen}, A.}
\newblock {Deep optimal stopping}.
\newblock {\em Journal of Machine Learning Research 20}, 74 (2019), 1--25.

\bibitem{BeckerCheriditoJentzen2019}
{\sc Becker, S., Cheridito, P., Jentzen, A., and Welti, T.}
\newblock Solving high-dimensional optimal stopping problems using deep
  learning.
\newblock {\em arXiv:1908.01602\/} (2019), 42 pages.

\bibitem{Berg2018AUD}
{\sc Berg, J., and Nystr{\"o}m, K.}
\newblock A unified deep artificial neural network approach to partial
  differential equations in complex geometries.
\newblock {\em Neurocomputing 317\/} (2018), 28--41.

\bibitem{BernerGrohsJentzen2018}
{\sc Berner, J., Grohs, P., and Jentzen, A.}
\newblock {Analysis of the generalization error: Empirical risk minimization
  over deep artificial neural networks overcomes the curse of dimensionality in
  the numerical approximation of Black-Scholes partial differential equations}.
\newblock {\em arXiv:1809.03062, Revision requested from SIAM J. Math. Data
  Sci.\/} (2018).

\bibitem{BouchardTanWarinZou2017}
{\sc Bouchard, B., Tan, X., Warin, X., and Zou, Y.}
\newblock Numerical approximation of {BSDE}s using local polynomial drivers and
  branching processes.
\newblock {\em Monte Carlo Methods and Applications 23}, 4 (2017), 241--263.

\bibitem{BouchardTouzi2004}
{\sc Bouchard, B., and Touzi, N.}
\newblock Discrete-time approximation and {Monte-Carlo} simulation of backward
  stochastic differential equations.
\newblock {\em Stochastic Processes and their applications 111}, 2 (2004),
  175--206.

\bibitem{BriandLabart2014}
{\sc Briand, P., and Labart, C.}
\newblock Simulation of {BSDE}s by {W}iener chaos expansion.
\newblock {\em The Annals of Applied Probability 24}, 3 (2014), 1129--1171.

\bibitem{ChanMikaelWarin2019}
{\sc Chan-Wai-Nam, Q., Mikael, J., and Warin, X.}
\newblock Machine learning for semi linear {PDE}s.
\newblock {\em J. Sci. Comput. 79}, 3 (2019), 1667--1712.

\bibitem{chen2019deep}
{\sc Chen, Y., and Wan, J.~W.}
\newblock Deep neural network framework based on backward stochastic
  differential equations for pricing and hedging american options in high
  dimensions.
\newblock {\em arXiv:1909.11532\/} (2019).

\bibitem{Dockhorn2019}
{\sc Dockhorn, T.}
\newblock A discussion on solving partial differential equations using neural
  networks.
\newblock {\em arXiv:1904.07200\/} (2019).

\bibitem{EHanJentzen2017CMStat}
{\sc E, W., Han, J., and Jentzen, A.}
\newblock Deep learning-based numerical methods for high-dimensional parabolic
  partial differential equations and backward stochastic differential
  equations.
\newblock {\em Commun. Math. Stat. 5}, 4 (2017), 349--380.

\bibitem{EHutzenthalerJentzenKruse2016}
{\sc E, W., Hutzenthaler, M., Jentzen, A., and Kruse, T.}
\newblock Multilevel {P}icard iterations for solving smooth semilinear
  parabolic heat equations.
\newblock {\em arXiv:1607.03295\/} (2016).

\bibitem{EHutzenthalerJentzenKruse2017}
{\sc E, W., Hutzenthaler, M., Jentzen, A., and Kruse, T.}
\newblock On multilevel {P}icard numerical approximations for high-dimensional
  nonlinear parabolic partial differential equations and high-dimensional
  nonlinear backward stochastic differential equations.
\newblock {\em Journal of Scientific Computing 79}, 3 (2019), 1534--1571.

\bibitem{EYu2018}
{\sc E, W., and Yu, B.}
\newblock The {Deep Ritz method: A} deep learning-based numerical algorithm for
  solving variational problems.
\newblock {\em Communications in Mathematics and Statistics 6}, 1 (2018),
  1--12.

\bibitem{ElKarouiPengQuenez1997}
{\sc El~Karoui, N., Peng, S., and Quenez, M.~C.}
\newblock Backward stochastic differential equations in finance.
\newblock {\em Mathematical finance 7}, 1 (1997), 1--71.

\bibitem{ElbraechterSchwab2018}
{\sc Elbr{\"a}chter, D., Grohs, P., Jentzen, A., and Schwab, C.}
\newblock {DNN Expression Rate Analysis of High-dimensional PDEs: Application
  to Option Pricing}.
\newblock {\em arXiv:1809.07669\/} (2018).

\bibitem{Farahmand2017DeepRL}
{\sc Farahmand, A.-m., Nabi, S., and Nikovski, D.}
\newblock Deep reinforcement learning for partial differential equation
  control.
\newblock {\em 2017 American Control Conference (ACC)\/} (2017), 3120--3127.

\bibitem{FujiiTakahashiTakahashi2017}
{\sc Fujii, M., Takahashi, A., and Takahashi, M.}
\newblock Asymptotic {E}xpansion as {P}rior {K}nowledge in {D}eep {L}earning
  {M}ethod for high dimensional {BSDEs}.
\newblock {\em arXiv:1710.07030\/} (2017).

\bibitem{geiss2016simulation}
{\sc Geiss, C., and Labart, C.}
\newblock Simulation of {BSDE}s with jumps by {W}iener chaos expansion.
\newblock {\em Stochastic processes and their applications 126}, 7 (2016),
  2123--2162.

\bibitem{giles20019generalised}
{\sc Giles, M.~B., Jentzen, A., and Welti, T.}
\newblock Generalised multilevel {P}icard approximations.
\newblock {\em arXiv:1911.03188\/} (2019).

\bibitem{GobetLemorWarin2005}
{\sc Gobet, E., Lemor, J.-P., and Warin, X.}
\newblock A regression-based {M}onte {C}arlo method to solve backward
  stochastic differential equations.
\newblock {\em The Annals of Applied Probability 15}, 3 (2005), 2172--2202.

\bibitem{gobet2016approximation}
{\sc Gobet, E., Turkedjiev, P., et~al.}
\newblock Approximation of backward stochastic differential equations using
  malliavin weights and least-squares regression.
\newblock {\em Bernoulli 22}, 1 (2016), 530--562.

\bibitem{GoudenegeMolent2019}
{\sc Gouden{\`e}ge, L., Molent, A., and Zanette, A.}
\newblock {Machine Learning for Pricing American Options in High Dimension}.
\newblock {\em arXiv:1903.11275\/} (2019), 11 pages.

\bibitem{GrohsWurstemberger2018}
{\sc {Grohs}, P., {Hornung}, F., {Jentzen}, A., and {von Wurstemberger}, P.}
\newblock {A proof that artificial neural networks overcome the curse of
  dimensionality in the numerical approximation of Black-Scholes partial
  differential equations}.
\newblock {\em to appear in Mem. Amer. Math. Soc.\/} (2019).

\bibitem{GrohsHornungJentzen2019}
{\sc Grohs, P., Hornung, F., Jentzen, A., and Zimmermann, P.}
\newblock Space-time error estimates for deep neural network approximations for
  differential equations.
\newblock {\em arXiv:1908.03833\/} (2019).

\bibitem{grohs2019deep}
{\sc Grohs, P., Jentzen, A., and Salimova, D.}
\newblock Deep neural network approximations for monte carlo algorithms.
\newblock {\em arXiv:1908.10828\/} (2019).

\bibitem{HanJentzenE2017}
{\sc Han, J., Jentzen, A., and E, W.}
\newblock {S}olving high-dimensional partial differential equations using deep
  learning.
\newblock {\em Proceedings of the National Academy of Sciences 115}, 34 (2018),
  8505--8510.

\bibitem{HanLong2018}
{\sc Han, J., and Long, J.}
\newblock {Convergence of the Deep BSDE Method for Coupled FBSDEs}.
\newblock {\em arXiv:1811.01165\/} (2018).

\bibitem{HenryLabordere2012}
{\sc Henry-Labord{\`e}re, P.}
\newblock Counterparty risk valuation: a marked branching diffusion approach.
\newblock {\em arXiv:1203.2369\/} (2012).

\bibitem{HenryLabordere2017}
{\sc Henry-Labord{\`e}re, P.}
\newblock Deep {Primal-Dual Algorithm for BSDEs}: {A}pplications of {M}achine
  {L}earning to {CVA and IM}.
\newblock {\em Available at SSRN: http://dx.doi.org/10.2139/ssrn.3071506\/}
  (2017).

\bibitem{HenryLabordereOudjaneTanTouziWarin2016}
{\sc Henry-Labord{\`e}re, P., Oudjane, N., Tan, X., Touzi, N., and Warin, X.}
\newblock Branching diffusion representation of semilinear {PDE}s and {M}onte
  {C}arlo approximation.
\newblock {\em Annales de l'Institut Henri Poincar{\'e}, Probabilit{\'e}s et
  Statistiques 55}, 1 (2019), 184--210.

\bibitem{HenryLabordereTanTouzi2014}
{\sc Henry-Labord{\`e}re, P., Tan, X., and Touzi, N.}
\newblock A numerical algorithm for a class of {BSDE}s via the branching
  process.
\newblock {\em Stochastic Process. Appl. 124}, 2 (2014), 1112--1140.

\bibitem{HurePhamWarin2019}
{\sc Hur{\'e}, C., Pham, H., and Warin, X.}
\newblock {Some machine learning schemes for high-dimensional nonlinear PDEs}.
\newblock {\em arXiv:1902.01599\/} (2019).

\bibitem{hutzenthaler2019proof}
{\sc Hutzenthaler, M., Jentzen, A., Kruse, T., and Nguyen, T.~A.}
\newblock A proof that rectified deep neural networks overcome the curse of
  dimensionality in the numerical approximation of semilinear heat equations.
\newblock {\em arXiv:1901.10854, Revision requested from SN Partial
  Differential Equations and Applications\/} (2019).

\bibitem{HJKNW2018}
{\sc Hutzenthaler, M., Jentzen, A., Kruse, T., Nguyen, T.~A., and von
  Wurstemberger, P.}
\newblock Overcoming the curse of dimensionality in the numerical approximation
  of semilinear parabolic partial differential equations.
\newblock {\em arXiv:1807.01212\/} (2018).

\bibitem{hutzenthaler2019overcoming}
{\sc Hutzenthaler, M., Jentzen, A., and von Wurstemberger, P.}
\newblock Overcoming the curse of dimensionality in the approximative pricing
  of financial derivatives with default risks.
\newblock {\em arXiv:1903.05985\/} (2019).

\bibitem{HutzenthalerKruse2017}
{\sc Hutzenthaler, M., and Kruse, T.}
\newblock Multi-level {P}icard approximations of high-dimensional semilinear
  parabolic differential equations with gradient-dependent nonlinearities.
\newblock {\em to appear in SIAM J. Numer. Anal.\/} (2019).

\bibitem{JacquierOumgari2019}
{\sc Jacquier, A., and Oumgari, M.}
\newblock {Deep PPDEs for rough local stochastic volatility}.
\newblock {\em arXiv:1906.02551\/} (2019).

\bibitem{JentzenSalimovaWelti2018}
{\sc Jentzen, A., Salimova, D., and Welti, T.}
\newblock A proof that deep artificial neural networks overcome the curse of
  dimensionality in the numerical approximation of {K}olmogorov partial
  differential equations with constant diffusion and nonlinear drift
  coefficients.
\newblock {\em arXiv:1809.07321\/} (2018).

\bibitem{LiLuo2003}
{\sc Jianyu, L., Siwei, L., Yingjian, Q., and Yaping, H.}
\newblock Numerical solution of elliptic partial differential equation using
  radial basis function neural networks.
\newblock {\em Neural Networks 16}, 5 (2003), 729 -- 734.

\bibitem{KutyniokPeterseb2019}
{\sc Kutyniok, G., Petersen, P., Raslan, M., and Schneider, R.}
\newblock A theoretical analysis of deep neural networks and parametric {PDE}s.
\newblock {\em arXiv:1904.00377\/} (2019).

\bibitem{Lagaris1998ArtificialNN}
{\sc Lagaris, I.~E., Likas, A., and Fotiadis, D.~I.}
\newblock Artificial neural networks for solving ordinary and partial
  differential equations.
\newblock {\em IEEE transactions on neural networks 9 (5)\/} (1998), 987--1000.

\bibitem{Lemor2006}
{\sc Lemor, J.-P., Gobet, E., and Warin, X.}
\newblock Rate of convergence of an empirical regression method for solving
  generalized backward stochastic differential equations.
\newblock {\em Bernoulli 12}, 5 (2006), 889--916.

\bibitem{LongLuMaDong2018}
{\sc Long, Z., Lu, Y., Ma, X., and Dong, B.}
\newblock {PDE-Net: Learning PDEs from Data}.
\newblock In {\em Proceedings of the 35th International Conference on Machine
  Learning\/} (2018), pp.~3208--3216.

\bibitem{LyeMishraRay2019}
{\sc Lye, K.~O., Mishra, S., and Ray, D.}
\newblock Deep learning observables in computational fluid dynamics.
\newblock {\em arXiv:1903.03040\/} (2019).

\bibitem{Magill2018NeuralNT}
{\sc Magill, M., Qureshi, F., and de~Haan, H.~W.}
\newblock Neural networks trained to solve differential equations learn general
  representations.
\newblock In {\em Advances in Neural Information Processing Systems\/} (2018),
  pp.~4071--4081.

\bibitem{MeadeFernandez1994}
{\sc Meade, Jr., A.~J., and Fern\'{a}ndez, A.~A.}
\newblock The numerical solution of linear ordinary differential equations by
  feedforward neural networks.
\newblock {\em Math. Comput. Modelling 19}, 12 (1994), 1--25.

\bibitem{PhamWarin2019}
{\sc Pham, H., and Warin, X.}
\newblock {Neural networks-based backward scheme for fully nonlinear PDEs}.
\newblock {\em arXiv:1908.00412\/} (2019).

\bibitem{Qi2010}
{\sc Qi, F.}
\newblock Bounds for the ratio of two gamma functions.
\newblock {\em Journal of Inequalities and Applications 2010}, 1 (2010),
  493058.

\bibitem{Raissi2018}
{\sc Raissi, M.}
\newblock {Forward-Backward} {S}tochastic {N}eural {N}etworks: {D}eep
  {L}earning of {H}igh-dimensional {P}artial {D}ifferential {E}quations.
\newblock {\em arXiv:1804.07010\/} (2018).

\bibitem{ReisingerZhang2019}
{\sc Reisinger, C., and Zhang, Y.}
\newblock Rectified deep neural networks overcome the curse of dimensionality
  for nonsmooth value functions in zero-sum games of nonlinear stiff systems.
\newblock {\em arXiv:1903.06652\/} (2019).

\bibitem{SirignanoSpiliopoulos2017}
{\sc Sirignano, J., and Spiliopoulos, K.}
\newblock {DGM: A} deep learning algorithm for solving partial differential
  equations.
\newblock {\em arXiv:1708.07469\/} (2017).

\bibitem{uchiyama1993solving}
{\sc Uchiyama, T., and Sonehara, N.}
\newblock Solving inverse problems in nonlinear {PDE}s by recurrent neural
  networks.
\newblock In {\em IEEE International Conference on Neural Networks\/} (1993),
  IEEE, pp.~99--102.

\bibitem{wendel1948note}
{\sc Wendel, J.}
\newblock Note on the gamma function.
\newblock {\em The American Mathematical Monthly 55}, 9 (1948), 563--564.

\end{thebibliography}
\def\cprime{$'$}

\end{document}